\documentclass{amsart}

\usepackage[T1]{fontenc}	
\usepackage{totpages}
\usepackage{graphicx}
\usepackage{amsmath, amsthm, amssymb, thmtools}
\usepackage{upgreek}
\usepackage{mathrsfs}
\usepackage{pdfsync}
\usepackage[a4paper,top=3.8cm,bottom=3.7cm,left=3.5cm,right=3.5cm, bindingoffset=0mm]{geometry}
\usepackage{cite}
\usepackage[dvipsnames]{xcolor}
\usepackage[inline,shortlabels]{enumitem}
\usepackage[hyphens]{url}									
\usepackage{verbatim}								
\usepackage{dsfont}									
\usepackage{bbm}
\usepackage{bm}
\usepackage[all]{xy}

\usepackage[ruled,vlined]{algorithm2e}

\usepackage{mathtools}
\usepackage{xparse}
\usepackage{microtype}

\usepackage{subcaption}
\usepackage{graphicx}
\usepackage[font=small]{caption}
\usepackage{nicefrac}

\usepackage{enumitem}

\usepackage[pdftex,bookmarks,
						hidelinks,					
						colorlinks,				
						breaklinks,
						citecolor=OliveGreen,
						linkcolor=Maroon,
						]{hyperref} 

\usepackage{cleveref}

\allowdisplaybreaks

\usepackage{xparse}

\ExplSyntaxOn
\NewDocumentCommand{\makeabbrev}{mmm}
 {
  \yoruk_makeabbrev:nnn { #1 } { #2 } { #3 }
 }

\cs_new_protected:Npn \yoruk_makeabbrev:nnn #1 #2 #3
 {
  \clist_map_inline:nn { #3 }
   {
    \cs_new_protected:cpn { #2 } { #1 { ##1 } }
   }
 }
 \ExplSyntaxOff

\makeabbrev{\textbf}{tbf#1}{a,b,c,d,e,f,g,h,i,j,k,l,m,n,o,p,q,r,s,t,u,v,w,x,y,z,A,B,C,D,E,F,G,H,I,J,K,L,M,N,O,P,Q,R,S,T,U,V,W,X,Y,Z}

\makeabbrev{\textbf}{bf#1}{a,b,c,d,e,f,g,h,i,j,k,l,m,n,o,p,q,r,s,t,u,v,w,x,y,z,A,B,C,D,E,F,G,H,I,J,K,L,M,N,O,P,Q,R,S,T,U,V,W,X,Y,Z}

\makeabbrev{\textsf}{tsf#1}{a,b,c,d,e,f,g,h,i,j,k,l,m,n,o,p,q,r,s,t,u,v,w,x,y,z,A,B,C,D,E,F,G,H,I,J,K,L,M,N,O,P,Q,R,S,T,U,V,W,X,Y,Z}

\makeabbrev{\mathsf}{mss#1}{a,b,c,d,e,f,g,h,i,j,k,l,m,n,o,p,q,r,s,t,u,v,w,x,y,z,A,B,C,D,E,F,G,H,I,J,K,L,M,N,O,P,Q,R,S,T,U,V,W,X,Y,Z}

\makeabbrev{\mathfrak}{mf#1}{a,b,c,d,e,f,g,h,i,j,k,l,m,n,o,p,q,r,s,t,u,v,w,x,y,z,A,B,C,D,E,F,G,H,I,J,K,L,M,N,O,P,Q,R,S,T,U,V,W,X,Y,Z,
sl,gl
}

\makeabbrev{\mathrm}{mrm#1}{a,b,c,d,e,f,g,h,i,j,k,l,m,n,o,p,q,r,s,t,u,v,w,x,y,z,A,B,C,D,E,F,G,H,I,J,K,L,M,N,O,P,Q,R,S,T,U,V,W,X,Y,Z}

\makeabbrev{\mathbf}{mbf#1}{a,b,c,d,e,f,g,h,i,j,k,l,m,n,o,p,q,r,s,t,u,v,w,x,y,z,A,B,C,D,E,F,G,H,I,J,K,L,M,N,O,P,Q,R,S,T,U,V,W,X,Y,Z}

\makeabbrev{\mathcal}{mc#1}{A,B,C,D,E,F,G,H,I,J,K,L,M,N,O,P,Q,R,S,T,U,V,W,X,Y,Z}

\makeabbrev{\mathbb}{mbb#1}{A,B,C,D,E,F,G,H,I,J,K,L,M,N,O,P,Q,R,S,T,U,V,W,X,Y,Z}

\makeabbrev{\mathscr}{ms#1}{A,B,C,D,E,F,G,H,I,J,K,L,M,N,O,P,Q,R,S,T,U,V,W,X,Y,Z}

\makeabbrev{\mathrm}{#1}{
id,ran,rk,diag,stab,ann,conv,pr,ev,tr,End,Hom,sgn,im,op,can,fin,ext,red,tot,lex,Aut,Inn,unit,
rot,usc,lsc,Lip,lip,bSymLip,osc,AC,loc,coz,z,
supp,Opt,Adm,Cpl,Geo,GeoOpt,GeoAdm,GeoCpl,reg,res,
bd,co,Ric,Exp,dExp,dist,seg,Seg,cut,fcut,Cut,SDiff,Iso,Isom,diam,cl,Homeo,Diff,Der,vol,dvol,inj,relint, Graph, sub,
var,law,Var,Poi,Gam,pa,so,iso,fs,inv,pqi,mix,erg,form,
TestF,
ob,cod,inp,
}

\makeabbrev{\mathsf}{#1}{CD,BE,MCP,Ent,wMTW,MTW,Ch,RCD,EVI,Rad,dRad,SL,cSL,dSL,ScL,Irr,SC,wFe,VA,MetMeas,UMeas,CSMet,Met,USp,Meas,Mbl,alg,Alg}

\makeabbrev{\mathsc}{#1}{mmaf,cg}

\newcommand{\mathsc}[1]{\text{\textsc{#1}}}

\usepackage{scrextend}						

\let\temp\phi
\let\phi\varphi
\let\varphi\temp

\let\temp\epsilon
\let\epsilon\varepsilon
\let\varepsilon\temp

\DeclareMathOperator*{\argmax}{arg\,max}
\DeclareMathOperator*{\argmin}{arg\,min}

\numberwithin{equation}{section}
\theoremstyle{plain}
\newtheorem{theorem}{Theorem}[section]
\newtheorem*{theorem*}{Theorem}
\crefname{theorem}{theorem}{Theorems}
\Crefname{Theorem}{Theorem}{Theorems}

\crefname{proposition}{proposition}{propositions}
\Crefname{Proposition}{Proposition}{Propositions}
\crefname{lemma}{lemma}{lemmas}
\Crefname{Lemma}{Lemma}{Lemmas}

\newtheorem{corollary}[theorem]{Corollary}
\crefname{corollary}{corollary}{corollaries}
\Crefname{Corollary}{Corollary}{Corollaries}

\theoremstyle{definition}
\newtheorem{definition}[theorem]{Definition}
\crefname{definition}{definition}{definitions}
\Crefname{Definition}{Definition}{Definitions}
\newtheorem*{defs*}{Definition}

\theoremstyle{remark}
\newtheorem{remark}[theorem]{Remark}
\crefname{remark}{remark}{remarks}
\Crefname{Remark}{Remark}{Remarks}
\newtheorem{example}[theorem]{Example}
\crefname{example}{example}{examples}
\Crefname{Example}{Example}{Examples}
\crefname{conjecture}{conjecture}{conjectures}
\Crefname{Conjecture}{Conjecture}{Conjectures}
\crefname{problem}{problem}{problems}
\Crefname{Problem}{Problem}{Problems}
\Crefname{assumption}{\textbf{A}\hspace{-3pt}}{\textbf{A}\hspace{-3pt}}
\crefname{assumption}{\textbf{H}}{\textbf{H}}
\newtheorem*{ass*}{Assumption}

\renewcommand{\paragraph}[1]{\medskip\emph{#1}.\quad}

\usepackage{comment}

\usepackage{todonotes}

\newcommand{\bes}{\begin{equation*}}
\newcommand{\ees}{\end{equation*}}
\newcommand{\beas}{\begin{eqnarray}}
\newcommand{\eeas}{\end{eqnarray}}
\newcommand{\bea}{\begin{eqnarray}}
\newcommand{\eea}{\end{eqnarray}}
\newcommand{\be}{\begin{equation}}
\newcommand{\ee}{\end{equation}}
\newcommand{\bei}{\begin{itemize}}
\newcommand{\eei}{\end{itemize}}
\newcommand{\bec}{\begin{cases}}
\newcommand{\eec}{\end{cases}}
\newcommand{\ben}{\begin{enumerate}}
\newcommand{\een}{\end{enumerate}}

\newcommand{\De}{\mathrm{d}}
\def\ie{\textit{i.e.}}

\begin{document}

\title{Sparsity for isotropic spherical random fields}

\dedicatory{In Honour of N.Leonenko's 75th Birthday}

\author{Giacomo Greco}
\address{Università degli studi di Roma Tor Vergata}
\curraddr{RoMaDS - Department of Mathematics, 00133RM Rome, Italy.}
\email{greco@mat.uniroma2.it}
\thanks{Both authors are associated to INdAM (Istituto Nazionale di Alta
Matematica ``Francesco Severi'') and the group GNAMPA; they are supported by the PRIN project GRAFIA (CUP:
E53D23005530006) and MUR Departement of Excellence Programme 2023-2027 MatMod@Tov (CUP: E83C23000330006).}

\author{Domenico Marinucci}
\address{Università degli studi di Roma Tor Vergata}
\curraddr{RoMaDS - Department of Mathematics, 00133RM Rome, Italy.}
\email{marinucc@mat.uniroma2.it}

\begin{abstract}
We introduce a simple representation for isotropic spherical random fields and we discuss how it allows to discuss different notions of sparsity under isotropy. We also show how a suitable construction of sparse fields can mimic well the angular power spectrum and the polyspectra of some popular non-Gaussian fields, at the same time allowing for computationally efficient simulation algorithms. Using related ideas we also show how it is possible to obtain sparse approximations of spherical random fields which preserve isotropy, thus addressing an issue which has been raised in the Cosmological literature.


\noindent\textbf{MSc Classification: }60G60, 33C55, 43A90

\end{abstract}

\maketitle

\section{Background and objectives}

Over the last two decades, isotropic spherical random fields have attracted
a lot of interest and have been considered for an enormous number of applications in very diverse
areas, including Cosmology, Climate Science, Geophysics, Quantum Mechanics,
Machine Learning and Neural Networks, among others. The overwhelming
majority of this literature has focused on Gaussian circumstances; this is
also due to the mathematical difficulty to introduce principled non-Gaussian
models which do not violate isotropy.

In particular, the main tool for the investigation of finite variance, isotropic spherical
random fields is the spectral representation theorem, which allows expressing $T(x),$ $x\in \mbbS^{2}$ as%
\begin{equation}  \label{eq:harmonic:decomposition}
T(x)=\sum_{\ell =0}^{\infty }\sum_{m=-\ell }^{\ell }a_{\ell m}Y_{\ell m}(x)%
\text{ ,}
\end{equation}
where the equality holds both in the $\mrmL^{2}(\Omega )$ and in the $\mrmL%
^{2}(\Omega \times \mbbS^{2})$ sense; the spherical harmonics $Y_{\ell m}(.)$ are an orthonormal
system of eigenfunctions of the spherical Laplacian operator, and
the coefficients $\left\{ a_{\ell m}\right\} $ are zero mean and
uncorrelated, with variance $\mbbE[|a_{\ell m}|^{2}]=C_{\ell }$, the latter
defining the angular power spectrum of the random field. The
randomness of the whole field is hence entirely encoded in the vector of harmonic
coefficients $(\mathbf{a}_{\ell})_{\ell\in\mbbN}=(a_{\ell m})_{\ell\in\mbbN%
,m=-\ell,\dots,+\ell}$, (see \cite{Yadrenko1983}, \cite{Leonenko1999}, \cite{marinucci2011random}) . 

A question that can appear natural is whether it is possible to achieve an alternative representation that, while preserving isotropy, may turn out to be "more sparse", i.e. to require a smaller number of random coefficients under some circumstances. The question arises in particular if one keeps in mind that many of the most popular modern statistical or machine learning procedures (such as the Lasso, see \cite{Vandegeer2011}) are based on various forms of penalization/sparsity-enforcing
algorithms. Naive applications trying to enforce sparsity on the coefficients $a_{\ell,m}$ in spectral representation \cref{eq:harmonic:decomposition} have been considered in the literature, but they necessarily
lead to anisotropic features, as discussed in \cite{CammarotaAcha2015}, \cite{Feeney}, \cite{Sloan2020} and others.

Indeed, it can be shown that a spherical random field is isotropic if and
only if the triangular array of coefficients $\left\{ a_{\ell m}\right\} $
satisfies the following invariance conditions; we must have that%
\begin{equation*}
\mathbf{a}_{\ell }\overset{Law}{=}D_{\ell }(g)\mathbf{a}_{\ell }\text{ ,}
\end{equation*}%
where we write $\mathbf{a}_{\ell }=(a_{\ell ,-\ell },\,\dots,\,a_{\ell \ell
})^{T}$ for the $(2\ell +1)\times 1$ column vector of random spherical
harmonic coefficients. Here we have introduced the set of ``Wigner'' $%
(2\ell +1)\times (2\ell +1)$ matrix-valued functions $D_{\ell }(g),$ $g\in
SO(3)$, which provide a family of irreducible representations for the special group of rotations $SO(3)$ (see \cite{Varshalovich1988}, Chapter 4 and \cite{marinucci2011random}, Chapter 3.3). In particular, this implies that the spectral
representation cannot be sparse for isotropic fields, i.e., it cannot be
imposed that some of these coefficients are zero-valued, because this would
clearly violate the invariance condition: the matrix functions $D_{\ell }(.)$
yield an irreducible representations of the group $SO(3)$ and hence no
non-trivial subspace of $\mathbb{R}^{2\ell +1}$ can be invariant to their
action; hence if some components of the vector of harmonic coefficients $%
\mathbf{a}_{\ell}$ are null, they must necessarily be non-zero for some
other rotation $g\in SO(3)$.

\medskip

In order to preserve isotropy and enforce sparsity the only existing approach that we are
aware of is the one proposed in \cite{Sloan2020}, where sparsity is enforced by introducing a regularization term on the sample angular power spectrum at different multipoles. As discussed in more detail in \cref{Comparison}, the approach is rigorous and computationally feasible; however, from a certain point of view it leads more to smoothness in a Sobolev sense than to sparsity in the sense we advocate in this paper, that is, a drastically lower number of random coefficients. Moreover, the approach preserves isotropy and Gaussianity, which is definitely a useful feature for many applications; however, it cannot enforce sparsity for given values of the angular power spectrum sequence, which can be a limitation for some applications.

The purpose of this paper is to introduce a novel notion of sparsity for
isotropic random fields, alongside with a general construction method that
generates isotropic (although generally non-Gaussian) spherical random fields. We
shall exploit this construction in different directions:

\begin{itemize}
\item We shall show how our construction is flexible with respect to
higher-order moments/cumulants of the spherical harmonic coefficients (the
so-called polyspectra). In particular, this construction allows us to generate isotropic, non-Gaussian random fields
with explicit prescribed power spectra and polyspectra, generating set of
random coefficients whose cardinality is (much)
smaller than constructions with similar angular resolutions, based instead on the
spectral representation \cref{eq:harmonic:decomposition}. 

\item We will show, in particular, that these models
 can easily mimic the most important features for Cosmological applications, i.e., those of ``local" non-Gaussianity, while also covering a plethora of cases which are not covered by local models. We refer to the classical survey \cite{BARTOLO2004} for a full overview on other forms of nonGaussianity which can be induced on Cosmic Microwave Background radiation by the so-called inflationary models, the dominant paradigm for the dynamics of the primordial Universe.

\item We shall show how, given the realization of a spherical random field, it is possible to approximate the latter to an arbitrary degree of accuracy in terms of this new representation. More precisely, we will introduce a sparse reconstruction algorithm which is invariant to the choice of coordinates and is computationally feasible, as illustrated in our Monte Carlo simulations. 
\end{itemize}

\medskip

The main idea behind our strategy is to consider a general decomposition of isotropic random fields built as weighted superpositions of deterministic waves centered in random directions
on the sphere, i.e., 
\begin{equation}
T(x)=\sum_{\ell \geq 0}\sum_{k=1}^{K_\ell}\eta _{\ell k}\frac{2\ell +1}{4\pi }%
P_{\ell }(\left\langle \xi _{k},x\right\rangle )\text{ , where }\xi _{k}\sim 
\mathrm{Unif}(\mbbS^{2})\,,  \label{eq:superposition:waves}
\end{equation}%
where $\left\{ \eta _{\ell k}\right\} $ is a finite-variance random sequence, whose properties will be discussed below. 

\begin{remark}
It is immediate to see that the random fields defined in \eqref{eq:superposition:waves} are strongly isotropic; indeed for any $g \in SO(3)$ we have that 
\begin{equation*}
    \begin{aligned}
        T^g(x):=&\,T(gx)=\sum_{\ell \geq 0}\sum_{k=1}^{K_\ell}\eta _{\ell k}\frac{2\ell +1}{4\pi }
P_{\ell }(\left\langle \xi _{k},gx\right\rangle )\\
=&\,\sum_{\ell \geq 0}\sum_{k=1}^{K_\ell}\eta _{\ell k}\frac{2\ell +1}{4\pi }%
P_{\ell }(\left\langle g^{-1}\xi _{k},x\right\rangle )\overset{Law}{=}T(x)\,.
    \end{aligned}
\end{equation*}
\end{remark}

\begin{remark} Neglecting the weighting factors $[\eta _{\cdot,\cdot}]$ and considering fields supported on a single frequency/multipole $\ell$, we obtain as special cases the Spherical Poisson Waves recently introduced by \cite{Bourguin2024}, see also \cite{castaldo2025}.
\end{remark}

We shall show (see \Cref{thm:superposition:representation}) that all isotropic random fields admit the representation~\eqref{eq:superposition:waves}; moreover, with suitable assumptions on the array of random variables $\{\eta_{\ell k}\}$ we will be able to construct fields which mimic the spectral properties of widely popular non-Gaussian cosmological models. 

Heuristically, we will consider an isotropic random field \emph{sparse} if the expansion in \eqref{eq:superposition:waves} involves a "small" number of components. More
precisely, we say that

\begin{definition}\emph{(Weak Sparsity)}
An isotropic spherical random field is weakly
sparse if for any $L\geq 0$ its projection on the spherical harmonics in the first $L$-multipoles $ \{Y_{\ell m}\}_{\ell=1,...,L}$ is measurable
with respect to a $\sigma $-algebra generated by $N$ real random variables, with $%
N=O(L^{2\gamma })$ and $\gamma \in \lbrack 0,1)$. 
\end{definition}

Clearly, any field such as~\eqref{eq:superposition:waves} is weakly sparse when $N=\sum_{\ell}K_{\ell}=O(L^{2\gamma })$, where the $\sigma$-algebra is the one generated by the collection $\{\eta _{\ell k},\xi _{k}\}$. Moreover, it is immediate to see that a Gaussian random field with strictly positive
angular power spectrum cannot be sparse, as its generated $\sigma $-algebra
includes $O(L^{2}$) non-zero, independent Gaussian spherical
harmonic coefficients.

 We refer to the notion above as \emph{weak sparsity}, since this definition still allows for fields where (weak) sparsity is achieved by entirely dropping some multipole components (in some sense, this form of sparsity can be considered closer in spirit to the approach by \cite{Sloan2020}, i.e., implementing a penalization at the level of the frequency components). In particular, isotropic Gaussian random fields can be weakly sparse as soon as their angular power spectrum $(C_\ell)_{\ell\geq 0}$ is non-zero only for a subset whose cardinality grows sublinearly with $\ell\geq 0$, \ie, such that  for any $L\geq 0$ we have $\# \{C_\ell\neq 0\,| \,0\leq \ell\leq L\}\leq L^\gamma$ with $\gamma\in[0,1)$. In order to rule out this possibility, we introduce here the following stronger notion of sparsity.

\begin{definition}\emph{(Strong Sparsity)}
An isotropic spherical random field is
strongly sparse if for any $L\geq 0$ its projection on the first $L$-multipoles spherical harmonics is
measurable with respect to a $\sigma $-algebra generated by $N$ real random
variables, with $N=O((\sum_{\ell\in \mcS(L)}2\ell+1)^\gamma$), $\gamma \in \lbrack 0,1)$ and where  $\mcS(L)$ is the support of the angular power
spectrum in $[0,L]$:
\[\mcS(L) \coloneqq \{0\leq\ell\leq L\,\colon\,C_\ell\neq 0 \}\,.\]
\end{definition}
For reader's convenience, notice that for any field with strictly positive angular power spectrum we have $\sum_{\ell\in \mcS}^L(2\ell+1)=(L+1)^2$.

Clearly strong sparsity implies weak sparsity, whereas the opposite does not hold. In particular, it is possible to obtain weakly sparse Gaussian fields by simply assuming that the angular power spectrum $\{C_{\ell}\}$ is exactly equal to zero for some multipoles $\ell$; however these fields would not be strongly sparse, indeed no Gaussian random field can belong to the strongly sparse class under isotropy (see also \Cref{cor:gauss:no:sparse}).

\subsection{Sparse fields with prescribed bispectrum}

In order to better understand our sparsity proposal, let us recall some properties
that the isotropy assumption imposes on the joint moments (and cumulants) of
the random spherical harmonic coefficients. As regards the first two
moments, it is very well known that the harmonic coefficients must have
expected value zero and be uncorrelated for $(\ell_1 ,m_1)\neq (\ell_2,m_2):$%
\begin{equation*}
\mbbE[a_{\ell_1m_1}\overline{a}_{\ell_2m_2}]=C_{\ell_1 }\delta _{\ell_1
}^{\ell_2}\delta _{m_1}^{m_2}\,,
\end{equation*}%
with $\delta _{a}^{b}$ denoting the Kronecker delta. For the third order
moment (the so-called bispectrum), we have 
\begin{equation*}
\mbbE[a_{\ell _{1}m_{1}}a_{\ell _{2}m_{2}}a_{\ell _{3}m_{3}}]=b_{\ell
_{1}\ell _{2}\ell _{3}}\mathcal{G}_{\ell _{1}\ell _{2}\ell
_{3}}^{m_{1}m_{2}m_{3}}\text{ ,}
\end{equation*}%
where $b_{\ell
_{1}\ell _{2}\ell _{3}}$ is the so-called \emph{reduced bispectrum} and $\mathcal{G}_{\ell _{1}\ell _{2}\ell _{3}}^{m_{1}m_{2}m_{3}}$ represents the
Gaunt integral:%
\begin{equation*}
\mathcal{G}_{\ell _{1}\ell _{2}\ell _{3}}^{m_{1}m_{2}m_{3}}=\int_{\mbbS%
^{2}}Y_{\ell _{1}m_{1}}(z)Y_{\ell _{2}m_{2}}(z)Y_{\ell _{3}m_{3}}(z) \mathrm{%
d} \vol\,.
\end{equation*}
The value of this integral is explicitly known as a function of the index
parameters $\ell _{1},\ell _{2},\ell _{3}$ and $m_{1},m_{2},m_{3}$ (see \cite{Varshalovich1988} Section 5.9 or \cite{marinucci2011random} Sections 3.5.2 and 3.5.3). This
shows that both second and third moments can be decomposed in a part which
solely depends only on the multiple levels $\ell_i$, and another component, depending on both the indexes $\ell_i,m_i$, $i=1,2,3$,
whose role is only to enforce isotropy and which does not carry any physical information on the law of a given field (see again \cite{marinucci2011random}, Chapters 6.5 and 9.2 for more discussion and details). Similar representations hold for joint moments and
cumulants of arbitrary order, which can always be written as a part
depending only on the multipoles $\ell _{1},...,\ell _{p}$ which encodes the
statistical property of the field and a deterministic component which
involves multiple integrals of spherical harmonics and enforces isotropy. 

We shall now show how the representation \eqref{eq:superposition:waves} can be exploited to generate spherical random fields to respect isotropy and hence to produce polyspectra with the required invariance properties. We shall consider the case where the coefficients $\{\eta_{\ell k}\}$ are fully independent from the sequence $\{\xi_k\}$; under these circumstances, starting from the case $K_\ell\equiv 1$, we obtain 
\begin{equation*}
\begin{aligned} &\mbbE[a_{\ell _{1}m_{1}}a_{\ell _{2}m_{2}}]=(4\pi)^{-1}\delta
_{m_1}^{m_2}\,\mbbE[\eta_{\ell_1}\eta_{\ell_2}]\\ &\mbbE[a_{\ell
_{1}m_{1}}a_{\ell _{2}m_{2}}a_{\ell _{3}m_{3}}]=(4\pi)^{-1}\,\mathcal{G}_{\ell _{1}\ell _{2}\ell _{3}}^{m_{1}m_{2}m_{3}}\mbbE[\eta _{\ell _{1}}\eta
_{\ell _{2}}\eta _{\ell _{3}}]
\,. \end{aligned}
\end{equation*}
This shows that moment constraints for the isotropic field are met once suitable constraints are imposed on the random array of weighting factors $
\{\eta_{\ell k}\}$. We refer the reader to \Cref{sec:fNL} where we show how our sparse proposals match the features of popular isotropic non-Gaussian models with spectrum and bispectrum constraints; we will focus in particular on the "local" non-Gaussian model based on the nonlinearity parameter $f_{\mathrm{NL}}$, see again \cite{BARTOLO2004}, \cite{PlanckNG2020} and the references therein for a much more detailed discussion. It is also easy to check (see again \Cref{sec:fNL}) that this can be obtained by means of $O(L)$ random coefficients, rather than $O(L^2)$, so sparsity is indeed achieved.

\begin{remark}[Power spectrum]\label{remark:varianza:eta}
When the coefficients $\{\eta_{\ell k}\}$ are fully independent from the sequence $\{\xi_k\}$, it is immediate to investigate the relationship between the variance of the random weights $\{\eta_{\ell k}\}$ and the angular power spectrum of the corresponding field. Indeed, writing as usual $T_{\ell}=\sum_{m}a_{\ell m} Y_{\ell m}$ for the projection of $T$ on the space spanned by spherical harmonics of order $\ell$, we have 
    \begin{equation*}
    \begin{aligned}
      C_\ell=&\,\frac{4\pi}{2\ell+1} \mbbE\biggl[ \biggl(\sum_{m}a_{\ell m} Y_{\ell m}(x)\biggr)^2\biggr]=\frac{4\pi}{2\ell+1}\mbbE[T_{\ell}^2(x)]\\
     =&\,\frac{4\pi}{2\ell+1}\mbbE\biggl[\biggl(\sum_{k=1}^K\eta _{\ell k}\frac{2\ell +1}{4\pi }
P_{\ell }(\left\langle x,\xi _{k}\right\rangle )\biggr)^2\biggr]\\
=&\,\frac{4\pi}{2\ell+1}\sum_{k,h=1}^{K_\ell}\mbbE[\eta_{\ell k}\eta_{\ell h}] \sum_{m,\tilde m}Y_{\ell m}(x)Y_{\ell \tilde m}(x) \mbbE[Y_{\ell m}(\xi_k)Y_{\ell \tilde m}(\xi_h)]\\
=&\,\frac{4\pi}{2\ell+1}\sum_{k,h=1}^{K_\ell}\mbbE[\eta_{\ell k}\eta_{\ell h}] \sum_{m,\tilde m}Y_{\ell m}(x)Y_{\ell \tilde m}(x)\,\frac{\delta_k^h\,\delta_{m}^{\tilde m}}{4\pi}\\
=&(2\ell+1)^{-1}\sum_{k=1}^{K_\ell}\mbbE[\eta_{\ell k}^2]\,\sum_{m=-\ell}^\ell |Y_{\ell m}(x)|^2=(4\pi)^{-1}\sum_{k=1}^K\mbbE[\eta_{\ell k}^2]
\end{aligned}
    \end{equation*}
    In particular, in case $K_\ell=1$ we have exactly $\mbbE[\eta_\ell^2]=4\pi\,C_\ell$.
\end{remark}

\subsection{Sparse reconstruction algorithm}

\Cref{thm:superposition:representation} shows that any monochromatic isotropic random field (of frequency $\ell$) can almost surely be expressed as the superposition of $2\ell+1$ random waves, centered upon random points. It is therefore natural that, in order to obtain any sparse approximation, we need to focus only a smaller number of \emph{most representatives} random waves and evaluate their superposition. By these means we plan to obtain a sparse approximation of the original random field, independent from the choice of coordinates, with any given degree of accuracy (Gaussianity, however, will be lost). The iterative algorithmic procedure we propose is explained below; it is immediate to verify that the procedure does not depend upon the choice of coordinates, as required to preserve isotropy.

\medskip

   \begin{algorithm}[H]
\SetAlgoLined
\KwIn{Monochromatic field $T_{\ell}$; 
Sparsity parameter $K\ll 2\ell+1$}
Initialize residual field $T_{\ell}(\cdot,0)=T_{\ell}$\\
\For{$k=1,\dots,\,K$}{

Find best direction $\xi_k\in \argmax_{x\in\mbbS^2} |T_{\ell}(x,k-1)|^2$\\
 Update residual field $T_{\ell}(\cdot,k)\coloneqq T_{\ell}(\cdot,k-1)-T_{\ell}(\xi_k,k-1)P_\ell(\langle \xi_k,\cdot\rangle)$
    }
\KwOut{Sparse superposition field $\sum_{k=1}^KT_{\ell}(\xi_k,k-1)P_\ell(\langle \xi_k,\cdot\rangle)$   }
\caption{Sparse reconstruction algorithm}\label{algo:sparse:approx}
\end{algorithm}

\medskip

We postpone to \Cref{sec:proof:algo} a detailed comment on this algorithm, in particular the discussion on the resulting spherical harmonic coefficients. Notice here that the algorithm we have proposed applies also to deterministic random fields. 

Some possible paths for the generalization of the previous iterative algorithm for general, non-monochromatic random fields $T=\sum_{\ell\geq 0}T_{\ell}$ is discussed in  \Cref{sec:proof:algo:colorati}. 
\\

\bigskip
\noindent\textbf{Plan of the paper.} The paper is organized as follows. In \Cref{sec:bispettri} we show how the class of random fields introduced in this paper can mimic the angular power spectrum and higher order spectra for a wide class of isotropic fields, thus providing a useful tool also for efficient simulations; in \Cref{sec:sparse:reconstruction} we prove our main result on sparse representations for isotropic fields and we include a short comparison with the existing literature on sparse approximations. Numerical evidence is given in \Cref{sec:numerics}, while some background material is collected in a short Appendix.

\section{Sparse isotropic random fields with prescribed bispectrum}\label{sec:bispettri}

\label{sec:fNL} In this section we show how our sparse model allows to generate easily some non-Gaussian fields with prescribed bispectrum. More precisely, we show how it is 
possible to generate a sparse random field with the same third order moments
as the ones considered in the very popular single field model of inflation in Cosmology; this model has been addressed and empirically investigated in literally hundreds of papers over the last two decades, see again \cite{BARTOLO2004}, \cite{PlanckNG2020} and the references therein for more discussion.

We are interested in nonlinearly perturbed random fields of the form 
\begin{equation}  \label{eq:fNL:field}
T(x)=T_G (x)+f_{\mathrm{NL}}(T_G^2(x)-\mbbE[T_G (x)]^2)\,,
\end{equation}
where $T_G$ is a mean-zero Gaussian isotropic field with the prescribed
spectrum, whereas $f_{\mathrm{NL}}$ is a (small) nonlinearity parameter. This class of perturbed fields
arises in standard models of inflationary Cosmology (see \cite{BARTOLO2004}, \cite{marinucci2011random}, \cite{PlanckNG2020}).
 The non-Gaussianity parameter $f_{\mathrm{NL}}$
is usually estimated from observed maps via the Komatsu–Spergel%
–Wandelt (KSW) estimator 
(see for example
\cite{KSW2023}, \cite{durastanti2025}).

In this section we are going to show that our sparse model can easily match
the spectrum and bispectrum generated by the class of fields %
\eqref{eq:fNL:field}. To this aim, let $K=1$ in our sparse model %
\eqref{eq:superposition:waves} and consider $\xi\sim \mathrm{Unif}(\mbbS%
^{2}) $ uniformly distributed on the sphere. In order to generate the
weights $\{\eta_{\ell}\}$ in \eqref{eq:superposition:waves}, consider a
sequence of independent, identically distributed real-valued standard
Gaussian random variables $\{z_{\ell}\}$, independent from $\xi$, and set 
\begin{equation*}
\eta_\ell\coloneqq \sqrt{4\pi\,C_\ell}\,z_\ell+3\,f_{%
\mathrm{NL}}\sum_{\ell_1, \ell_2}\sqrt{C_{\ell_1} C_{\ell_2}%
}\,\colon z_{\ell_1}\,z_{\ell_2}\colon
\end{equation*}
where $\colon z_{\ell_1}\,z_{\ell_2}\colon$ denotes the Wick product, see \Cref{App:Wick}.
 Then, one can
easily see that the harmonic coefficients in~%
\eqref{eq:harmonic:decomposition} associated to the sparse field %
\eqref{eq:superposition:waves} (with $N=1$) equal 
\begin{equation*}
\begin{aligned} a_{\ell m}=\int_{\mbbS^{2}}T(x)\overline{Y}_{\ell
m}(x)\mathrm{d} \vol=\sum_{\ell^\prime\geq 0} \eta _{\ell^\prime
}\frac{2\ell^\prime +1}{4\pi }\int_{\mbbS^{2}}P_{\ell^\prime }(\left\langle
\xi ,x\right\rangle )\overline{Y}_{\ell m}(x)\mathrm{d}\vol\\
=\sum_{\ell^\prime\geq 0} \sum_{m^\prime}\eta _{\ell^\prime }
\overline{Y}_{\ell^\prime m^\prime}(\xi)\int_{\mbbS^{2}}Y_{\ell^\prime
m^\prime} (x)\overline{Y}_{\ell
m}(x)\mathrm{d}\vol=\eta_\ell\,\overline{Y}_{\ell m}(\xi)\,. \end{aligned}
\end{equation*}
and hence that the spectrum of this field corresponds to 
\begin{equation*}
\begin{aligned} \mbbE[a_{\ell m}\overline{a}_{\ell ^{\prime }m^{\prime
}}]=\mbbE[\overline{Y}_{\ell m}(\xi)Y_{\ell^\prime
m^\prime}(\xi)]\,\mbbE[\eta_\ell\,\eta_{\ell^\prime}]=(4\pi)^{-1}\,\mbbE[%
\eta_\ell\,\eta_{\ell^\prime}]\,\int_{\mbbS^{2}}\overline{Y}_{\ell
m}Y_{\ell^\prime m^\prime}\,\mathrm{d} \vol\\
=(4\pi)^{-1}\,\delta_{\ell}^{\ell^\prime}\delta_m^{m^\prime}\,\mbbE\biggl[%
\biggl(
\sqrt{4\pi\,C_\ell}\,z_\ell+3\,f_{\mathrm{NL}}\sum_{\ell_1,%
\ell_2}\sqrt{C_{\ell_1}C_{\ell_2}}\,\colon z_{\ell_1}\,z_{\ell_2}\colon%
\biggr)^2\biggr]\\
=\delta_{\ell}^{\ell^\prime}\delta_m^{m^\prime}\,\biggl(C_{\ell
}+9\frac{f_{\mathrm{NL}}^{2}}{4\pi}\sum_{\ell _{1},\ell
_{2}  }(1+\delta_{\ell_1}^{\ell_2})\,C_{\ell _{1}}C_{\ell
_{2}}\biggr)=\delta_{\ell}^{\ell^\prime}\delta_m^{m^\prime}\,\biggl(C_{\ell
}+O(f_{\mathrm{NL}}^{2})\biggr)\,. \end{aligned}
\end{equation*}
Up to lower order terms, this expression is consistent with the spectrum of the perturbed field %
\eqref{eq:fNL:field} whose spectrum indeed equates $C_{\ell }+O(f_{\mathrm{NL}}^{2})$
(see for instance \cite{marinucci2011random}, Chapter 6.7 and \cite{durastanti2025}, Proposition 2 e 3, together with the references in both these works).

Likewise, we may compute the bispectrum as 
\begin{equation*}
\begin{aligned} \mbbE[a_{\ell _{1}m_{1}}a_{\ell _{2}m_{2}}a_{\ell
_{3}m_{3}}]= \mbbE[\eta _{\ell _{1}}\eta _{\ell _{2}}\eta _{\ell _{3}}]
\,(4\pi)^{-1}\,\int_{\mbbS^{2}}\overline{Y}_{\ell
_{1}m_{1}}\overline{Y}_{\ell _{2}m_{2}}\overline{Y}_{\ell
_{3}m_{3}}\mathrm{d}\vol\\ =\mbbE[\eta _{\ell _{1}}\eta _{\ell _{2}}\eta
_{\ell _{3}}] \,\sqrt{\frac{(2\ell_1+1)(2\ell_2+1)(2\ell_3+1)}{(4\pi)^3}}
\begin{pmatrix} \ell_1 & \ell_2 & \ell_3 \\ 0 & 0 & 0 \end{pmatrix}
\begin{pmatrix} \ell_1 & \ell_2 & \ell_3 \\ m_1 & m_2 & m_3 \end{pmatrix}\\
=\mbbE[\eta _{\ell _{1}}\eta _{\ell _{2}}\eta _{\ell _{3}}]
\,(4\pi)^{-1}\,\mathcal{G}_{\ell _{1}\ell _{2}\ell _{3}}^{m_{1}m_{2}m_{3}}
\end{aligned}
\end{equation*}
where we have expressed the Gaunt integral $\mathcal{G}_{\ell _{1}\ell
_{2}\ell _{3}}^{m_{1}m_{2}m_{3}} $ in terms of the Wigner 3j-symbols \cite[Equation 5 in Section 5.9]{Varshalovich1988}.  
Therefore the reduced bispectrum equals 
\begin{equation*}
\begin{aligned}
b_{\ell_1\ell_2\ell_3}=&\,(4\pi)^{-1} \mbbE[\eta _{\ell _{1}}\eta _{\ell
_{2}}\eta _{\ell _{3}}]
=3\,f_{\mathrm{NL}}\,\mbbE\biggl[\biggl(\sqrt{C_{\ell_1}C_{\ell_2}}z_{\ell_1}z_{\ell_2}+\sqrt{C_{\ell_2}C_{\ell_3}}z_{\ell_2}z_{\ell_3}\\
+&\sqrt{C_{\ell_3}C_{\ell_1}}z_{\ell_3}z_{\ell_1}\biggr)\sum_{\tilde\ell_1,\tilde\ell_2}\sqrt{C_{\tilde\ell_1}C_{\tilde\ell_2}}\colon z_{\tilde\ell_1}z_{\tilde\ell_2}\colon\biggr]
+\frac{27\,f_{\mathrm{NL}}^3}{4\pi}\,\mbbE\biggl[\biggl(\sum_{\tilde\ell_1,\tilde\ell_2}\sqrt{C_{\tilde\ell_1}C_{\tilde\ell_2}}\colon z_{\tilde\ell_1}z_{\tilde\ell_2}\colon\biggr)^3\biggr]\\
=&\,6\,f_{\mathrm{NL}}\biggl(C_{\ell_1}C_{\ell_2}+C_{\ell_2}C_{\ell_3}+C_{\ell_3}C_{\ell_1}\biggr)+\frac{54\,f_{\mathrm{NL}}^3}{\pi}\biggl(\sum_{\ell}C_{\ell}\biggr)^3\,,
\end{aligned}
\end{equation*}
which matches the $f_{\mathrm{NL}}$-first order bispectrum for the Sachs-Wolfe model with Bardeen's potential associated to \eqref{eq:fNL:field}  \cite[Section 6.7.1, Equation 6.55]{marinucci2011random}, commonly considered for studying the Cosmic Microwave Background radiation. 

\begin{remark}
Our model can be easily generalized in order to take into account different non-Gaussianities, beyond the bispectrum prescription. Indeed we may in general consider a strongly sparse field \eqref{eq:superposition:waves}, with $K_\ell=1$ generated by
  $\xi\sim \mathrm{Unif}(\mbbS%
^{2}) $ uniformly distributed on the sphere, by considering the random weights
\begin{equation*}
\eta _{\ell }:=\sqrt{4\pi C_{\ell }}z_{\ell }+\sum_{\ell _{1},\ell _{2}}c^\ell_{\ell _{1},\ell _{2}}\,\colon z_{\ell _{1}}z_{\ell _{2}}\colon+\sum_{\ell
_{1}, \ell _{2},\ell _{3} }c^\ell_{\ell _{1},\ell
_{2},\ell _{3}}\,\colon z_{\ell _{1}}z_{\ell _{2}}z_{\ell _{3}}\colon+...
\end{equation*}%
where $\colon\dots\colon$ denotes the Wick product, $\{z_{\ell}\}$ is 
an independent sequence of i.i.d. real-valued standard
Gaussian random variables, and $c^\ell_{\ell _{1},\ell _{2}\dots}$ are parameters that should be chosen in order to fit the prescribed non-Gaussian statistics. For the general model, the reduced bispectrum would be%
\begin{equation*}
\begin{aligned}
b_{\ell_1\ell_2\ell_3}=&(4\pi)^{-1} \mbbE[\eta _{\ell _{1}}\eta _{\ell
_{2}}\eta _{\ell _{3}}]
=2\sqrt{C_{\ell _{1}}C_{\ell _{2}}}\,c^{\ell _{3}}_{\ell _{1},\ell
_{2}}+2\sqrt{C_{\ell
_{2}}C_{\ell _{3}}}\,c^{\ell _{1}}_{\ell _{2},\ell _{3}}\\
&+2\sqrt{C_{\ell
_{3}}C_{\ell _{1}}}\,c^{\ell _{2}}_{\ell _{1},\ell _{3}} 
+\frac{2}{\pi}\sum_{\tilde\ell _{1},\tilde\ell _{2}, \tilde\ell _{3}}c^{\tilde\ell _{1}}_{\tilde\ell
_{2},\tilde\ell _{3}} c^{\tilde\ell _{2}}_{\tilde\ell
_{1},\tilde\ell _{3}} c^{\tilde\ell _{3}}_{\tilde\ell
_{1},\tilde\ell _{2}}+\dots
\end{aligned}
\end{equation*}%
where the lower-order terms depends on the Wick products of higher order. 
For this bispectrum to be maximized by equilateral configurations where $%
\ell _{1}\simeq \ell _{2}\simeq \ell _{3},$ it is enough to assume that $%
C_{\ell _{3}}$ decays fast and $c^\ell_{\ell _{1},\ell _{2}}$ resembles a delta
function. For instance, assuming that%
\begin{equation*}
c^\ell_{\ell _{1},\ell _{2}}=\frac{\sqrt{C_{\ell }}}{(1+|\ell _{1}-\ell
_{2}|)^{\gamma }}\text{ , }\gamma \gg 0\text{ ,}
\end{equation*}%
leads to
\begin{eqnarray*}
b_{\ell_1\ell_2\ell_3}&\approx&\frac{2\sqrt{C_{\ell_1}C_{\ell _{2} }C_{\ell_3}}}{(1+|\ell _{2}-\ell
_{3}|)^{\gamma }}+\frac{2\sqrt{C_{\ell_1}C_{\ell _2}C_{\ell_3}}}{(1+|\ell _{1}-\ell _{3}|)^{\gamma }} +\frac{2\sqrt{C_{\ell_1}C_{\ell _{2} }C_{\ell_3}}}{(1+|\ell
_{1}-\ell _{2}|)^{\gamma }}\\
&&+\frac{2}{\pi}\sum_{\tilde\ell _{1},\tilde\ell _{2}, \tilde\ell _{3}}%
\frac{\sqrt{C_{\tilde\ell _{3}}}}{(1+|\tilde\ell _{1}-\tilde\ell _{2}|)^{\gamma }}\frac{\sqrt{%
C_{\tilde\ell _{1}}}}{(1+|\tilde\ell _{2}-\tilde\ell _{3}|)^{\gamma }}\frac{\sqrt{C_{\tilde\ell _{2}}%
}}{(1+|\tilde\ell _{1}-\tilde\ell _{3}|)^{\gamma }}\, .
\end{eqnarray*}
\end{remark}
We refer to \cite{NG2025} for much more discussion on alternative forms of nonGaussianity arising from the inflationary scenarios (for instance, those generated by so-called \emph{multiple fields inflation}) and the corresponding bispectrum shapes.

\section{A general decomposition for isotropic random fields}

\label{sec:sparse:reconstruction}

As mentioned above, one of our aims is to show that each isotropic random field admits a general decomposition of the form \eqref{eq:superposition:waves}. 
More precisely, we shall show that
\begin{theorem}\label{thm:superposition:representation}
For any isotropic spherical random field  $T=\sum_{\ell
=0}^{\infty}a_{\ell m}Y_{\ell m}$ there exist a sequence $\{K_\ell\}_{\ell\geq 0}$ with $K_\ell\leq 2\ell+1$ and a sequence of real random weights $\{\eta _{\ell k}\}$ such that almost surely it holds
\begin{equation*}
T(x)=\sum_{\ell \geq 0}\sum_{k=1}^{K_\ell}\eta _{\ell k}\frac{2\ell +1}{4\pi }
P_{\ell }(\left\langle x,\xi _{k}\right\rangle )\text{ , where }\xi _{k}\sim 
\mathrm{Unif}(\mbbS^{2})\text{ , }k=1,2,...
\end{equation*}
\end{theorem}
 This result provides a representation of any isotropic random field from which it is easier to investigate its (strong or weak) sparsity properties.
Moreover,  it suggests how to achieve sparse approximations for isotropic spherical random fields: in a first step it is possible to approximate with a finite multipole expansion and the desired degree of accuracy any field, and in the second to use an expansion such as \eqref{eq:superposition:waves} for random directions that grow sublinearly (i.e., $K_{\ell}=o(\ell)$).

 In this section we prove \Cref{thm:superposition:representation}; we also argue that no Gaussian isotropic random field can be strongly sparse. Before we do so, however, we believe it can be useful to recall a few previous attempts to achieve sparse representations for isotropic random fields and to compare them with our current proposal.

\subsection{A comparison with the existing literature}\label{Comparison}
Quite a few papers over the last 15 years have attempted different forms of sparse reconstruction techniques for isotropic random fields on $\mbbS^2$. The reasons for this interest is easily understood, again referring to Cosmological applications: Cosmic Microwave Background maps are in fact observed with several forms of super-imposed noise, which can have either instrumental origin or can be generated by astrophysical contaminants (foregrounds). Overall, we are in the framework of image reconstruction techniques, where of course sparsity enforcing tools have long proved to be efficient and practical; in this case, however, we want to ensure also that the reconstruction will not produce spurious statistical features (like anisotropies) in the random map that we wish to analyze. Probably the single most popular sparsity enforcing techniques in the mathematical statistics literature is the $Lasso$, which imposes a $\ell_1$ convex regularization term into a standard quadratic loss procedure; in the case of spherical random fields, this suggests to search for solutions of the following convex minimization problem:

\begin{equation}\label{Lasso}
\left [ T^{+}:=\sum_{\ell,m}a^{+}_{\ell,m}Y_{\ell,m}\right ]:= \argmin_{T^{b}:=\sum_{\ell,m}b_{\ell,m}Y_{\ell,m}} \left [\|T(\cdot)-T^{b}(\cdot)\|_{\mrmL^2} + \lambda \sum_{\ell,m}|b_{\ell,m}|\right ],
\end{equation}
where $\lambda$ is a penalization parameter. It is well-known from the standard literature (see e.g. \cite{Vandegeer2011}) that these optimization procedures provide a convex approximation for the solution with a $\ell_0$ loss; the latter would lead to sparsity in a strict sense (in some sense analogous to our terminology in this paper), but by itself is $NP-$hard from the computational point of view and hence numerically unfeasible. The convex regularization procedures leads to solutions which are arbitrarily close with very high probability, see some classical textbooks such as \cite{Vandegeer2011}; in the context of CMB data analysis, regularization procedures based on the Lasso are advocated for instance in \cite{Starck2013}. It was however argued in \cite{Feeney} and \cite{CammarotaAcha2015} that these forms of Lasso procedures will lead to maps which do not satisfy Gaussianity and isotropy assumptions; heuristically, it is indeed immediate to notice that the minimization procedure introduced in \cref{Lasso} is not invariant with respect to the choice of coordinates (the $||\cdot||_{\mrmL^2}$, "fidelity" term is rotationally invariant, but the $|\cdot|_{\ell_1}$ regularization term clearly isn't - see \cite{Sloan2020} for a simple and illuminating counterexample).

To overcome these issues sparsity enforcing techniques that can preserve isotropy and Gaussianity have been recently proposed in \cite{Sloan2020}, whose authors suggest the regularized estimates

\begin{equation}\label{Sloan}
\left [ T^{+}:=\sum_{\ell,m}a^{+}_{\ell,m}Y_{\ell,m}\right ]:= \argmin_{T^{b}:=\sum_{\ell,m}b_{\ell,m}Y_{\ell,m}} \left [\|T(\cdot)-T^{b}(\cdot)\|_{\mrmL^2} + \lambda \sum_{\ell}\beta_{\ell}\sqrt{\sum_{m}|b_{\ell,m}|^2}\right ],
\end{equation}
for constants $\beta_{\ell},\lambda$ to be determined. This procedure is shown to preserve Gaussianity and isotropy and to perform well on simulations; in this sense, it is theoretically well-grounded and empirically valid. However, by imposing a penalization on the reconstructed angular power spectrum $(2\ell+1)\widehat{C}_{\ell}:=\sum_{m}|b_{\ell,m}|^2$ it is actually performing a form of Sobolev regularization, so ensuring smoothness more than sparsity in the sense defined in this paper. 

Summing up, the procedure we advocate in this work has different goals and takes a new perspective. It does not preserve Gaussianity, nor it could because we argued that sparsity and Gaussianity are incompatible; on the other hand, it does achieve a form of sparsity/data compression, in the sense that it approximates the input field with a much lower number of coefficients, at the same time preserving isotropy (and invariance to the choice of coordinates).

\subsection{The proof of the general decomposition}\label{sec:proof:thm}

\begin{proof}[Proof of \Cref{thm:superposition:representation}]
Firstly, take $T=\sum_{\ell
=0}^{\infty}\sum_{m=-\ell}^\ell a_{\ell m}Y_{\ell m}$ and take i.i.d. uniform random variables on the sphere $(\xi_k)_{k\in\mbbN} \overset{\text{i.i.d.}}{\sim}\mathrm{Unif}(\mbbS^2)$. 
Now fix $\ell\geq 0$, any non-negative number $K_\ell$ and for each $k\in\{1,\dots,K_\ell\}$ consider the column vector%
\begin{equation*}
\mathbf{Y}_{\ell }(\xi_k):=(Y_{\ell ,-\ell }(\xi _{k}),...,Y_{\ell ,\ell }(\xi_{k}))^{T}\,,
\end{equation*}%
and the corresponding random matrix generated by these column vectors 
\begin{equation*}
\mathbb{Y}_{\ell ,K_\ell}:=\left\{ \mathbf{Y}_{\ell }(\xi_1),...\mathbf{Y}_{\ell
}(\xi_{K_\ell})\right\} \,.
\end{equation*}%

\medskip

\noindent\textbf{Step 1.}
We claim that $\mathrm{rank}(\mathbb{Y}_{\ell ,K_\ell})=K_\ell\wedge (2\ell+1)$ with probability one. This follows from the left action $\mcL$  of $SO(3)$ on $\mrmL^2(\mbbS^2)$ which induces a representation of $SO(3)$ on $\mrmL^2(\mbbS^2)$. More precisely, if $\mcH_{\ell}$ denotes the function space generated by the spherical harmonics $\{Y_{\ell,m}\}_{m=-\ell,\dots,\ell}$, then the couple $(\mcL,\mcH_{\ell})$ defines an irreducible representation of $SO(3)$ \cite[Proposition 3.27]{marinucci2011random}, which means that no non-trivial subspace of $\mcH_\ell$ is preserved under the action~$\mcL$. This implies the existence of $2\ell+1$ points $\zeta_k\in\mbbS^2$  such that $\{\textbf{Y}_\ell(\zeta_1),\dots,\textbf{Y}_\ell(\zeta_{2\ell+1})\}$ is a non-degenerate matrix (see also \cite[Lemma 6]{Muller1966SphericalHarmonics}). Therefore if we set $d=K\wedge (2\ell+1)$, there exists at least a $d\times d$ minor $[\textbf{Y}_\ell(\zeta_1),\dots,\textbf{Y}_\ell(\zeta_d)]_{d\times d}$ non-degenerate and hence such that $\det([\textbf{Y}_\ell(\zeta_1),\dots,\textbf{Y}_\ell(\zeta_d)]_{d\times d})\neq 0$ (up to relabeling the points we may assume it involves the first $d$ points, and hereafter we fix one such minor and for any matrix $[\cdot]_{d\times d}$ denotes its minor in correspondence with our first choice). 
Next, consider the function $F\colon (\mbbR^3)^d\to\mbbC$ defined as
\[F(x_1,x_2,\dots,x_d)\coloneqq \det\biggl(\biggr[\textbf{Y}_\ell\biggl(\frac{x_1}{\|x_1\|}\biggr),\dots,\textbf{Y}_\ell\biggl(\frac{x_d}{\|x_d\|}\biggr)\biggr]_{d\times d} \biggr)\,. \]
On the open domain $\Omega=(\mbbR^3)^d\setminus B_{\nicefrac12}(0)$ this is a real-analytic function, more precisely both its real and imaginary components $\textit{Re}(F),\,\textit{Im}(F)$ are real-analytic. 

Now, suppose by absurd that 
\begin{equation*}
\mbbP[\mathrm{rank}([\mathbb{Y}_{\ell ,K_\ell}]_{d\times d})<d]=\mbbP[\det([\textbf{Y}_\ell(\xi_1),\dots,\textbf{Y}_\ell(\xi_d)]_{d\times d} )]=0]=\mbbP[F(\xi_1,\dots,\xi_d)]>0\,,
\end{equation*}
and hence that $F$ vanishes on a subset of the sphere with positive spherical-volume. Then, since $F$ is constant along radial trajectories, this further implies that $F$ vanishes on a subset of $\Omega$ with positive Lebesgue measure. Since $\textit{Re}(F),\,\textit{Im}(F)$ are real-analytic and their zeros sets have positive measure we conclude that both $\textit{Re}(F)$ and $\textit{Im}(F)$ are null \cite[Proposition 1]{Mityagin2020}. This yields to $F\equiv 0$ on $\Omega$ which clearly contradicts the fact that $F(\xi_1,\dots,\xi_d)=\det([\textbf{Y}_\ell(\zeta_1),\dots,\textbf{Y}_\ell(\zeta_d)]_{d\times d})\neq 0$.

Therefore we may conclude that $\mbbP[\mathrm{rank}([\mathbb{Y}_{\ell ,K_\ell}]_{d\times d})<d]=0$
and hence that with probability one the random matrix $\mathbb{Y}_{\ell ,K_\ell}$ has a $d\times d$ minor of full-rank and hence that  $\mathrm{rank}(\mathbb{Y}_{\ell ,K_\ell})=K_\ell\wedge (2\ell+1)$.

\medskip

\noindent\textbf{Step 2.} As a direct consequence of the first step, if we take $K_\ell=2\ell+1$, almost surely there exist coefficients  $\eta
_{\ell 1},...,\eta_{\ell K_\ell}\in\mbbC$ such that%
\begin{equation*}
\mathbf{a}_{\ell }\coloneqq (a_{\ell m})_{m=-\ell,\dots,\ell}=\sum_{k=1}^{K_\ell}\eta_{\ell k }\mathbf{Y}_{\ell}(\xi_k)=\mathbb{Y}
_{\ell , K} \mathbf{\eta}_{\ell }\,,
\end{equation*}%
where $\mathbf{\eta }_{\ell }=(\eta _{\ell 1},...,\eta _{\ell K_\ell})^{T}$. Therefore, for any fixed $\ell\in\mbbN$ we have almost surely 
\begin{equation}\label{eq:dec:T:ell}
\begin{aligned}
    T_\ell(x)\coloneqq&\,\sum_{m=-\ell}^\ell a_{\ell m}Y_{\ell m}(x)=\mathbf{a}_{\ell }^T\, \mathbf{Y}_{\ell}(x)=\mathbf{\eta}_{\ell }^T\mathbb{Y}_{\ell , K_\ell}^T\mathbf{Y}_{\ell}(x)\\
    =&\,\sum_{k=1}^{K_\ell} \eta_{\ell k}\, \sum_{m=-\ell}^\ell Y_{\ell m}(\xi_k)Y_{\ell m}(x)=\sum_{k=1}^{K_\ell} \eta_{\ell k}\frac{2\ell+1}{4\pi} P_{\ell}(\langle x,\xi_k\rangle)\,.
\end{aligned}
\end{equation}
Moreover, since the field $T$ (and hence $T_\ell$) is real, from the independence of the uniform directions $(\xi_k)_{k=1,\dots,K_\ell} \overset{\text{i.i.d.}}{\sim}\mathrm{Unif}(\mbbS^2)$ and from the above expression we conclude that the random weights $(\eta_{\ell k})_{\ell,k}$ are real as well. Finally, since  $\ell\in\mbbN$ and for any fixed $\ell\geq 0$ the decomposition~\eqref{eq:dec:T:ell} holds almost surely, we may conclude that our decomposition holds almost surely for the whole field $T=\sum_{\ell=0}^\infty T_\ell $.
\end{proof}

The previous proof further provides an explicit construction for the coefficients $(\eta_{\ell k})_{k=1,\dots,K}$ as a function of harmonic coefficients. Indeed, if the $\ell^{th}$ frequency component of a strongly sparse random field equals $T_{\ell}=\sum_{m=-\ell}^\ell a_{\ell m}Y_{\ell m}$, then the coefficients of the strongly sparse representation can be easily computed  by applying the Moore-Penrose inverse:
\[\mathbf{\eta}_{\ell \cdot}=(\mathbb{Y}_{\ell,K}^\star\mathbb{Y}_{\ell,K})^{-1}\mathbb{Y}_{\ell,K}^\star \mathbf{a}_\ell\,.\]
This further suggests that the class of strongly sparse random fields does not contain Gaussian fields. This is a standard result that we recall here with our notation.

\begin{corollary}\label{cor:gauss:no:sparse}
    No monochromatic isotropic Gaussian random field  can be sparse.
    As a consequence, no Gaussian isotropic random field can be strongly sparse.
\end{corollary}
\begin{proof}
It suffices to show that no monochromatic Gaussian random field $G_{\ell}=\sum_{m=-\ell}^\ell a_{\ell m}^GY_{\ell m}$, \ie, such that there are $K<2\ell+1$  i.i.d. uniform random variables on the sphere $(\xi_k)_{k=1,\dots,K}\overset{\text{i.i.d.}}{\sim}\mathrm{Unif}(\mbbS^2)$ such that  
    \begin{equation*}
        G_{\ell}(x)=\sum_{k=1}^{K} \eta_{\ell k}\frac{2\ell+1}{4\pi} P_{\ell}(\langle x,\xi_k\rangle)\,.
    \end{equation*}
Now, assuming that such decomposition holds and considering the real-spherical harmonics $\{Y_{\ell m}^R\}_{m=-\ell,\dots,\ell}$ it is immediate to see that the (real-valued) harmonic coefficients associated to $G_{\ell}$ are given by
 \begin{equation*}
   \begin{aligned}
  a_{\ell m}^R=&\,\int_{\mbbS^2}G_{\ell}(x)\,Y_{\ell m}(x)\De \vol=\sum_{k=1}^{K} \eta_{\ell k}\frac{2\ell+1}{4\pi} \int P_{\ell}(\langle x,\xi_k\rangle)Y_{\ell m}^R(x)\,\De\vol\\
  =&\,\sum_{k=1}^{K} \eta_{\ell k}\,Y_{\ell m}^R(\xi_k)\,.
   \end{aligned}
   \end{equation*}
   Therefore, the $(2\ell+1)$-Gaussian vector $\textbf{a}^R_{\ell}$ can be written as a combination of $K<2\ell+1$ random variables $\{\textbf{Y}_{\ell}^R(\xi_k)\}_{k=1,\dots,K}$, which is clearly a contradiction.

\end{proof}

\begin{remark}[Empirical power spectrum]\label{remark:varianza:empirica:eta}
Even though the random weights $\eta_{\ell k}$ are not independent of the directions $\xi_k$ in general, we can still study the relationship between the realization of these random variables and the empirical angular power spectrum $\widehat{C}_\ell$ associated to our field's realization. More precisely we may mimic the computations performed for \Cref{remark:varianza:eta} and deduce that, almost surely for any fixed realization, we have
    \begin{equation*}
    \begin{aligned}
      \widehat{C}_\ell=&\frac{1}{2\ell+1}\sum_{m=-\ell}^\ell |\widehat{a}|^2_{\ell m} =\frac{1}{2\ell+1}\sum_{m=-\ell}^\ell  \biggl|\int_{\mbbS^2} T_\ell(x)Y_{\ell m}(x)\,\De \vol(x)\biggr|^2\\
      =&\,\frac{1}{2\ell+1}\sum_{m=-\ell}^\ell\biggl|\int_{\mbbS^2}\sum_{k=1}^{K_\ell}\eta _{\ell k}\frac{2\ell +1}{4\pi }
P_{\ell }(\left\langle x,\xi _{k}\right\rangle )Y_{\ell m}(x)\,\De \vol(x)\biggr|^2\\
=&\,\frac{1}{2\ell+1}\sum_{m=-\ell}^\ell\biggl(\int_{\mbbS^2}\sum_{k=1}^{K_\ell}\eta _{\ell k}\sum_{m^\prime=-\ell}^\ell Y_{\ell m^\prime}(\xi_k)\overline{Y}_{\ell m^\prime}(x) Y_{\ell m}(x)\,\De \vol(x)\biggr)^2\\
=&\,\frac{1}{2\ell+1}\sum_{m=-\ell}^\ell\biggl|\sum_{k=1}^{K_\ell}\eta _{\ell k}Y_{\ell m}(\xi_k)\biggr|^2=\frac{1}{2\ell+1}\sum_{h,k=1}^{K_\ell}\eta_{\ell k}\eta_{\ell h}\mathbf{Y}_\ell(\xi_k)^T\overline{\mathbf{Y}}_\ell(\xi_h)\,.
\end{aligned}
    \end{equation*}
    Therefore $\widehat{C}_\ell$ can explicitly be calculated as the quadratic norm of the vector of weights $\{\eta_{\ell k}\}_{k}$ weighted with the spherical harmonics scalar product  $\mathbf{Y}_\ell(\xi_k)^T\overline{\mathbf{Y}}_\ell(\xi_h)$. In particular, for $K_{\ell}=1$ we have  $4\pi\widehat{C}_\ell=\eta_{\ell}^2$.
\end{remark}

\subsection{Sparse reconstruction algorithm for monochromatic fields}\label{sec:proof:algo}
We conclude this section by analyzing \Cref{algo:sparse:approx} by showing how it can be equivalently defined at the harmonic coefficients' level. Recall that at each iteration step we choose the best direction by solving 
\begin{equation}\label{eq:def:xi_k}
    \xi_k\in \argmax_{x\in\mbbS^2} |T_{\ell}(x,k-1)|^2= \argmax_{x\in\mbbS^2} |\langle \mathbf{a}_\ell(k-1),\,\mathbf{Y}_\ell(x)\rangle|^2\,,
\end{equation}
where $\mathbf{a}_\ell(k)$ denote the harmonic coefficients of the residual field $T_{\ell}(\cdot,k)$ which is iteratively defined via
\begin{equation*}T_{\ell}(\cdot,k)\coloneqq T_{\ell}(\cdot,k-1)-T_{\ell}(\xi_k,k-1)P_\ell(\langle \xi_k,\cdot\rangle)\,.
\end{equation*}
Therefore, the update for the harmonic coefficients can be directly defined as a Gram-Schmidt orthogonalization since
\begin{equation}\label{eq:sum:over:k:iterative:step}\begin{aligned}
    \mathbf{a}_\ell(k)=&\,\int_{\mbbS^2} T_{\ell}(x,k)\mathbf{Y}_\ell(x)\De\vol(x)\\
    =&\int_{\mbbS^2} T_{\ell}(x,k-1)\mathbf{Y}_\ell(x)\De\vol(x)-T_{\ell}(\xi_k,k-1)\int_{\mbbS^2} P_\ell(\langle \xi_k,x\rangle)\mathbf{Y}_\ell(x)\De\vol(x)\\
    =&\,\mathbf{a}_\ell(k-1)-T_{\ell}(\xi_k,k-1)\frac{4\pi }{2\ell +1}\mathbf{Y}_\ell(\xi_k)=\mathbf{a}_\ell(k-1)-\frac{\langle \mathbf{a}_\ell(k-1),\,\mathbf{Y}_\ell(\xi_k)\rangle}{\|\mathbf{Y}_\ell(\xi_k)\|^2}\,\mathbf{Y}_\ell(\xi_k)\,.
\end{aligned}\end{equation}
Therefore, our algorithm reads as follows

\medskip

   \begin{algorithm}[H]
\SetAlgoLined
\KwIn{Monochromatic field's harmonics $\mathbf{a}_\ell$; 
Sparsity parameter $K\ll 2\ell+1$}
Initialize residual harmonic coefficients $\mathbf{a}_\ell(0)=\mathbf{a}_\ell$\\
\For{$k=1,\dots,\,K$}{

Find best direction $\xi_k\in \argmax_{x\in\mbbS^2} |\langle \mathbf{a}_\ell(k-1),\,\mathbf{Y}_\ell(x)\rangle|^2$\\
 Update residual harmonics $\mathbf{a}_\ell(k)\coloneqq \mathbf{a}_\ell(k-1)-\frac{\langle \mathbf{a}_\ell(k-1),\,\mathbf{Y}_\ell(\xi_k)\rangle}{\|\mathbf{Y}_\ell(\xi_k)\|^2}\,\mathbf{Y}_\ell(\xi_k)$
    }
\KwOut{Sparse superposition field $S^K(\cdot)\coloneqq \sum_{k=1}^K\langle \mathbf{a}_\ell(k-1),\,\mathbf{Y}_\ell(\xi_k)\rangle P_\ell(\langle \xi_k,\cdot\rangle)$  }
\caption{Sparse reconstruction algorithm: harmonic coefficients}\label{algo:sparse:approx:harmonic}
\end{algorithm}

\medskip

Notice that by summing over $k=1,\,\dots,\,K$ in~\eqref{eq:sum:over:k:iterative:step} we immediately see that our sparse approximation field equals
 \begin{equation}\label{eq:approx:field:equiv}
 \begin{aligned}
S^K(\cdot)=&\,\sum_{k=1}^K\langle \mathbf{a}_\ell(k-1),\,\mathbf{Y}_\ell(\xi_k)\rangle P_\ell(\langle \xi_k,\cdot\rangle)=\sum_{m=-\ell}^\ell\biggl(\sum_{k=1}^K\frac{\langle \mathbf{a}_\ell(k-1),\,\mathbf{Y}_\ell(\xi_k)\rangle}{\|\mathbf{Y}_\ell(\xi_k)\|^2}\,\mathbf{Y}_{\ell m}(\xi_k)\biggr)Y_{\ell m}(\cdot)\\
=&\,\biggl\langle \sum_{k=1}^K\frac{\langle \mathbf{a}_\ell(k-1),\,\mathbf{Y}_\ell(\xi_k)\rangle}{\|\mathbf{Y}_\ell(\xi_k)\|^2}\,\mathbf{Y}_{\ell}(\xi_k),\mathbf{Y}_\ell(\cdot)\biggr\rangle=\langle \mathbf{a}_\ell(0)-\mathbf{a}_\ell(K),\mathbf{Y}_\ell(\cdot)\rangle\\
=&\,\langle \mathbf{a}_\ell,\mathbf{Y}_\ell(\cdot)\rangle -\langle \mathbf{a}_\ell(K),\mathbf{Y}_\ell(\cdot)\rangle =T_{\ell}(\cdot)-T_{\ell}(\cdot,K)\,.
\end{aligned}
 \end{equation}

 \bigskip

The harmonic coefficients formulation obtained in~\Cref{algo:sparse:approx:harmonic} further helps us in understanding the performance of the algorithm; the latter can be seen as a Gram-Schmidt orthogonalization for $\mathbf{a}_\ell$, as we detail in the next result.

\begin{theorem}\label{thm:conv:algo:mono}
Given $(\xi_k)_{k=1,\dots,2\ell+1}$ computed according to  \Cref{algo:sparse:approx:harmonic} with $K=2\ell+1$, the vectors $\{\mathbf{Y}_\ell(\xi_1),\dots,\mathbf{Y}_\ell(\xi_{2\ell+1})\}\subseteq \mbbC^{2\ell+1}$ are orthogonal and hence linearly independent. Thus, $\mathbf{a}_\ell(2\ell+1)=\underline{0}$ and $T_{\ell}(\cdot,2\ell+1)\equiv 0$. Equivalently, the reconstruction algorithm automatically stops after $K=2\ell+1$ steps and fully recovers the original field
\begin{equation}\label{eq:decomposition:sparse:algo}
    T_{\ell}=\sum_{k=1}^{2\ell +1}T_{\ell }(\xi_{k};k-1)P_{\ell
}(\langle \xi_{k},\cdot\rangle )=\sum_{k=1}^K\langle \mathbf{a}_\ell(k-1),\,\mathbf{Y}_\ell(\xi_k)\rangle P_\ell(\langle \xi_k,\cdot\rangle)\,.
\end{equation}
\end{theorem}
\begin{proof}
    It is enough to show that the vectors $\{\mathbf{Y}_\ell(\xi_1),\dots,\mathbf{Y}_\ell(\xi_{2\ell+1})\}\subseteq \mbbC^{2\ell+1}$ are orthogonal.
    In view of that, let us first notice that
    \begin{equation*}
        \begin{aligned}
             \mathbf{a}_\ell(k)=\mathbf{a}_\ell(k-1)-\frac{\langle \mathbf{a}_\ell(k-1),\,\mathbf{Y}_\ell(\xi_k)\rangle}{\|\mathbf{Y}_\ell(\xi_k)\|^2}\,\mathbf{Y}_\ell(\xi_k)=\mathrm{proj}_{\mathbf{Y_\ell}(\xi_k)^\perp}(\mathbf{a}_\ell(k-1))\,.
        \end{aligned}
    \end{equation*}
We will show by induction that $\xi_k$ is chosen such that $\mathbf{Y}_\ell(\xi_k)\in\{\mathbf{Y}_\ell(\xi_1),\dots,\mathbf{Y_\ell}(\xi_{k-1})\}^\perp$  and  that
\begin{equation*}
        \begin{aligned}
             \mathbf{a}_\ell(k)=\mathrm{proj}_{\{\mathbf{Y}_\ell(\xi_1),\dots,\mathbf{Y_\ell}(\xi_k)\}^\perp}(\mathbf{a}_\ell(0))=\mathrm{proj}_{\{\mathbf{Y}_\ell(\xi_1),\dots,\mathbf{Y_\ell}(\xi_k)\}^\perp}(\mathbf{a}_\ell)\,,
        \end{aligned}
    \end{equation*}
The base case $k=1$ trivially follows from our previous discussion. Now assume our claim holds for $k-1$; by  definition of $\xi_{k}$ we immediately see that
    \begin{equation*}
    \xi_k\in  \argmax_{x\in\mbbS^2} |\langle \mathbf{a}_\ell(k-1),\,\mathbf{Y}_\ell(x)\rangle|^2=\argmax_{x\in\mbbS^2} \left|\biggl\langle \mathrm{proj}_{\{\mathbf{Y}_\ell(\xi_1),\dots,\mathbf{Y_\ell}(\xi_{k-1})\}^\perp}(\mathbf{a}_\ell),\,\mathbf{Y}_\ell(x)\biggr\rangle\right|^2\,.
\end{equation*}
This shows that $\xi_k$ is chosen such that $\mathbf{Y}_\ell(\xi_k)\in\{\mathbf{Y}_\ell(\xi_1),\dots,\mathbf{Y_\ell}(\xi_{k-1})\}^\perp$ in order to maximize its projection onto such orthogonal space. From this we may conclude that 
\begin{equation*}
    \begin{aligned}
             \mathbf{a}_\ell(k)=\mathrm{proj}_{\mathbf{Y_\ell}(\xi_k)^\perp}\circ\mathrm{proj}_{\{\mathbf{Y}_\ell(\xi_1),\dots,\mathbf{Y_\ell}(\xi_{k-1})\}^\perp}(\mathbf{a}_\ell)=\mathrm{proj}_{\{\mathbf{Y}_\ell(\xi_1),\dots,\mathbf{Y_\ell}(\xi_{k})\}^\perp}(\mathbf{a}_\ell)\,,
        \end{aligned}
\end{equation*}
where the last step follows from the fact that $\mathbf{Y}_\ell(\xi_k)\in\{\mathbf{Y}_\ell(\xi_1),\dots,\mathbf{Y_\ell}(\xi_{k-1})\}^\perp$ and hence the two orthogonal projection operators commute.

This shows, by induction, the orthogonality of $\{\mathbf{Y}_\ell(\xi_1),\dots,\mathbf{Y}_\ell(\xi_{2\ell+1})\}\subseteq \mbbC^{2\ell+1}$,  and hence linear independence. Moreover, from this it immediately follows that the algorithm stops in $K=2\ell+1$ steps. Hence $\mathbf{a}_\ell(2\ell+1)=\underline{0}$ and $T_{\ell}(\cdot,2\ell+1)\equiv 0$ since  $\{\mathbf{Y}_\ell(\xi_1),\dots,\mathbf{Y}_\ell(\xi_{2\ell+1})\}^\perp=\emptyset$. Finally, \eqref{eq:decomposition:sparse:algo} follows from \eqref{eq:approx:field:equiv} with $K=2\ell+1$.
\end{proof}

\begin{remark}
 In practice, the $\xi_k$ are chosen as the (ranked) most informative directions of the field, similarly to what happens when considering principal components in PCA for covariance matrices. 
\end{remark}

\begin{remark}
In the Gaussian case, any pair of multipoles $T_{\ell},T_{\ell '}$ is independent for $\ell \neq \ell '$. As a consequence, the random variables $\xi_{\ell,k}, \xi_{\ell ',k'}$ will also be independent, and their sample realization could be exploited as a test for isotropy and Gaussianity. The possibility that some multipole components $T_{\ell},T_{\ell '}$ may show some unexpected alignment has drawn a lot of attention and a heated debate in the literature on CMB data, see in particular \cite{Copi2006}, \cite{Oliveira2020}; in these papers, possible alignments were probed by means of so-called Maxwell's vectors (see \cite{Dennis2005JPhA}, \cite{Dennis2005JPhA}), so the construction we advocate here could be used to implement an alternative testing procedure. 
\end{remark}

Lastly, let us notice that the orthogonality of the vectors$\{\mathbf{Y}_\ell(\xi_1),\dots,\mathbf{Y}_\ell(\xi_{2\ell+1})\}\subseteq \mbbC^{2\ell+1}$  established in the proof of \Cref{thm:conv:algo:mono} combined with \Cref{remark:varianza:empirica:eta} immediately implies the following result.
\begin{corollary}\label{cor:varianza:empirica}
    Fix $ K= 2\ell +1$ and consider the output sparse field obtained by applying \Cref{algo:sparse:approx} (or \Cref{algo:sparse:approx:harmonic}) to one single realization of a monochromatic field $T_\ell$, that is the field with $\eta_{\ell k}=T_{\ell }(\xi_{k};k-1)$
    \begin{equation*}
        T_\ell=\sum_{k=1}^{2\ell+1}T_{\ell }(\xi_{k};k-1)P_{\ell
}(\langle \xi_{k},\cdot\rangle )=\sum_{k=1}^{2\ell+1}\langle \mathbf{a}_\ell(k-1),\,\mathbf{Y}_\ell(\xi_k)\rangle P_\ell(\langle \xi_k,\cdot\rangle)\,.
    \end{equation*}
    Then the empirical power spectrum equals
    \begin{equation*}
        \widehat{C}_\ell=\frac{1}{2\ell+1}\sum_{k=1}^{2\ell+1}\eta_{\ell k}^2\,|\mathbf{Y}_\ell(\xi_k)|^2=(4\pi)^{-1}\sum_{k=1}^{2\ell+1}\eta_{\ell k}^2\,.
    \end{equation*}
\end{corollary}

\subsection{Sparse reconstruction algorithm for polychromatic fields}\label{sec:proof:algo:colorati}
In this last section we suggest a possible generalization of our algorithm for polychromatic isotropic random fields. In what follows we fix a spherical random field $T$ with power spectrum $(C_\ell)_{\ell\geq 0}$ and sample empirical power spectrum 
\begin{equation*}
   \widehat{C}_\ell=\frac{1}{2\ell+1}\sum_{m=-\ell}^\ell \widehat{a}^2_{\ell m} =\frac{1}{2\ell+1}\sum_{m=-\ell}^\ell \langle T, Y_{\ell m}\rangle_{\mrmL^2(\mbbS^2)}^2\,,
\end{equation*}
which we consider here being a random variable depending on the realization of the original random field $T(x)=T(\omega;x)$.

A priori we further assume the original field $T$ to admit a decomposition as in \Cref{thm:superposition:representation} with positive weights. A discussion on this last assumption is postponed to the end of the current section.

Owing to the computations portrayed in \Cref{remark:varianza:empirica:eta} and \Cref{cor:varianza:empirica}, we know that any sparse approximation of the original field with $K=1$ must satisfy 
$\eta_\ell^2=4\pi \widehat{C}_\ell$. Therefore, in this single-wave first approximation a good weight candidate is given by the empirical spectrum, \ie, we can consider the random weights  $\widehat{\eta}_{\ell 1}\coloneqq \sqrt{4\pi \widehat{C}_\ell}$, which guaranties our variance constraint. Therefore, we can consider as a first approximation the random field
\begin{equation*}
    G(x,\xi)\coloneqq \sum_{\ell\geq 0}\widehat{\eta}_{\ell 1} \frac{2\ell+1}{4\pi} P_\ell(\langle x,\xi\rangle)\,, 
\end{equation*}
where the parameter $\xi\in\mbbS^2$ is chosen in order to maximize the projection of the original field onto $G(x,\xi)$. More precisely, we seek
\begin{equation*}
    \xi_1\in \argmax_{\xi\in\mbbS^2}\frac{|\langle G(\cdot,\xi),\,T(\cdot)\rangle_{\mrmL^2(\mbbS^2)}|}{\|T\|_{\mrmL^2(\mbbS^2)}^2}=\argmax_{\xi\in\mbbS^2}\frac{|\langle G(\cdot,\xi),\,T(\cdot)\rangle_{\mrmL^2(\mbbS^2)}|}{\sum_{\ell\geq 0}(2\ell+1)\widehat{C}_\ell}\,,
\end{equation*}
and hence a direction which maximizes the explanatory power of our approximation. Since 
\begin{equation*}
\begin{aligned}
   \langle G(\cdot,\xi),\,T(\cdot)\rangle_{\mrmL^2(\mbbS^2)}=&\,\sum_{\ell\geq 0}\sum_{m=-\ell}^\ell\int \widehat{\eta}_{\ell 1}\frac{2\ell+1}{4\pi}P_\ell(\langle x,\xi\rangle)\widehat{a}_{\ell m} Y_{\ell m}(x)\De \vol\\
   =&\,\sum_{\ell\geq 0}\sum_{m=-\ell}^\ell \widehat{\eta}_{\ell 1} \widehat{a}_{\ell m} Y_{\ell m}(\xi)=\sum_{\ell\geq 0}\sqrt{4\pi \widehat{C}_\ell} \,T_{\ell}(\xi)\,,
\end{aligned}
\end{equation*}
we can equivalently seek for $\xi_1\in\mbbS^2$ maximizing the projection index 
\begin{equation*}
    R (\xi)\coloneqq\sqrt{4\pi}\frac{|\sum_{\ell\geq 0}\sqrt{\widehat{C}_\ell} \,T_{\ell}(\xi)|}{\sum_{\ell\geq 0}(2\ell+1)\widehat{C}_\ell}\,.
\end{equation*}

\medskip

To summarize, so far we have defined the sparse field with $K=1$
\begin{equation*}
    T(\cdot\,;1)\coloneqq \sum_{\ell\geq 0}\widehat{\eta}_{\ell 1} \frac{2\ell+1}{4\pi} P_\ell(\langle \cdot,\xi_1\rangle)\,,
\end{equation*}
and $R(\xi_1)=\frac{|\langle T(\cdot\,;1),\,T\rangle_{\mrmL^2(\mbbS^2)}|}{\|T\|^2_{\mrmL^2(\mbbS^2)}}$ measures its expressiveness. More precisely, if for a fixed small tolerance $\varepsilon>0$ we have $R(\xi_1)>1-\varepsilon$ then $T(\cdot\,;1)$ faithfully represents the original field. 

On the contrary, if $R(\xi_1)\leq 1-\varepsilon$ we can consider the residual field
$T^R(\cdot)=T(\cdot)-R(\xi_1)T(\cdot\,;1)$ and iterate the previous steps for this residual random field. This would eventually generate a new sequence of weights $\widehat{\eta}_{\ell 2}$ and a direction $\xi_2\in\mbbS^2$ together with a projection index $R(\xi_2)$. The latter will  measure how far the residual field would be from this new single wave approximation; hence, it will allow to control whether the procedure should be halted or whether it should proceed with a new iteration.
\medskip

We summarize the above discussion in the following algorithm; here, we take a fixed parameter $K$ as we did in  \Cref{algo:sparse:approx} and \Cref{algo:sparse:approx:harmonic}.

\medskip

   \begin{algorithm}[H]
\SetAlgoLined
\KwIn{Realization of a random field T with empirical power spectrum $\widehat{C}_\ell$; 
Tolerance parameter $\varepsilon>0$; Maximal sparsity parameter $K$}
Start counter $k=0$\\
Initialize sparse approximation component $S(\cdot\,;0)=0$ and output field $F(\cdot)=0$\\
Initialize projection index $R=0$\\
Initialize residual field $T(\cdot\,;0)=T(\cdot)$ and its empirical power spectrum $\widehat{C}_\ell(0)=\widehat{C}_\ell$\\
\While{$R\leq 1-\varepsilon$ and $k<K$}{
Update output field $F(\cdot)\leftarrow F+R\,S(\cdot;k)$
Update residual field $T(\cdot\,; k+1)=T(\cdot\,;k)-R \,S(\cdot\,;k)=T(\cdot)-F(\cdot)$\\
$k\leftarrow k+1$\\
Compute residual empirical power spectrum $\widehat{C}_\ell(k)$\\
Define $\hat{\eta}_{\ell k}=\sqrt{4\pi \widehat{C}_\ell(k)}$\\
Find $\xi_k=\argmax\sum_{\ell\geq 0}\sqrt{ \widehat{C}_\ell(k)}T_{\ell}(\cdot\,;k)$\\
Built new sparse component $S(\cdot\,;k)= \sum_{\ell\geq 0}\widehat{\eta}_{\ell k} \frac{2\ell+1}{4\pi} P_\ell(\langle \cdot,\xi_k\rangle)$\\
Update projection index $R=\sqrt{4\pi}\frac{\sum_{\ell\geq 0}\sqrt{ \widehat{C}_\ell(k)}T_{\ell}(\cdot\,;k)}{\sum_{\ell\geq 0}(2\ell+1)\widehat{C}_\ell(k)}$
 }
 Last update of output field $F(\cdot)\leftarrow F(\cdot)+S(\cdot\,;k)$\\
\KwOut{Return sparse field $F(\cdot)$  }
\caption{Sparse reconstruction algorithm for a polychromatic field}\label{algo:sparse:approx:colorato}
\end{algorithm}

\medskip

Lastly, let us briefly comment on the assumption that our original field $T$ admits a decomposition as in \Cref{thm:superposition:representation} with positive weights. This is a simplicity assumption, motivated by our algorithm design since at each iteration we define $\hat{\eta}_{\ell k}=\sqrt{4\pi \widehat{C}_\ell(k)}$ in order to enforce the validity of $|\widehat{\eta}_{\ell k}|^2=4\pi C_\ell(k)$ (see also the discussion in \Cref{remark:varianza:empirica:eta} and \Cref{cor:varianza:empirica}). However, this last condition is clearly met also for the choice $\hat{\eta}_{\ell k}=-\sqrt{4\pi \widehat{C}_\ell(k)}$. We hence believe that a suitable modification of the algorithm should be able to relax our current assumption; we plan to address the mathematical analysis and generalization of this polychromatic sparse algorithm in future works.

\section{Some Numerical Evidence}\label{sec:numerics}
In this section we present some numerical examples; indeed, we plot strongly sparse random fields which are meant to approximate three different Whittle--Matérn fields, \ie, the isotropic Gaussian random fields with power spectrum
\begin{equation*}
    C_\ell=(1+\ell)^{-2\beta}\,, \quad\text{ for }\beta\in\{1.01,\,1.5,\,2\}\,.
\end{equation*}

Actually for the exact Whittle--Matérn fields the angular power spectrum should take the form (up to constants) $C_\ell=(1+\ell(\ell+1))^{-\beta}$; our formulation entails only a negligible numerical approximation. More precisely, for each choice of $\beta$ we consider  random fields generated according to ~\eqref{eq:superposition:waves}, \ie, of the form
\begin{equation*}
T(x)=\sum_{\ell=0}^{L_{\max}}\sum_{k=1}^{K}\eta _{\ell k}\frac{2\ell +1}{4\pi }%
P_{\ell }(\left\langle \xi _{k},x\right\rangle )\text{ , where }\xi _{k}\sim 
\mathrm{Unif}(\mbbS^{2})\,,  
\end{equation*}%
with $K\ll L_{\max}$ and with $\{\eta_{\ell k}\}_{\ell k}$ independent of the random directions $\{\xi_k\}_{k=1,\dots,K}$. We have chosen an i.i.d. collection of random weights $\{\eta_{\ell k}\}_{\ell k}$ for which, owing to \Cref{remark:varianza:eta}, we know it must hold
\begin{equation*}
    \mbbE[\eta_{\ell k}^2]=\frac{4\pi}{K}\, C_\ell\,,\quad\forall \,k\in\{1,\dots,K\} \text{ and }\forall\,\ell\geq 0\,.
\end{equation*}
In particular, we consider random weights of the form 
\[\eta_{\ell k}\coloneqq  u_{\ell k}\,\sqrt{\frac{4\pi}{K}\, C_\ell}\,,\]
where $\{u_{\ell k}\}_{\ell k}$ is a sequence of unit-variance centered random variables. In our  simulations $\{u_{\ell k}\}_{\ell, k}$ we took either a sequence of i.i.d. standard Gaussians (left columns in \Cref{fig:101,fig:150,fig:200} ) or a sequence of i.i.d. centered Bernoulli random variables (taking values $\{\pm1\}$, right columns in \Cref{fig:101,fig:150,fig:200}).

\medskip

Our plots are based on the Python package Healpy, currently the standard for spherical random fields simulations in dimension 2; it is based on the HEALPix\footnote{\url{http://healpix.sourceforge.net}} C++ library \cite{Zonca2019, 2005ApJ...622..759G}. For the pixelization of the sphere we have considered $n_\mathrm{side}=64$. 
In the simulations reported in \Cref{fig:101,fig:150,fig:200} we have fixed $L_{\max}=128$ and considered $K=4,24,40$ directions. Therefore, the number of parameters required for our strongly-sparse simulations is bounded by $K+K\times (L_{\max}+1)\leq 5200$, far smaller than the $(L_{\max}+1)^2=16641$ parameters (the harmonic coefficients) required for the generation of Gaussian random fields with the same resolution $L_{\max}$. 
Heuristically, \Cref{fig:101,fig:150,fig:200} shows that even the choice $K=24$ guaranties a good visual approximation of the original Gaussian random field, though of course more formal testing procedures will be able to detect significant differences. Indeed, these simulations could also provide a quick procedure to generate nonGaussian isotropic fields, to be used as testing benchmark for nonGaussianity detection procedures. 

\medskip

Finally, we repeated the numerical experiment with the same values of $K$, but higher resolution maps (i.e., $L_{\max}=256$ and $n_{\mathrm{side}}=128$). \Cref{fig:256_101,fig:256_150} show that the parameter $K$ does not need to grow above the values we considered previously to achieve very good visual approximations; note that these maps are generated by means of $K+K \times (L_{\max}+1)=10.320 $, random variables, rather than the $(L_{\max}+1)^2=66.049$ required for Gaussian maps.

\bigskip

\begin{figure}[h!]
    \centering
    \begin{subfigure}{0.45\textwidth}
        \includegraphics[width=\linewidth]{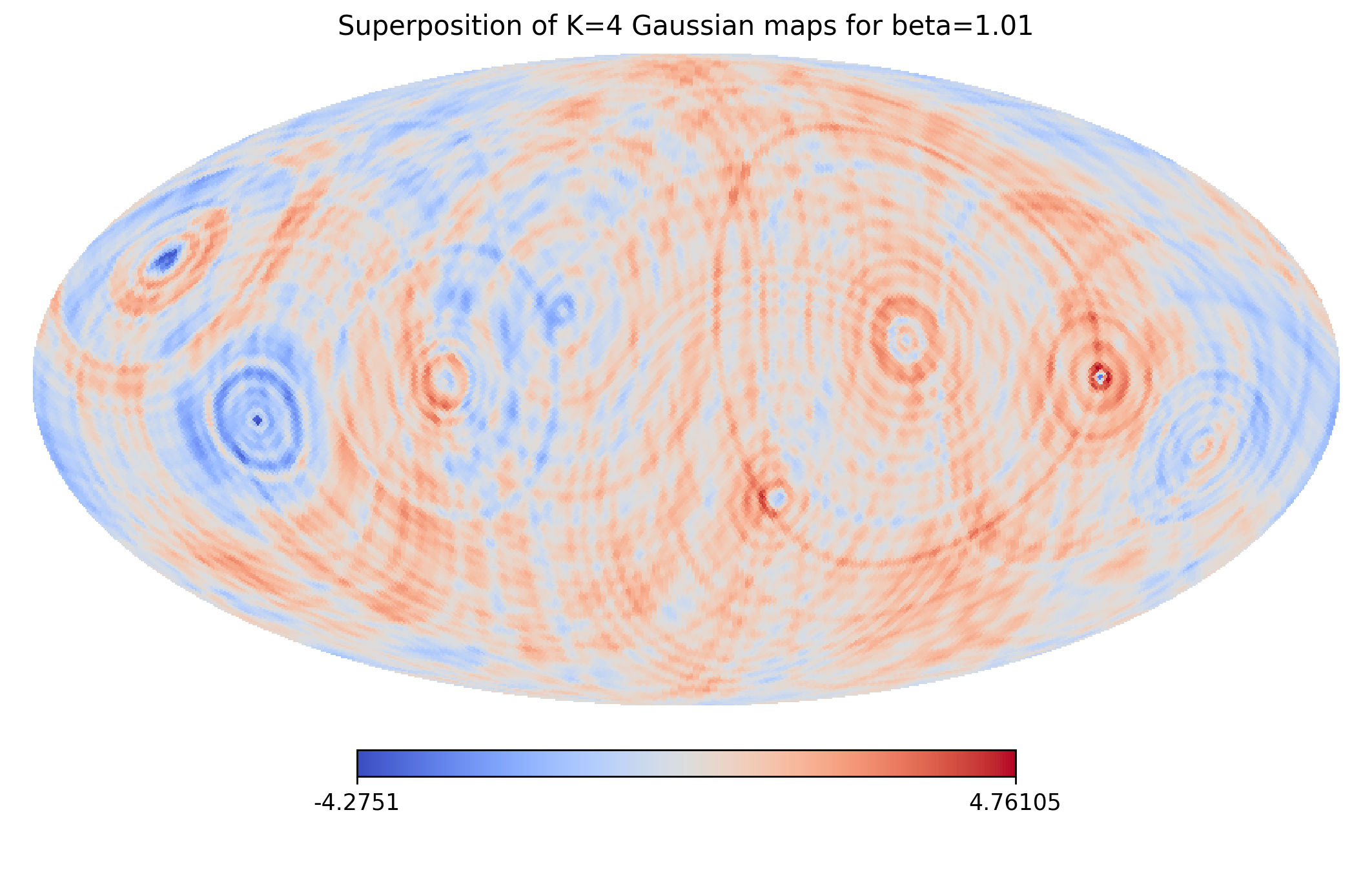}
    \end{subfigure}
        \hfill
    \begin{subfigure}{0.45\textwidth}
        \includegraphics[width=\linewidth]{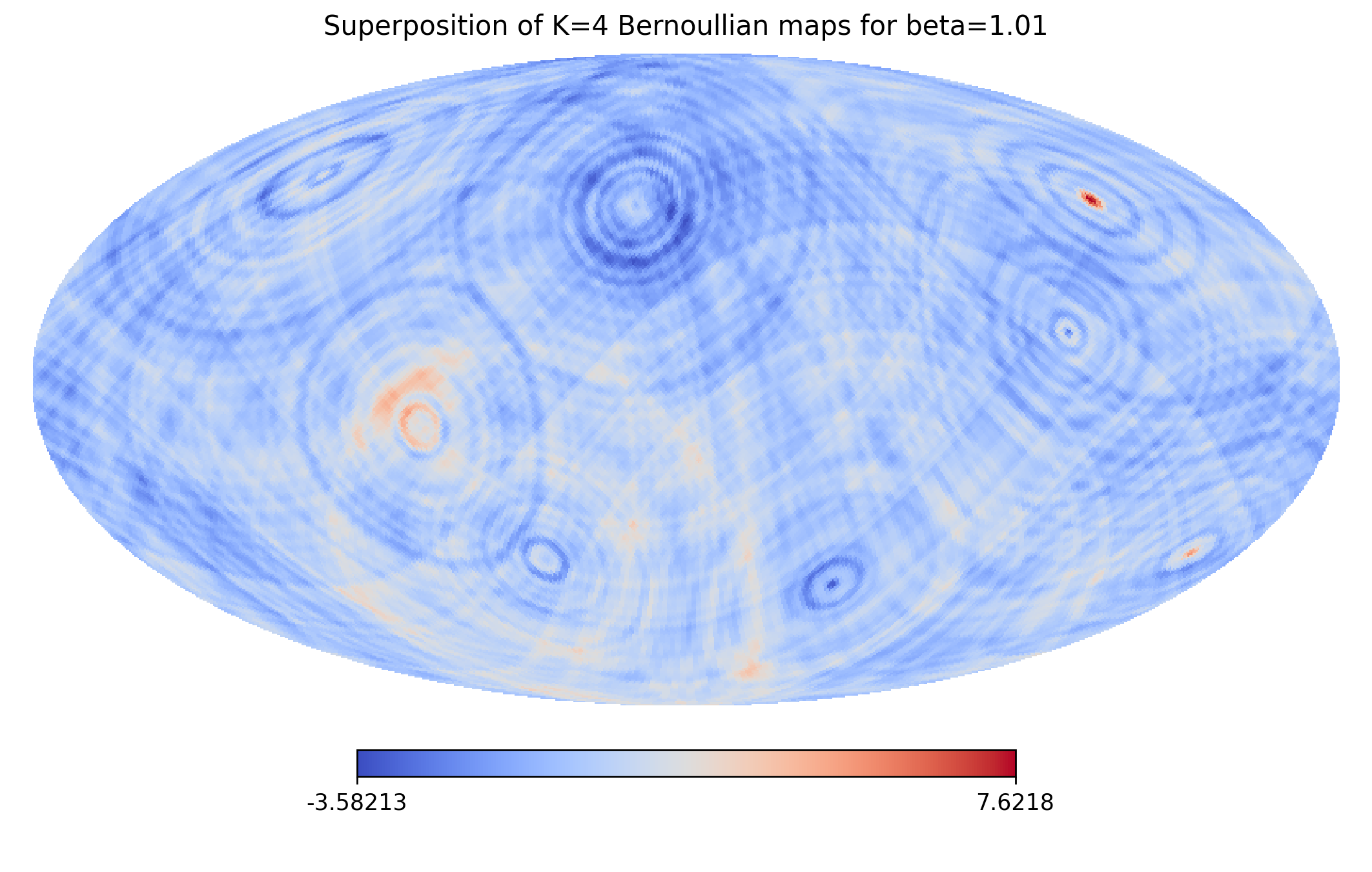}
    \end{subfigure}
        
     \begin{subfigure}{0.45\textwidth}
        \includegraphics[width=\linewidth]{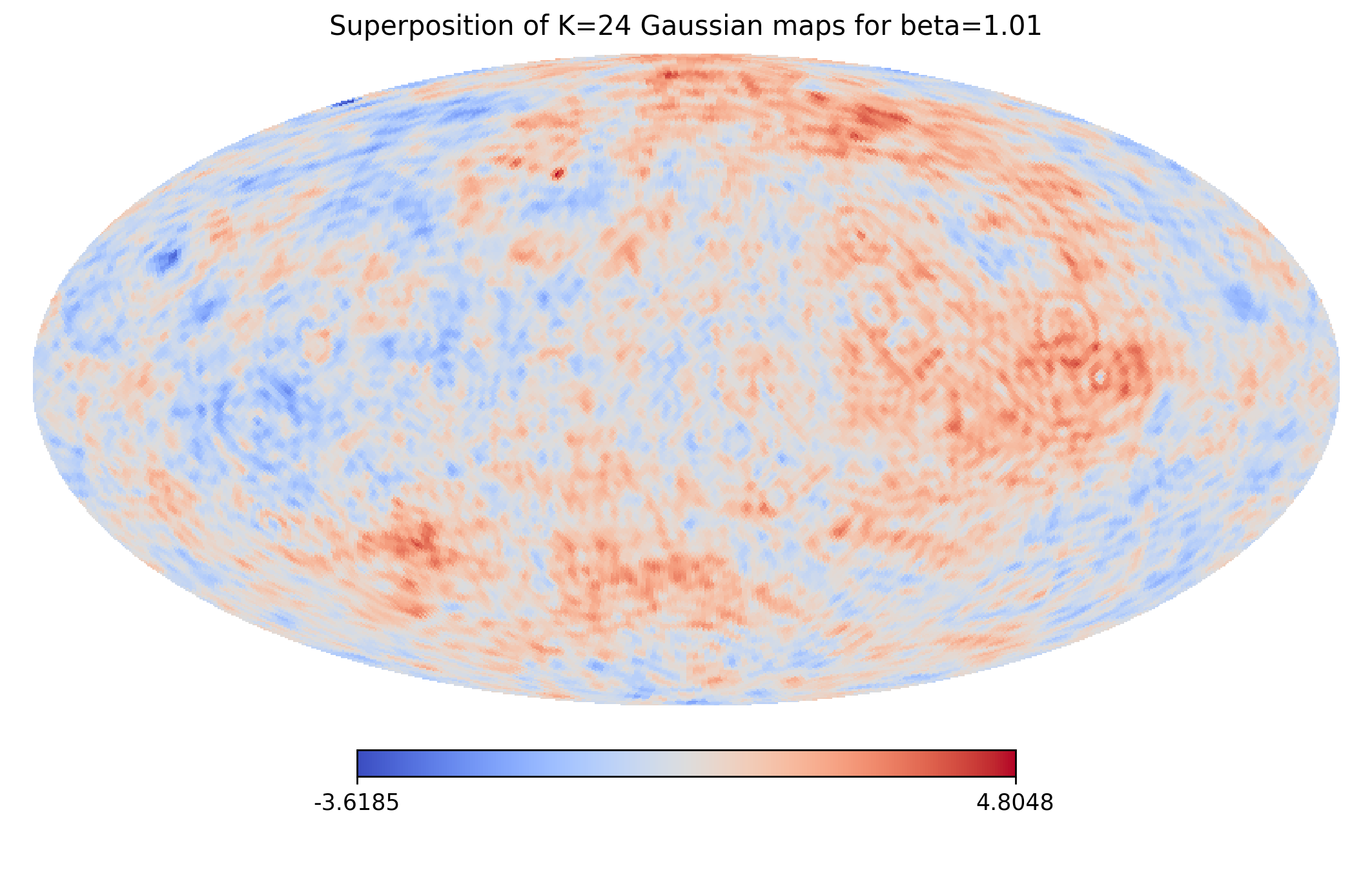}
    \end{subfigure}
    \hfill
    \begin{subfigure}{0.45\textwidth}
        \includegraphics[width=\linewidth]{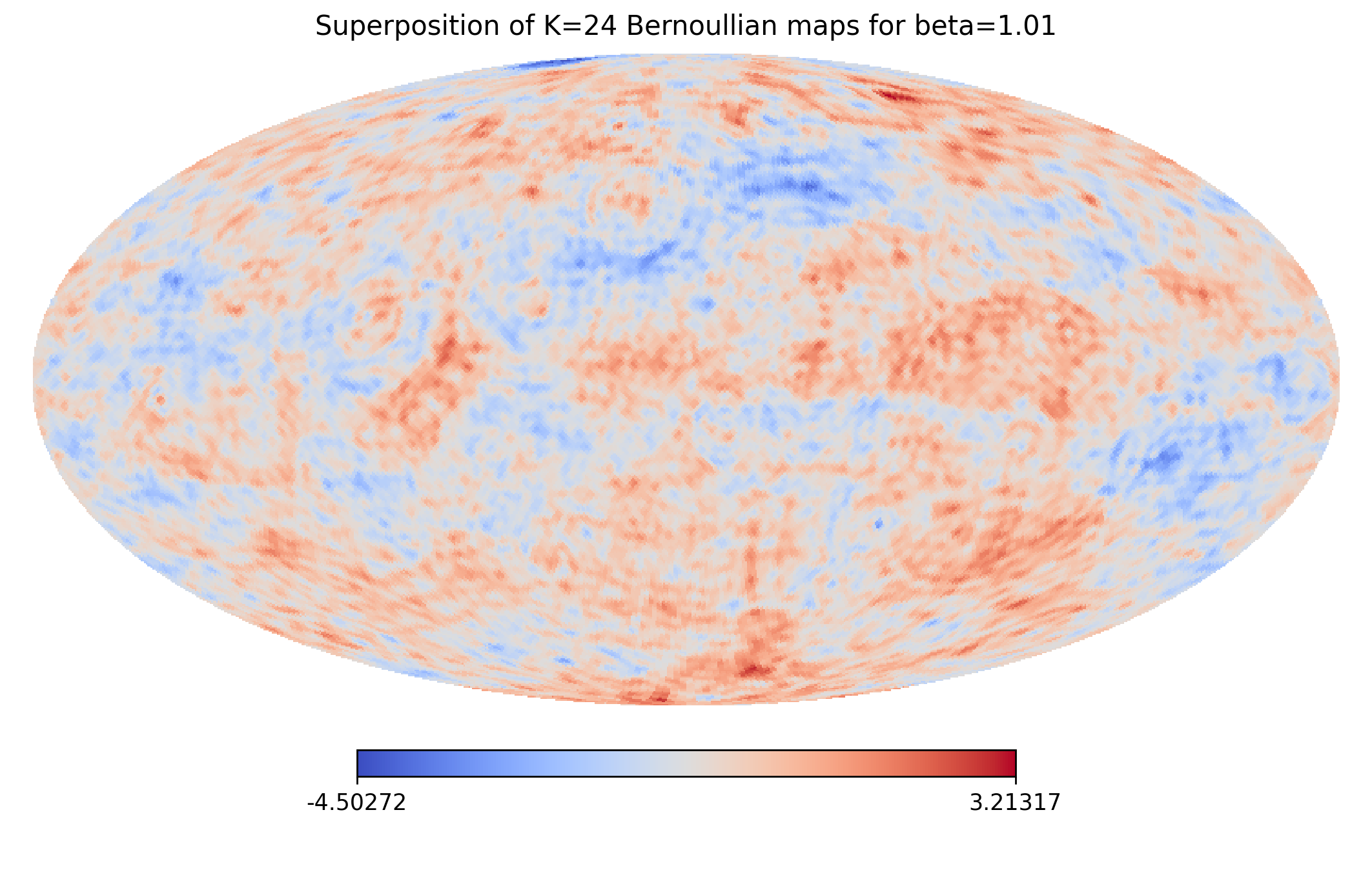}
    \end{subfigure}

     \begin{subfigure}{0.45\textwidth}
       \includegraphics[width=\linewidth]{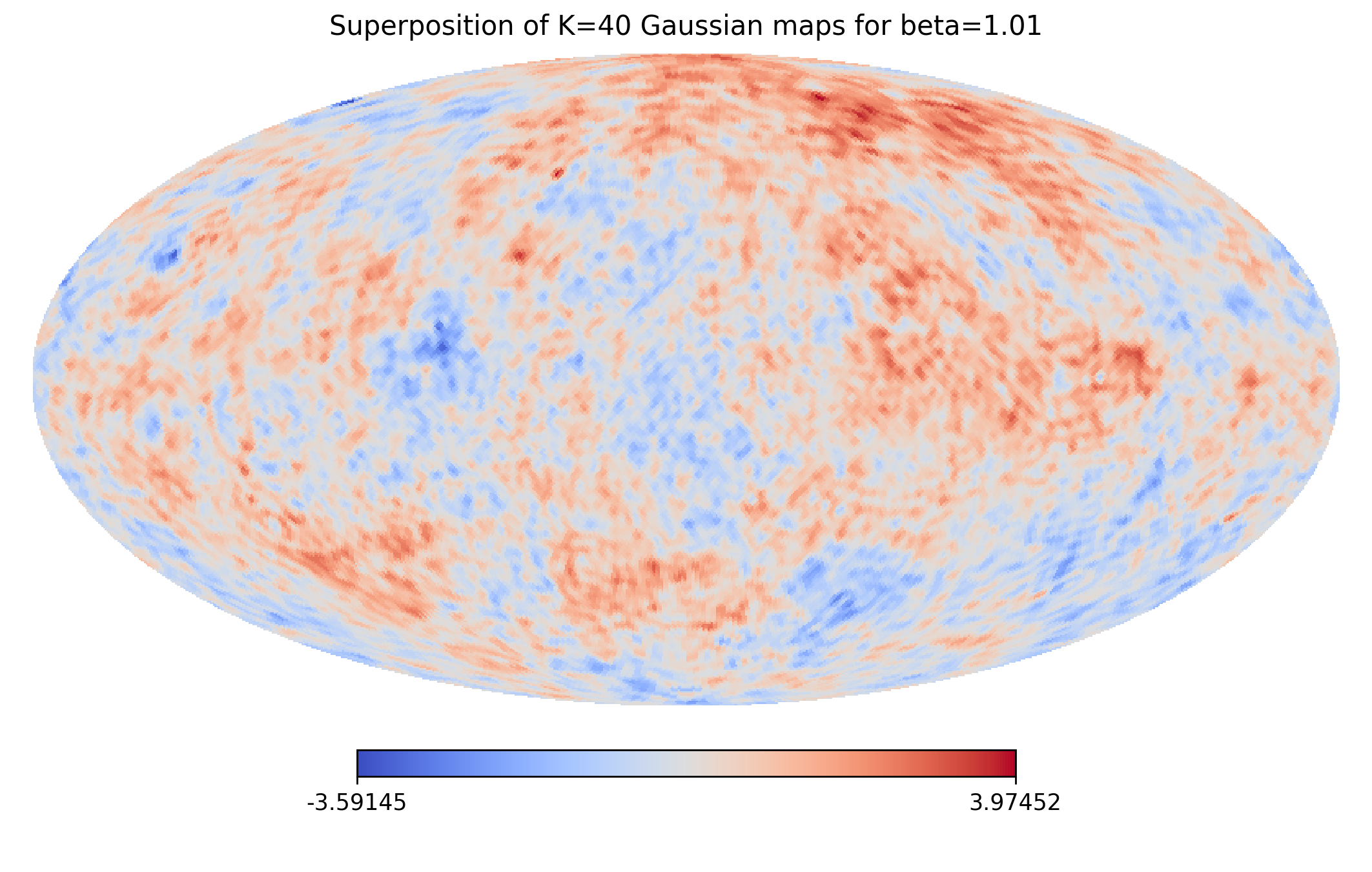}
    \end{subfigure}
    \hfill
    \begin{subfigure}{0.45\textwidth}
        \includegraphics[width=\linewidth]{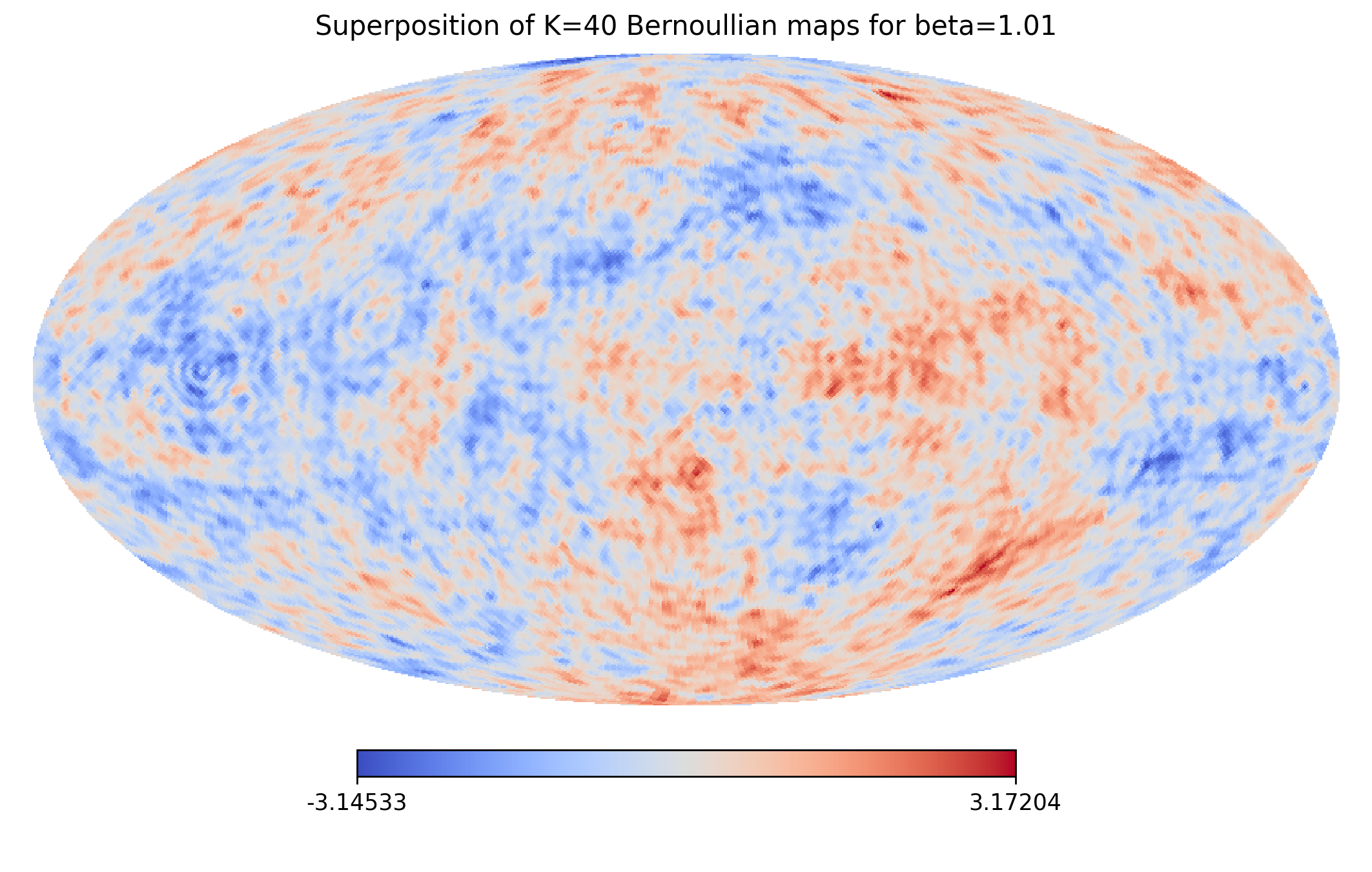}
    \end{subfigure}

     \begin{subfigure}{0.7\textwidth}
       \includegraphics[width=\linewidth]{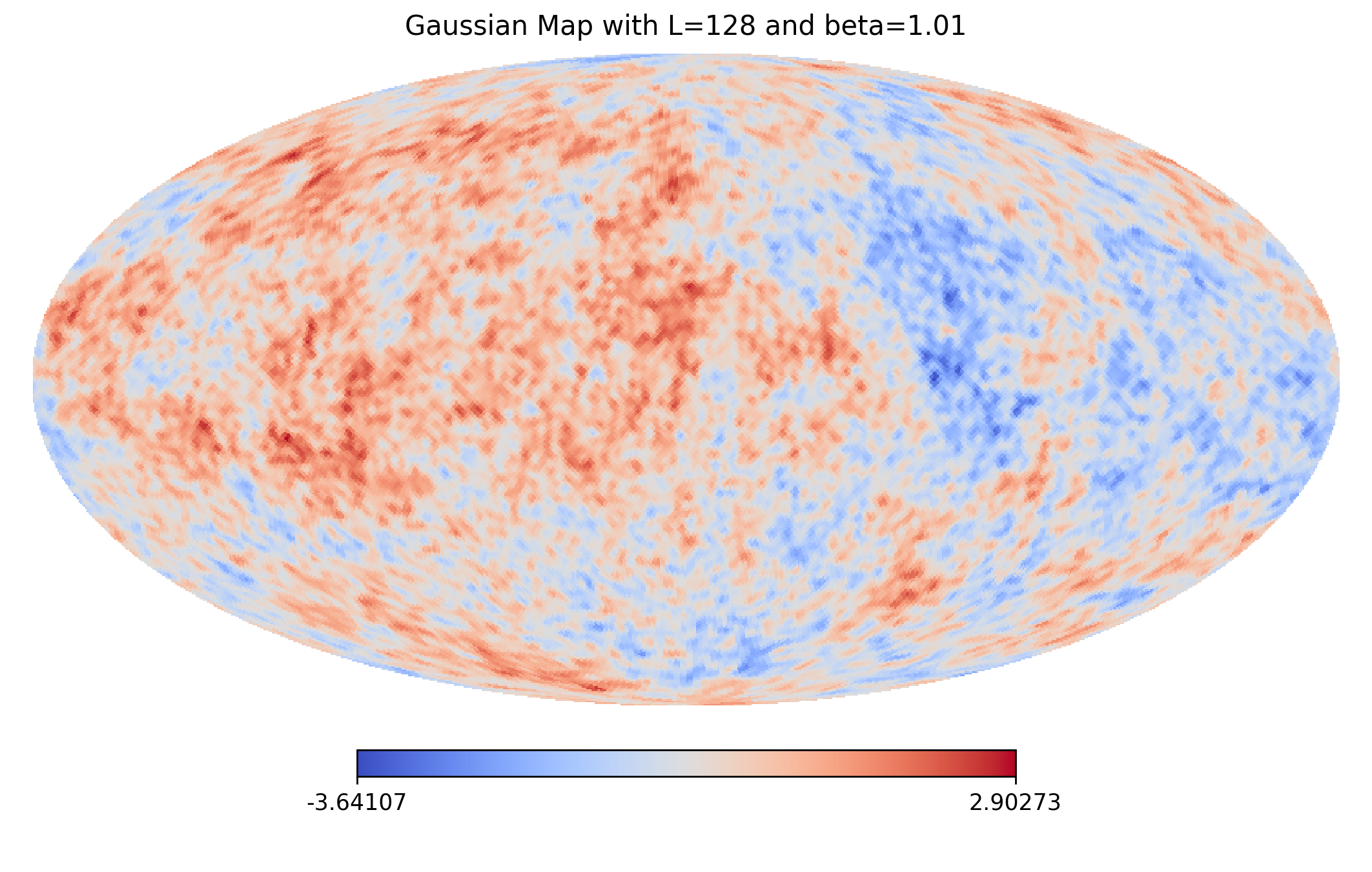}
    \end{subfigure}
    \hfill
    \caption{Simulation of a Whittle-Matérn field with $\beta=1.01$ (bottom image) with resolution $L_{\max}=128$ and $n_{\mathrm{side}}=64$. In the first three rows, we have plotted the sparse random fields written as superposition of $K=4,\,24,\,40$ random waves. On the left column, the random weights $\{\eta_{\ell k}\}$ are taken to be Gaussian random variables with variance $4\pi C_\ell/K$. On the right column, the random weights $\{\eta_{\ell k}\}$ are centered Bernoulli random variables normalized with $4\pi C_\ell/K$.}
    \label{fig:101}
\end{figure}

\begin{figure}[h!]
    \centering
    \begin{subfigure}{0.45\textwidth}
        \includegraphics[width=\linewidth]{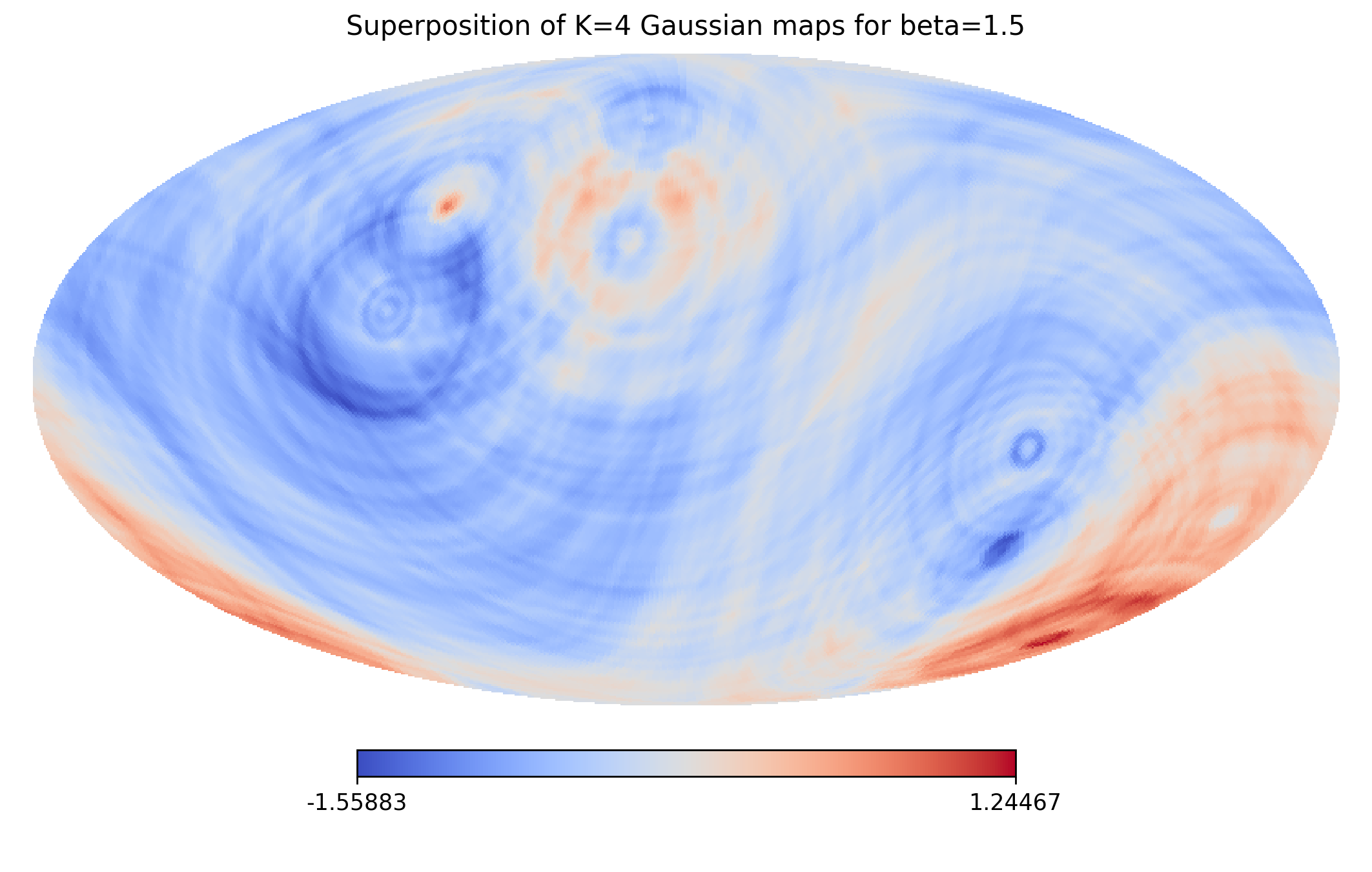}
    \end{subfigure}
        \hfill
    \begin{subfigure}{0.45\textwidth}
        \includegraphics[width=\linewidth]{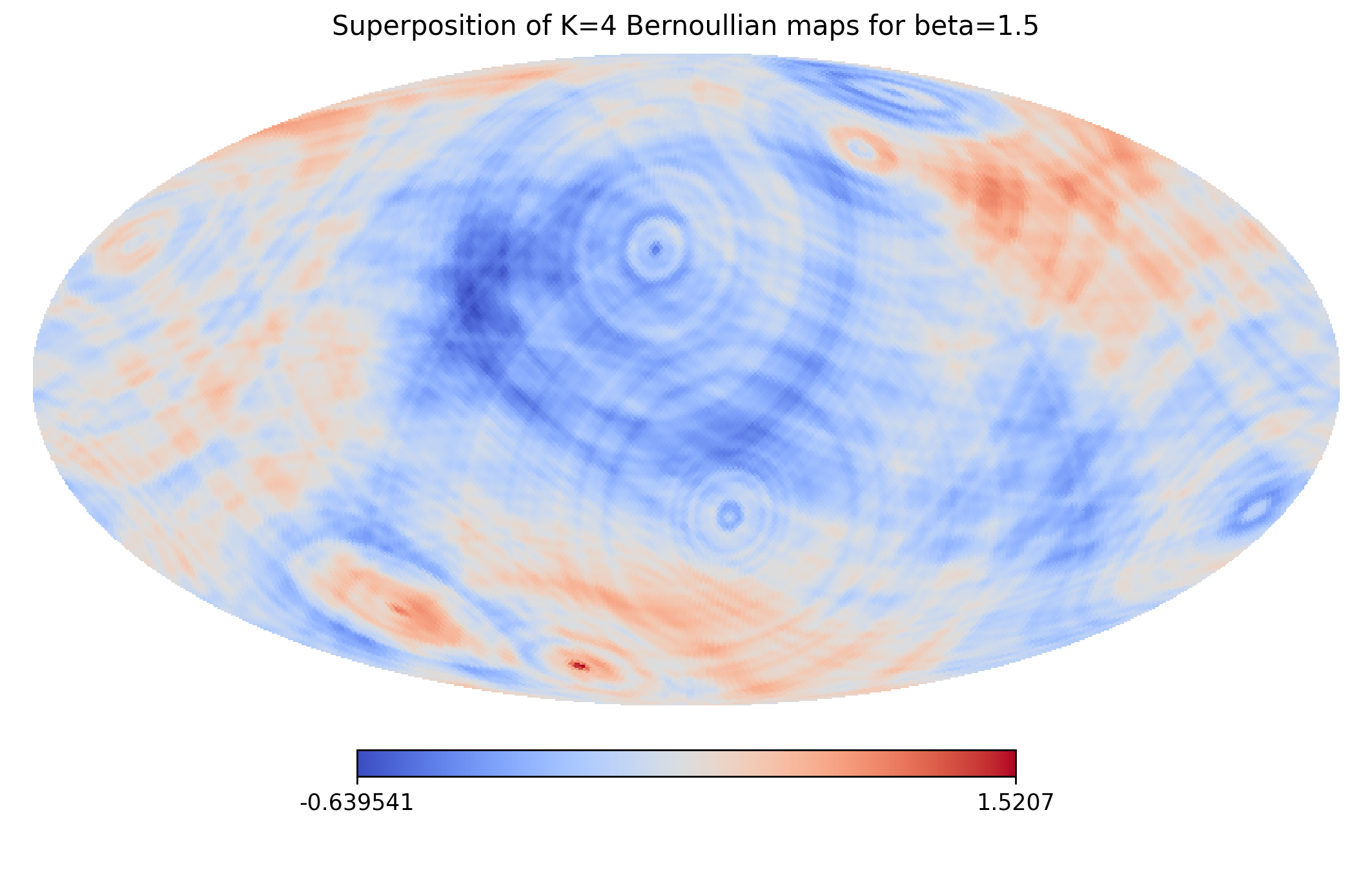}
    \end{subfigure}
        
     \begin{subfigure}{0.45\textwidth}
        \includegraphics[width=\linewidth]{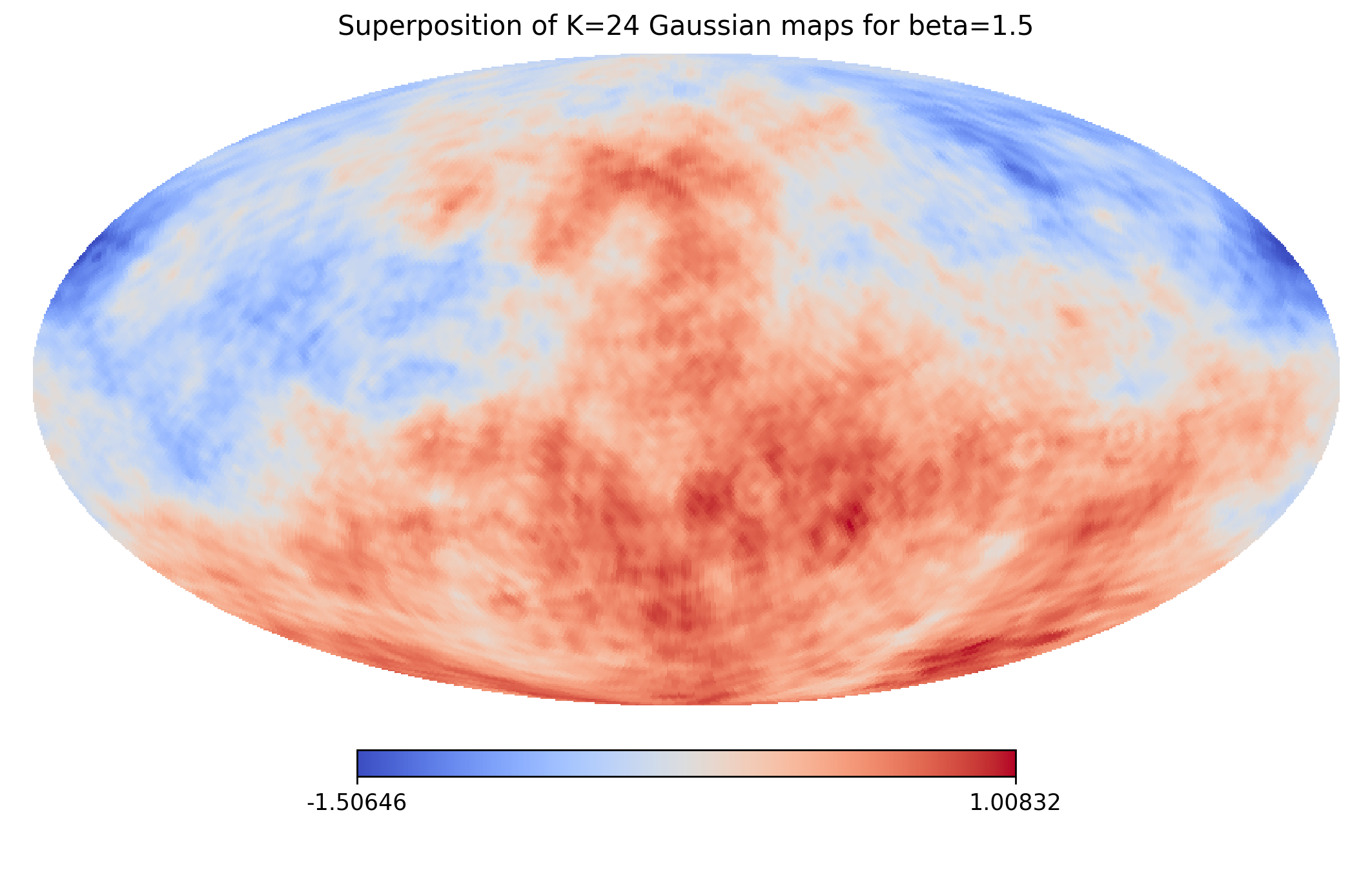}
    \end{subfigure}
    \hfill
    \begin{subfigure}{0.45\textwidth}
        \includegraphics[width=\linewidth]{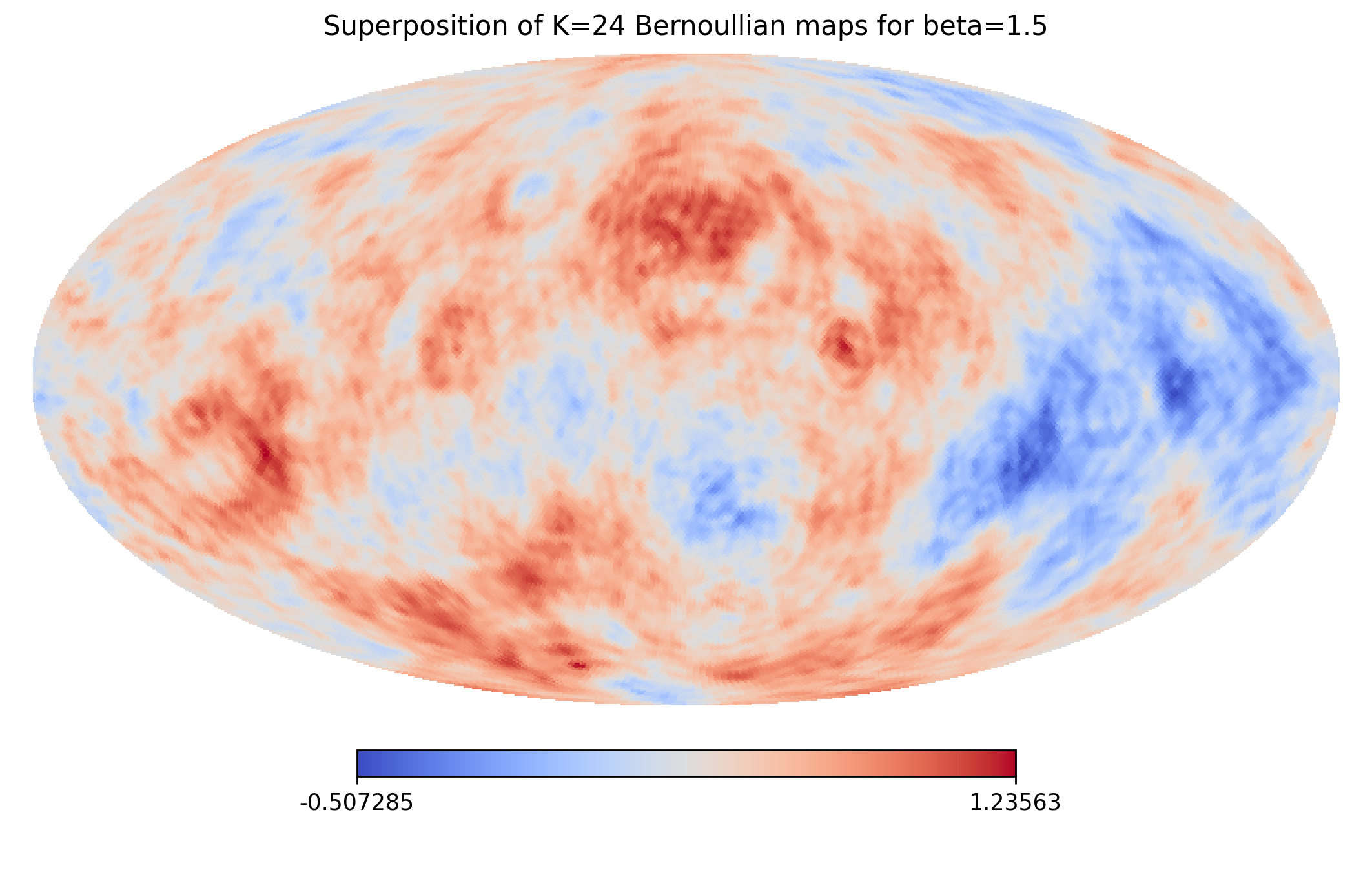}
    \end{subfigure}

     \begin{subfigure}{0.45\textwidth}
       \includegraphics[width=\linewidth]{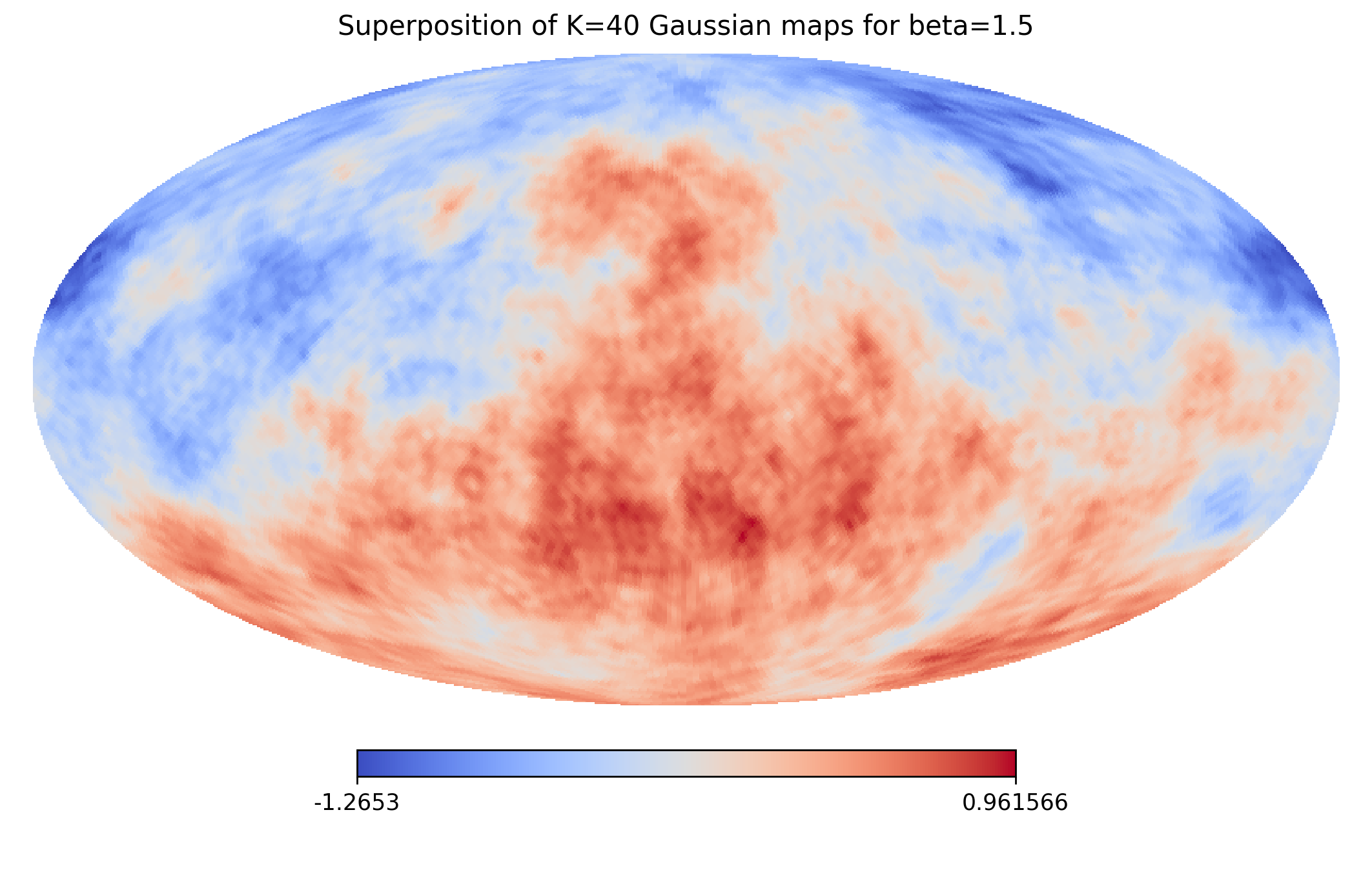}
    \end{subfigure}
    \hfill
    \begin{subfigure}{0.45\textwidth}
        \includegraphics[width=\linewidth]{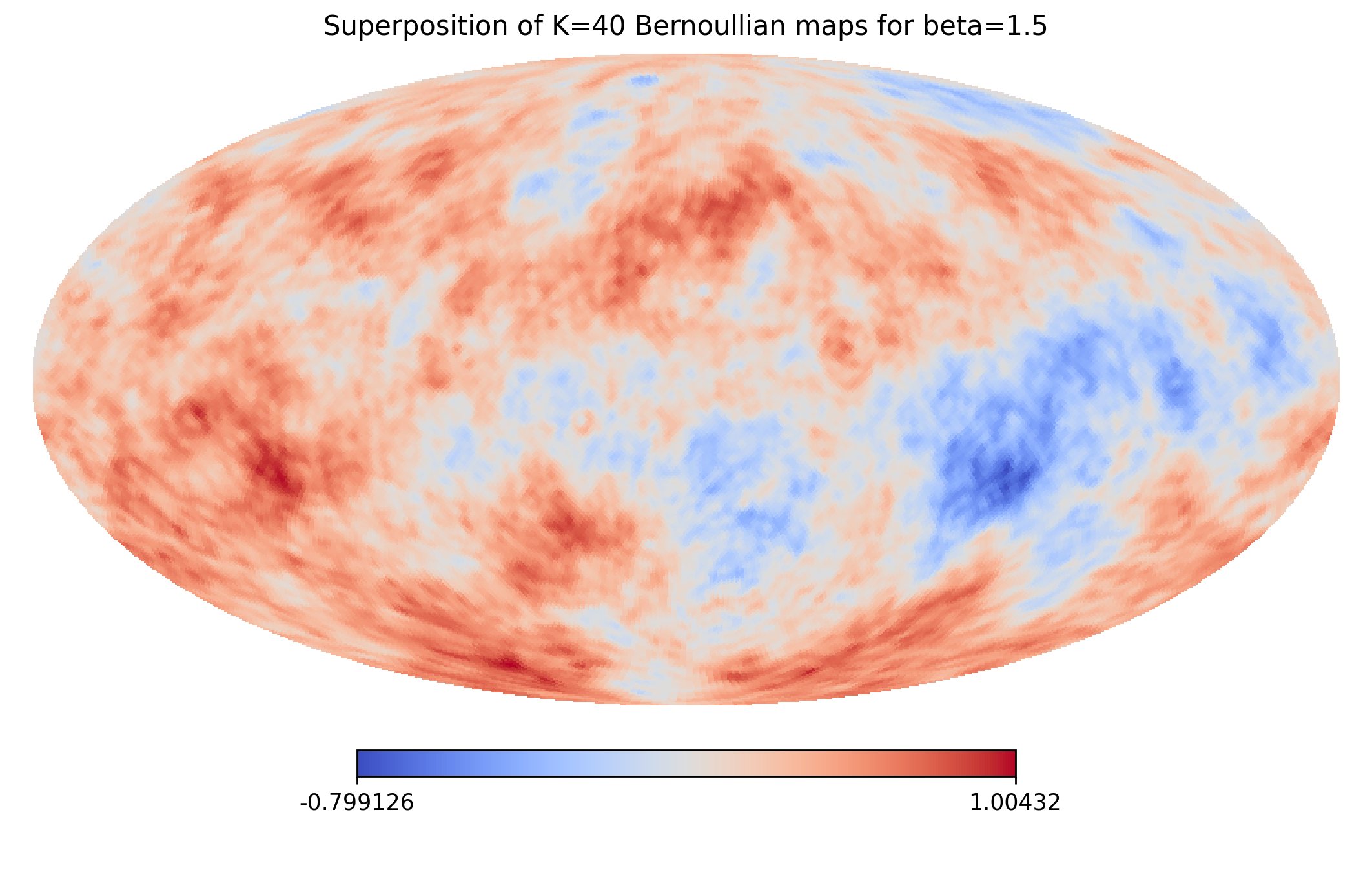}
    \end{subfigure}

     \begin{subfigure}{0.7\textwidth}
       \includegraphics[width=\linewidth]{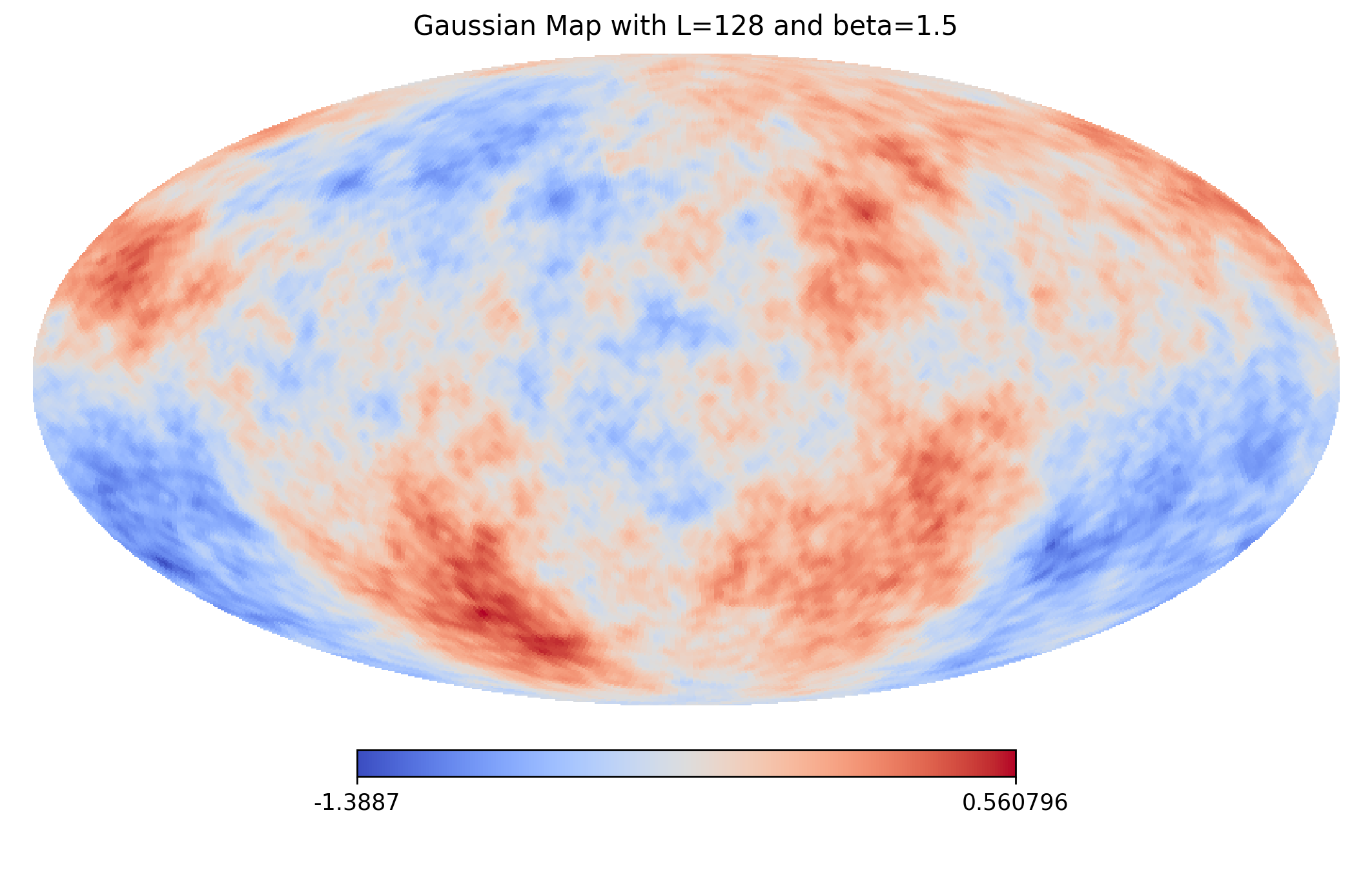}
    \end{subfigure}
    \hfill

     \caption{Simulation of a Whittle-Matérn field with $\beta=1.5$ (bottom image) with resolution $L_{\max}=128$ and $n_{\mathrm{side}}=64$. In the first three rows, we have plotted the sparse random fields written as superposition of $K=4,\,24,\,40$ random waves. On the left column, the random weights $\{\eta_{\ell k}\}$ are taken to be Gaussian random variables with variance $4\pi C_\ell/K$. On the right column, the random weights $\{\eta_{\ell k}\}$ are centered Bernoulli random variables normalized with $4\pi C_\ell/K$.}
    \label{fig:150}
\end{figure}

\begin{figure}[h!]
    \centering
    \begin{subfigure}{0.45\textwidth}
        \includegraphics[width=\linewidth]{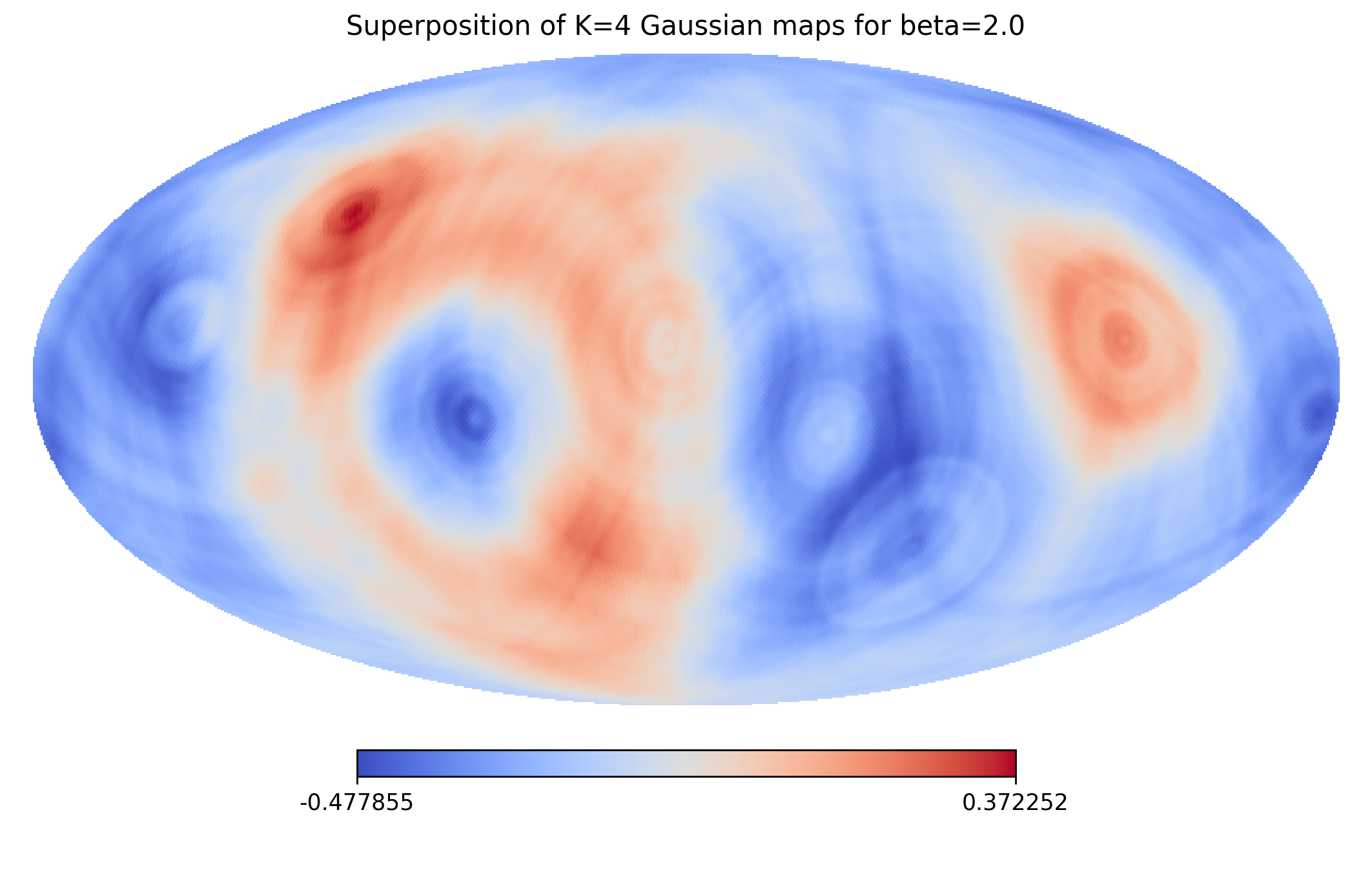}
    \end{subfigure}
        \hfill
    \begin{subfigure}{0.45\textwidth}
        \includegraphics[width=\linewidth]{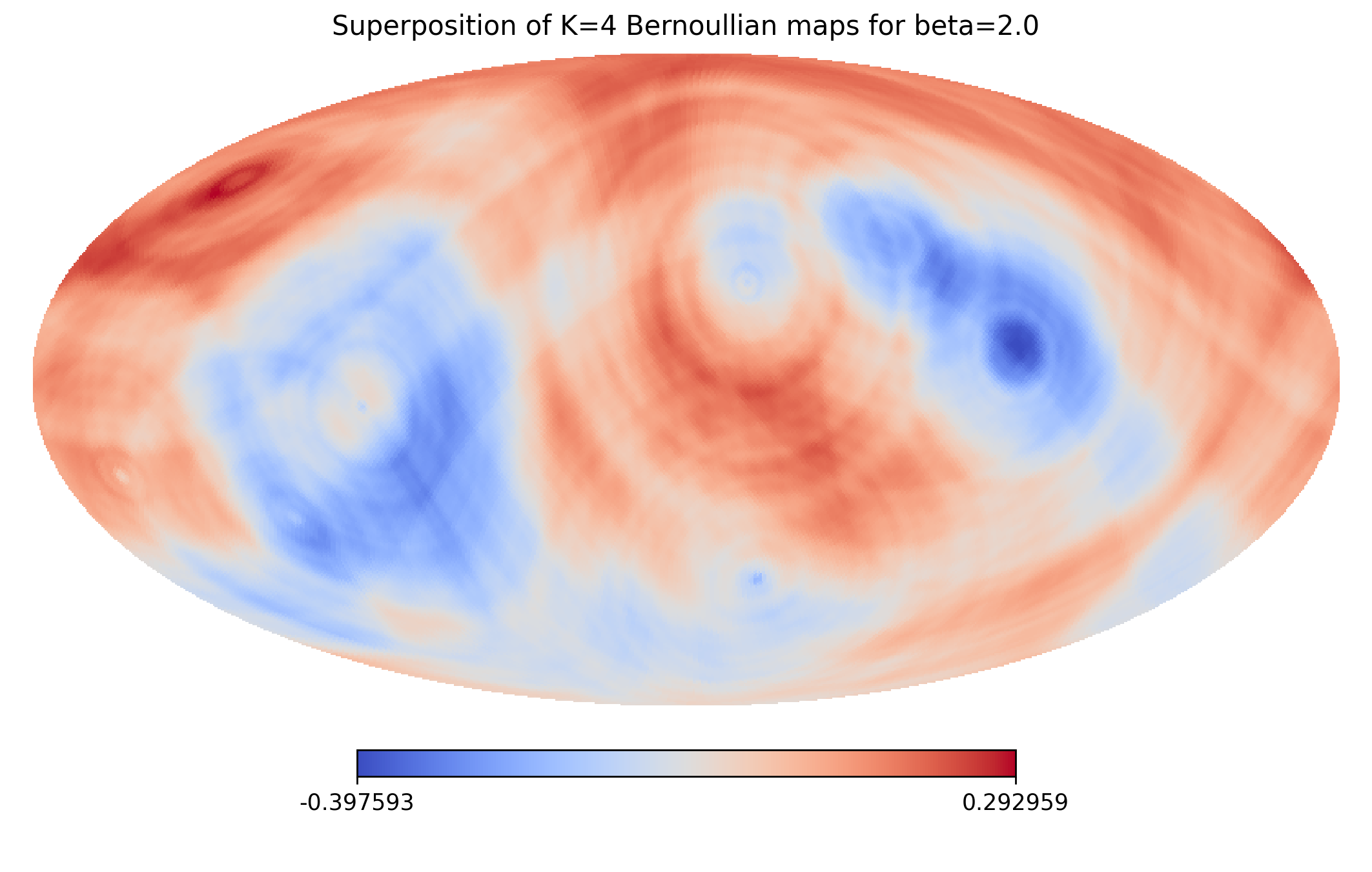}
    \end{subfigure}
        
     \begin{subfigure}{0.45\textwidth}
        \includegraphics[width=\linewidth]{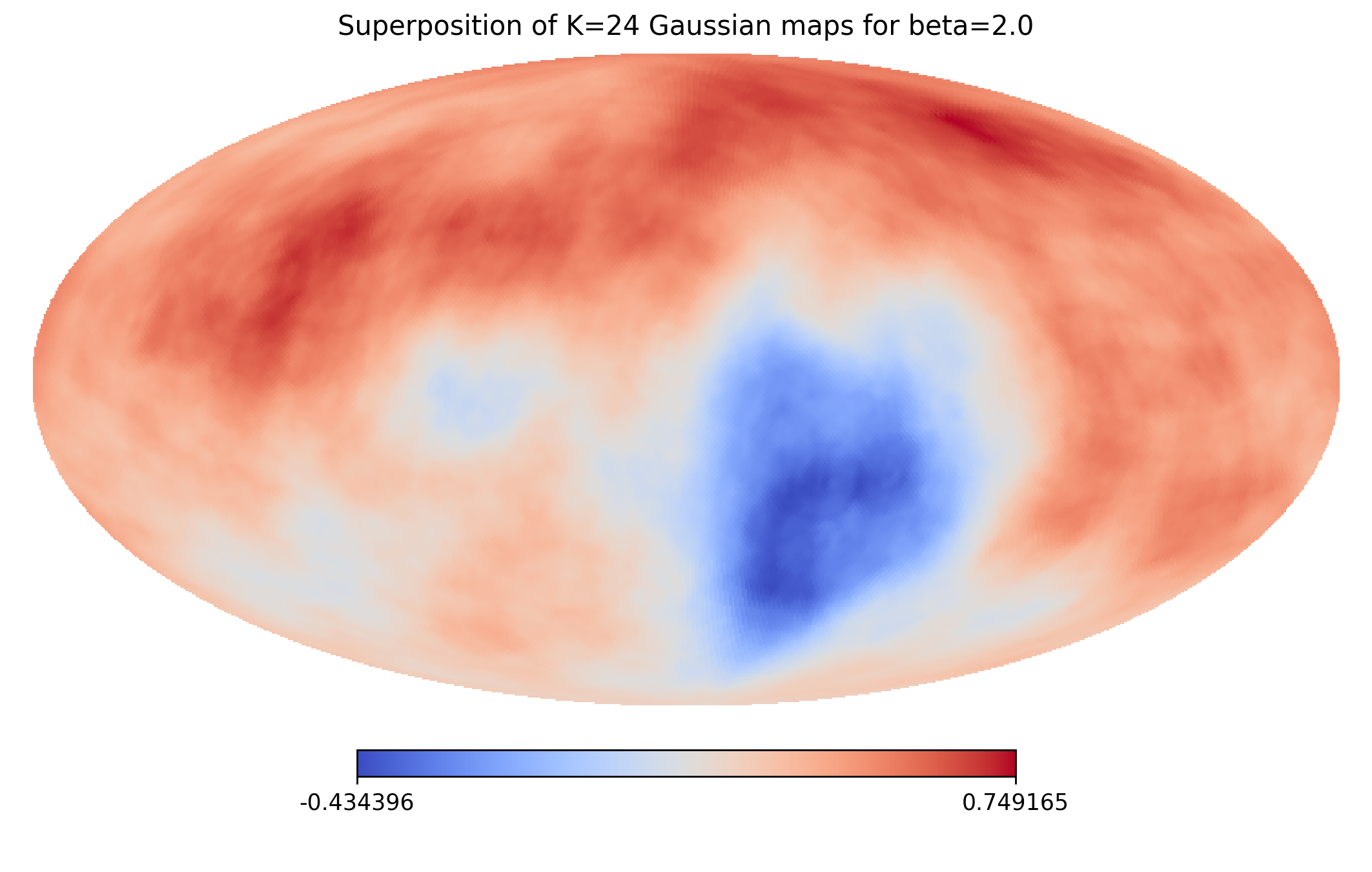}
    \end{subfigure}
    \hfill
    \begin{subfigure}{0.45\textwidth}
        \includegraphics[width=\linewidth]{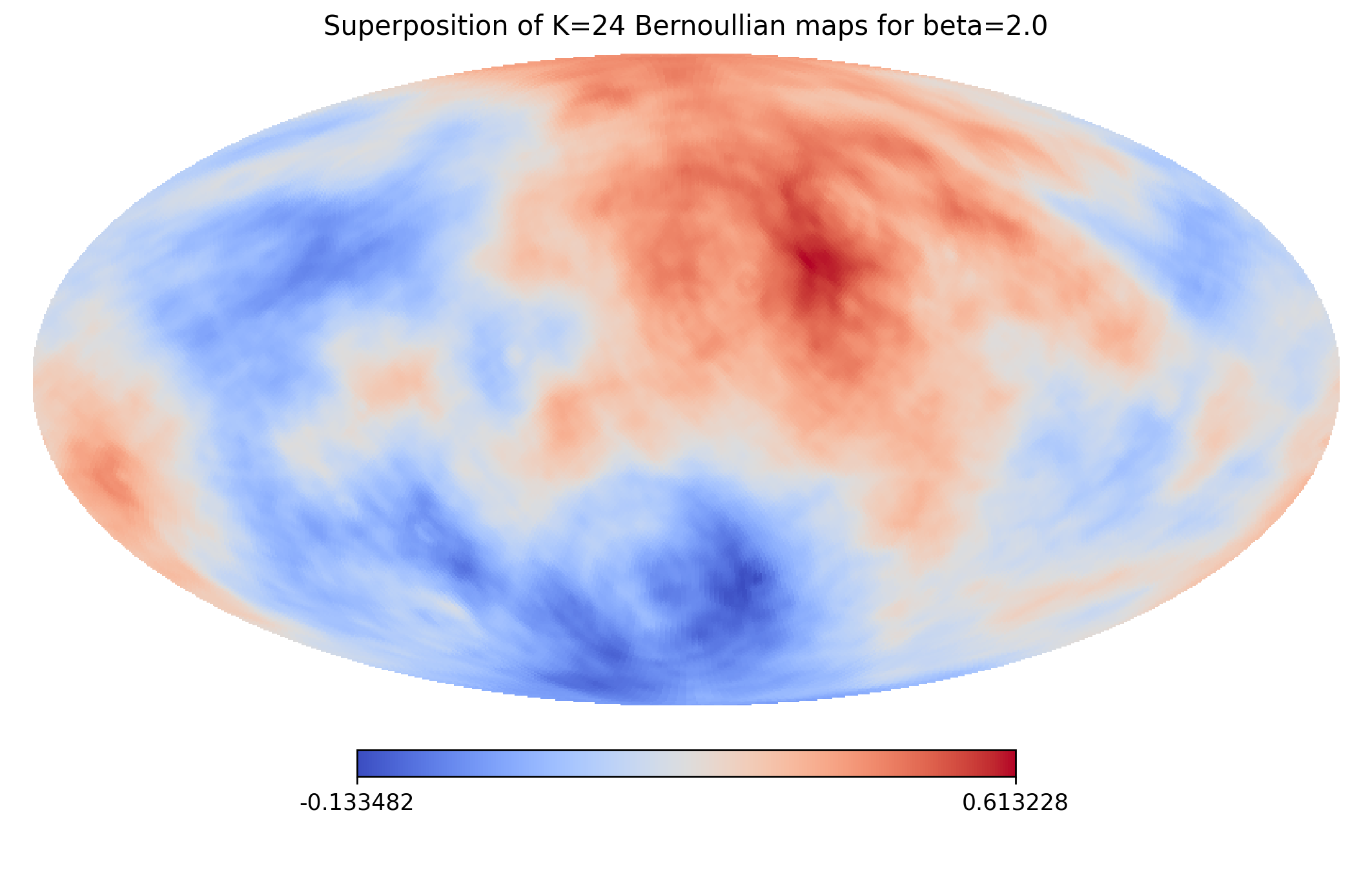}
    \end{subfigure}

     \begin{subfigure}{0.45\textwidth}
       \includegraphics[width=\linewidth]{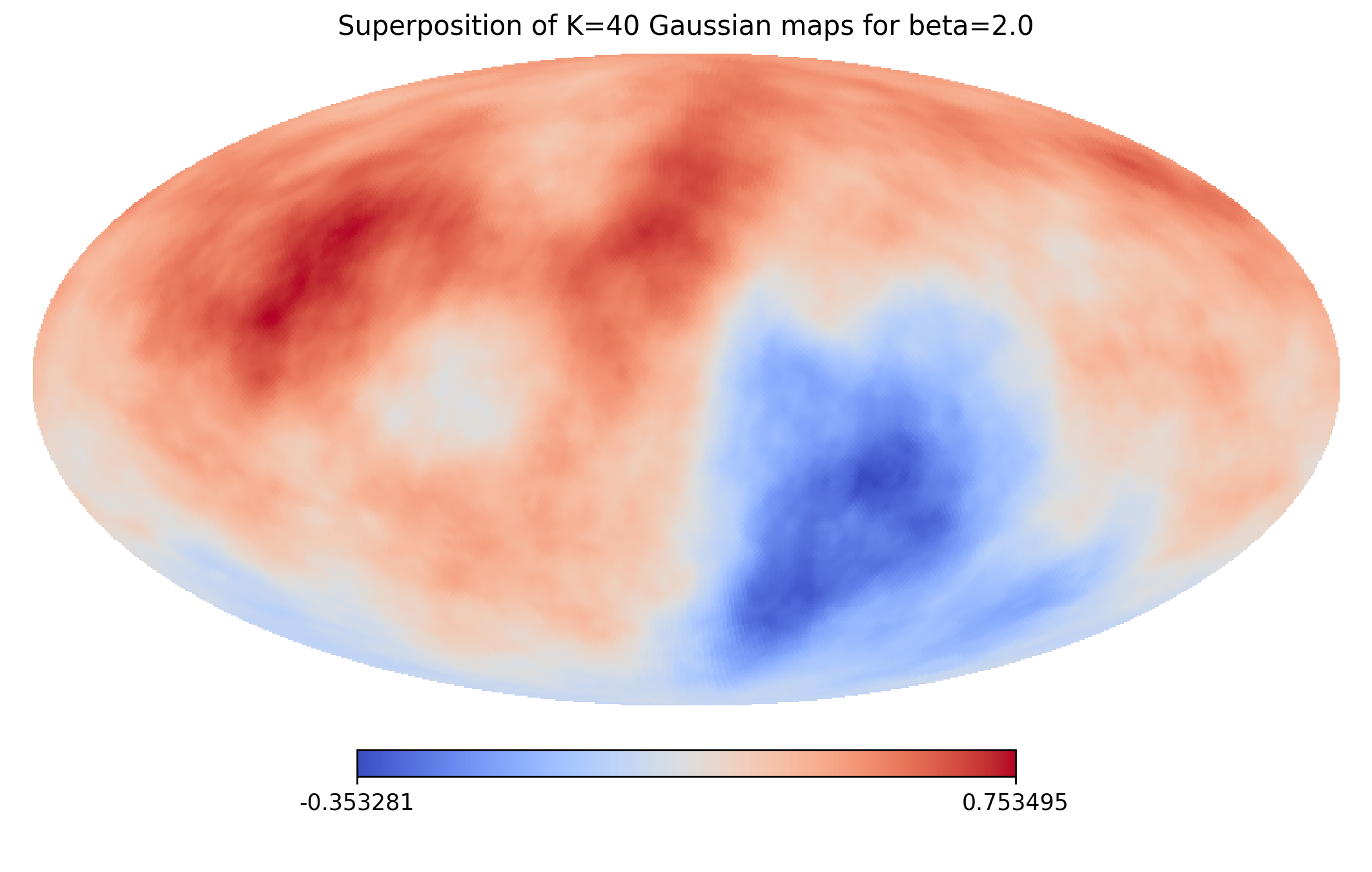}
    \end{subfigure}
    \hfill
    \begin{subfigure}{0.45\textwidth}
        \includegraphics[width=\linewidth]{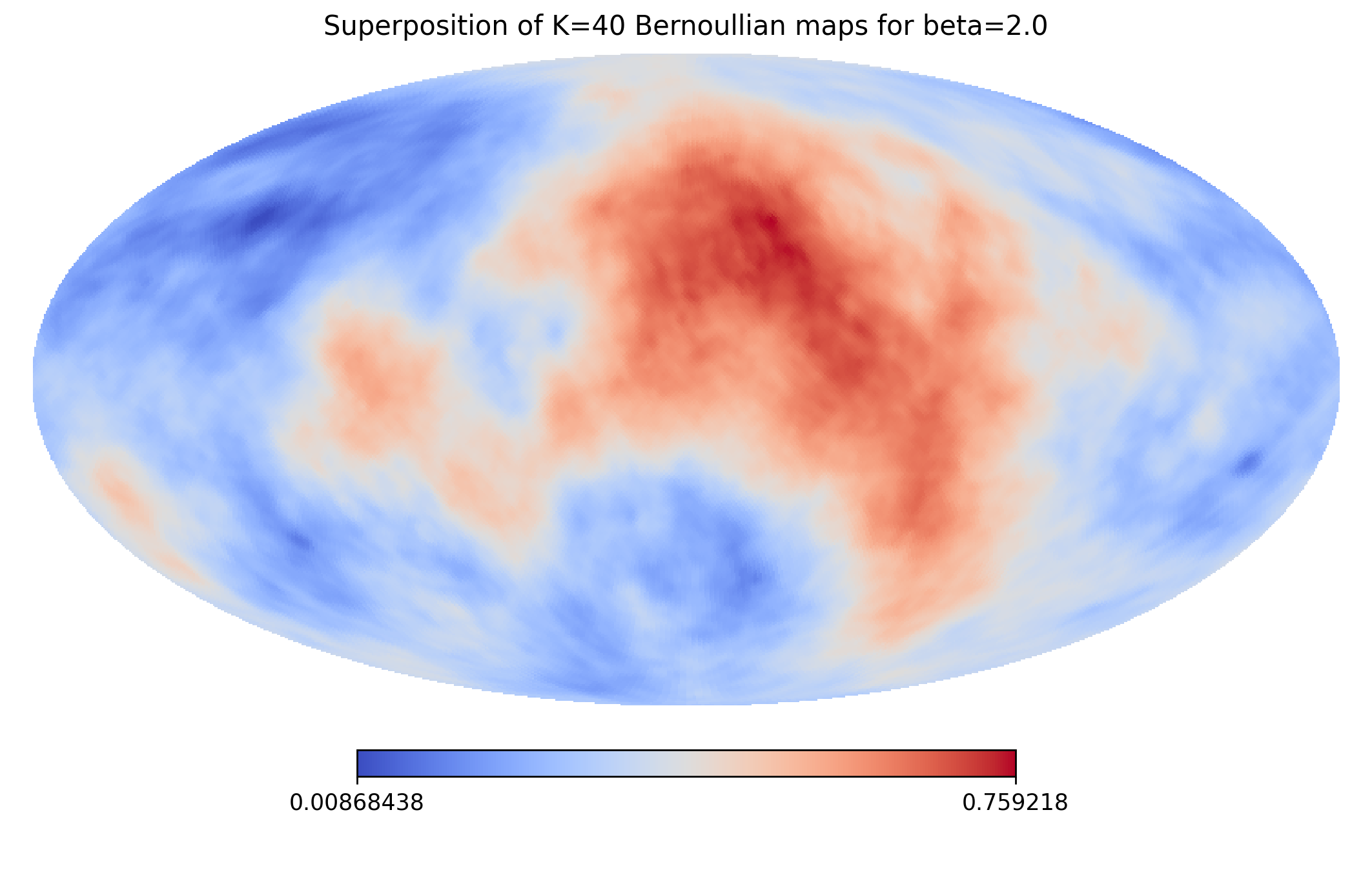}
    \end{subfigure}

     \begin{subfigure}{0.7\textwidth}
       \includegraphics[width=\linewidth]{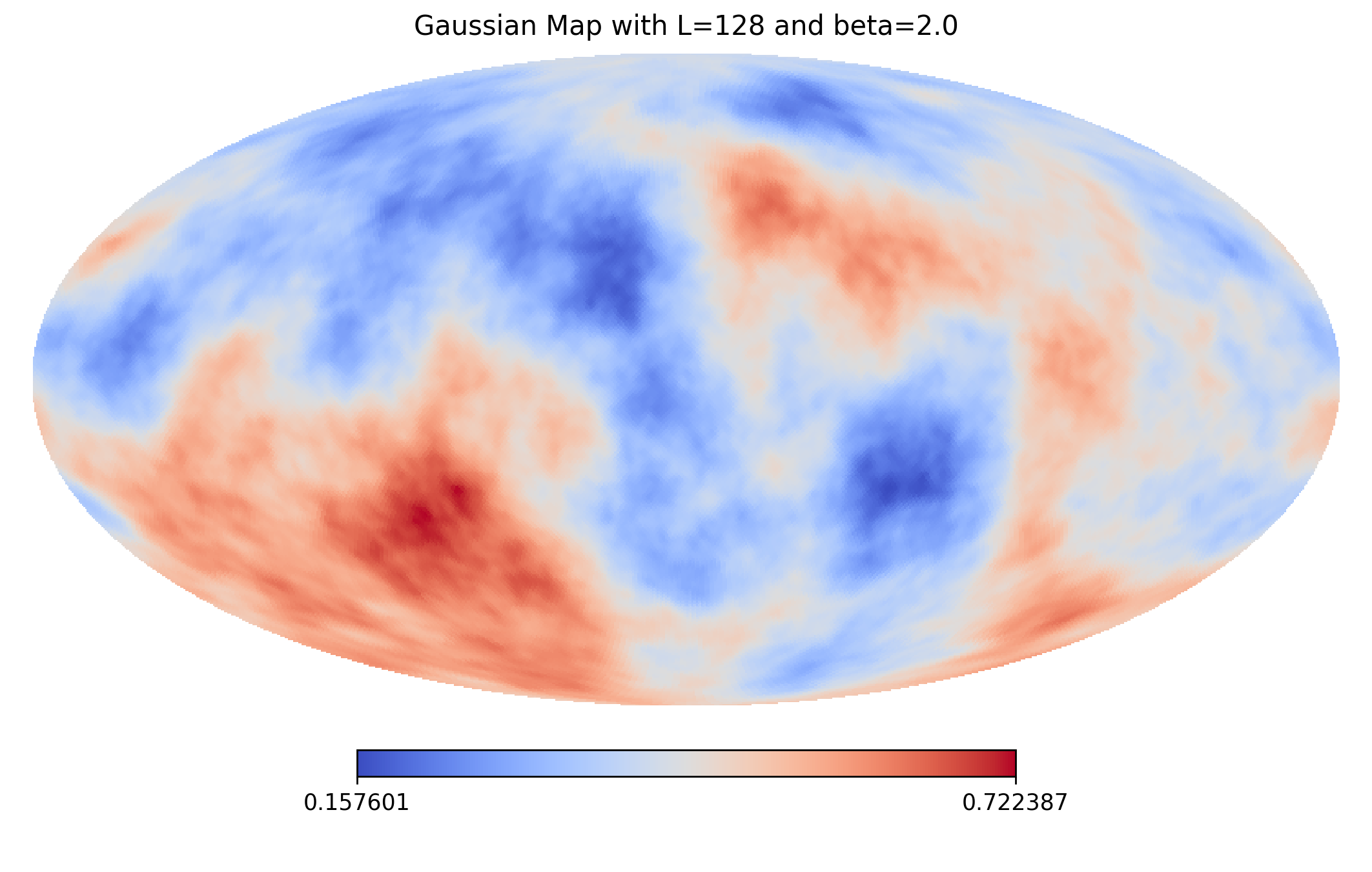}
    \end{subfigure}
    \hfill

      \caption{Simulation of a Whittle-Matérn field with $\beta=2$ (bottom image) with resolution $L_{\max}=128$ and $n_{\mathrm{side}}=64$. In the first three rows, we have plotted the sparse random fields written as superposition of $K=4,\,24,\,40$ random waves. On the left column, the random weights $\{\eta_{\ell k}\}$ are taken to be Gaussian random variables with variance $4\pi C_\ell/K$. On the right column, the random weights $\{\eta_{\ell k}\}$ are centered Bernoulli random variables normalized with $4\pi C_\ell/K$.}
    \label{fig:200}
\end{figure}

\begin{figure}[h!]
    \centering
    \begin{subfigure}{0.45\textwidth}
        \includegraphics[width=\linewidth]{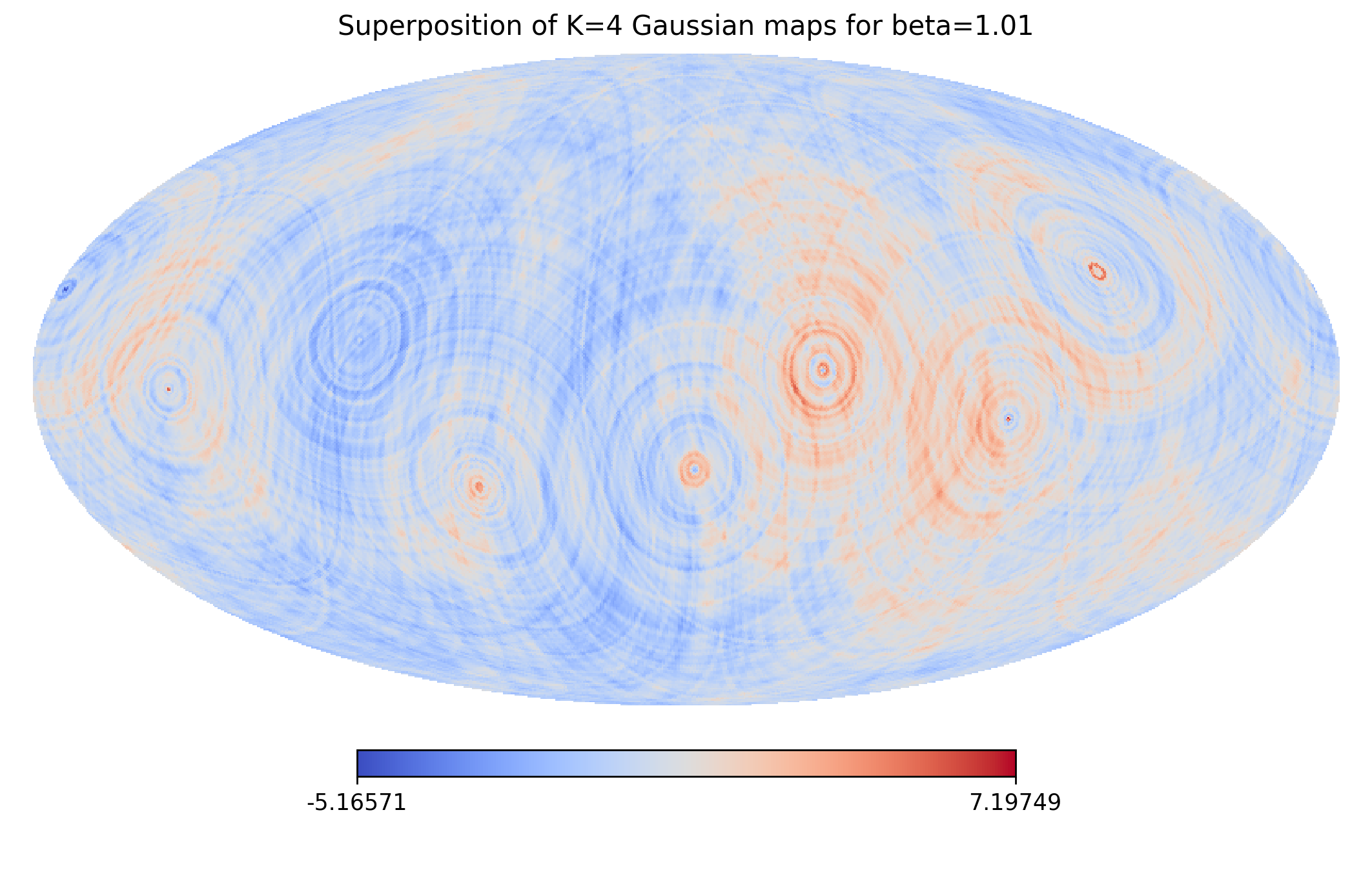}
    \end{subfigure}
        \hfill
    \begin{subfigure}{0.45\textwidth}
        \includegraphics[width=\linewidth]{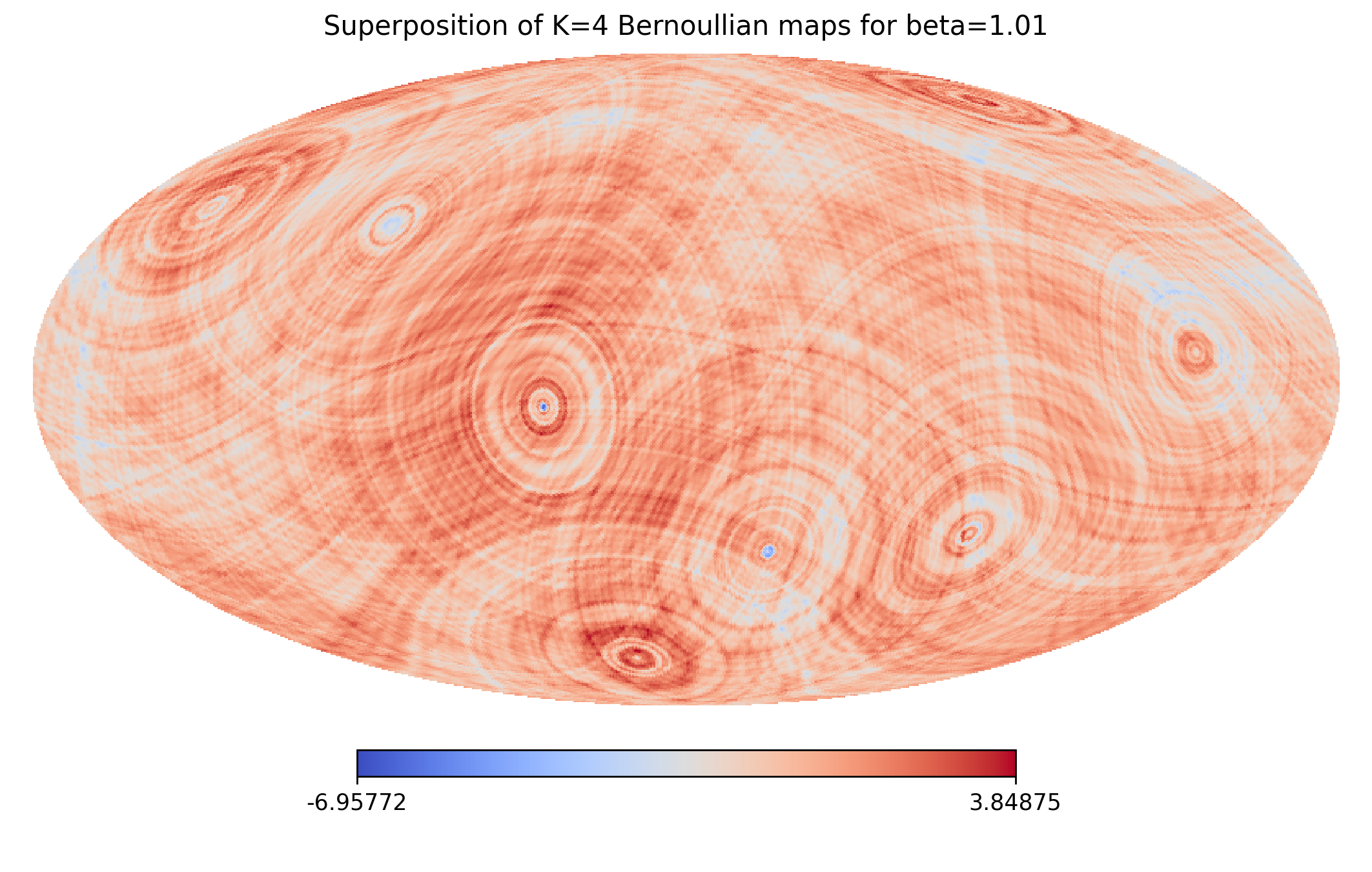}
    \end{subfigure}
        
     \begin{subfigure}{0.45\textwidth}
        \includegraphics[width=\linewidth]{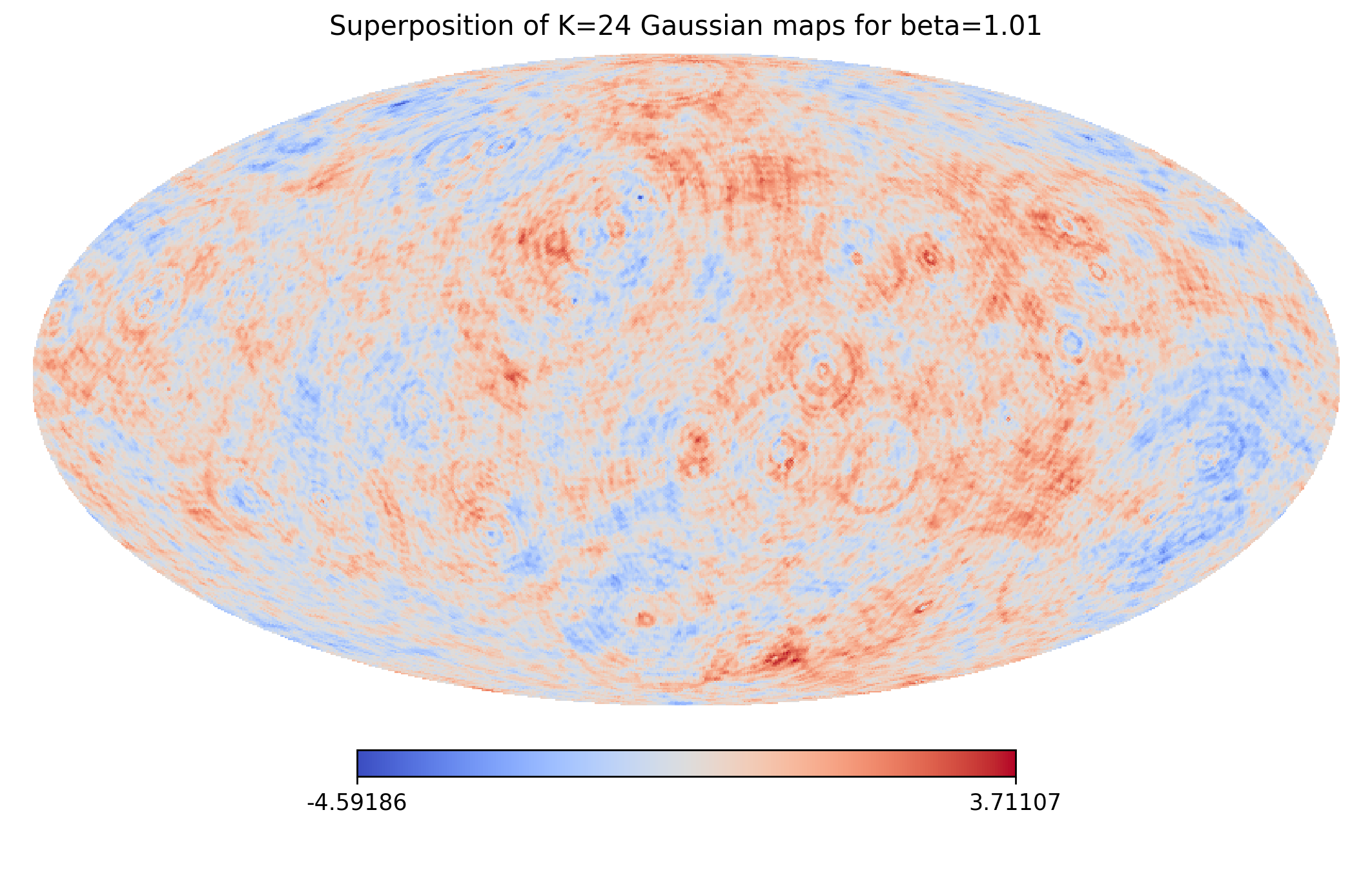}
    \end{subfigure}
    \hfill
    \begin{subfigure}{0.45\textwidth}
        \includegraphics[width=\linewidth]{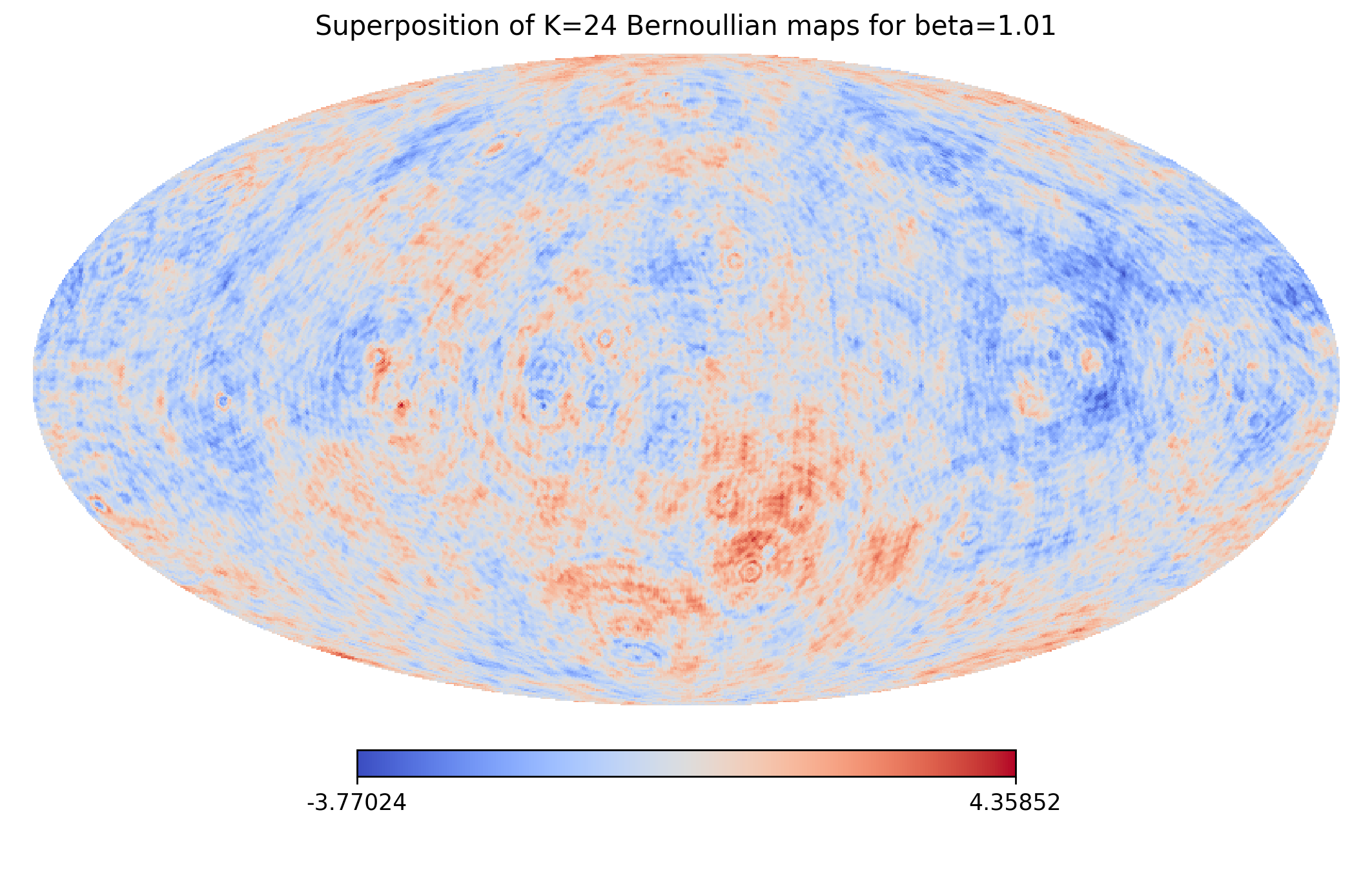}
    \end{subfigure}

     \begin{subfigure}{0.45\textwidth}
       \includegraphics[width=\linewidth]{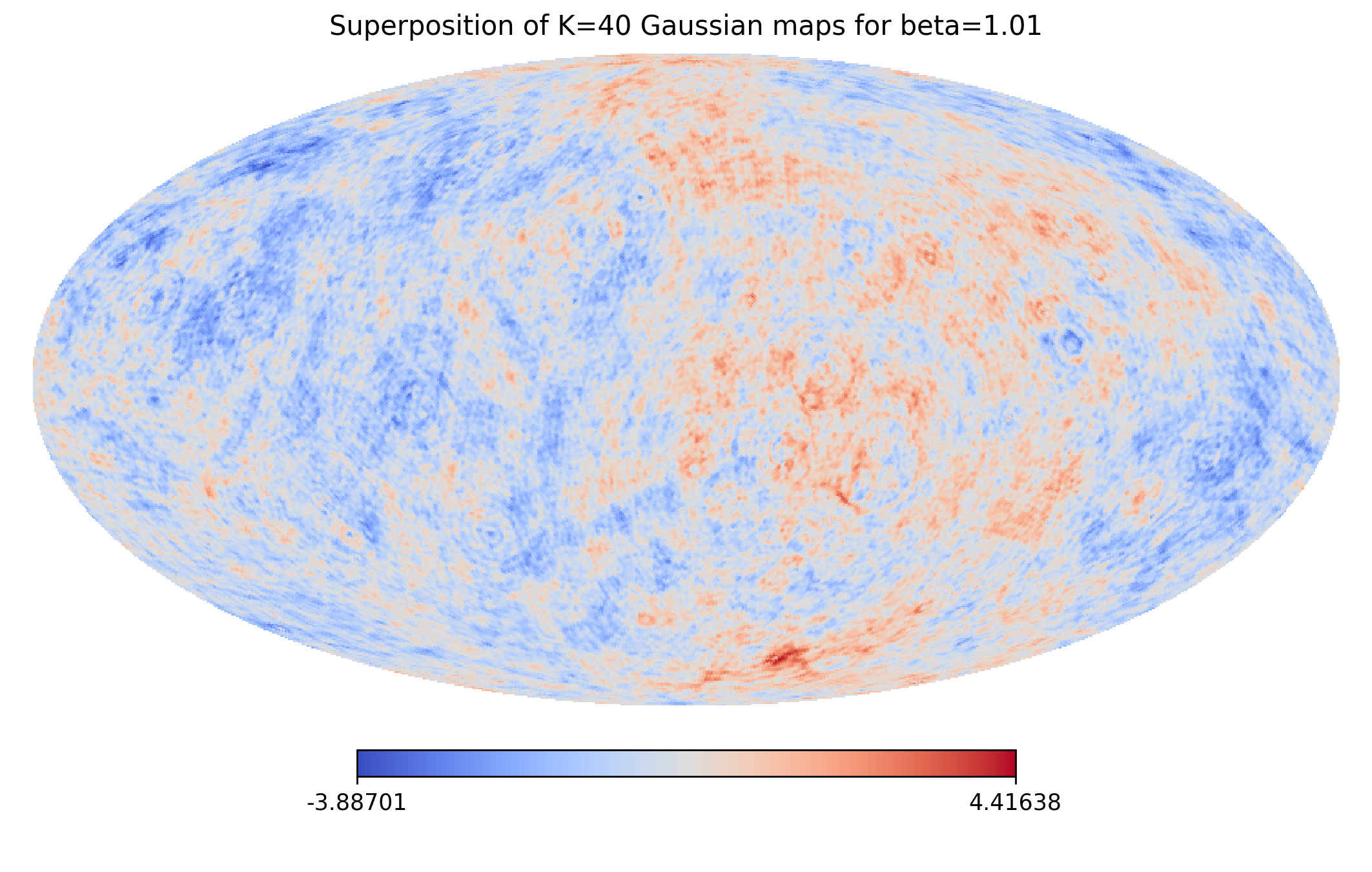}
    \end{subfigure}
    \hfill
    \begin{subfigure}{0.45\textwidth}
        \includegraphics[width=\linewidth]{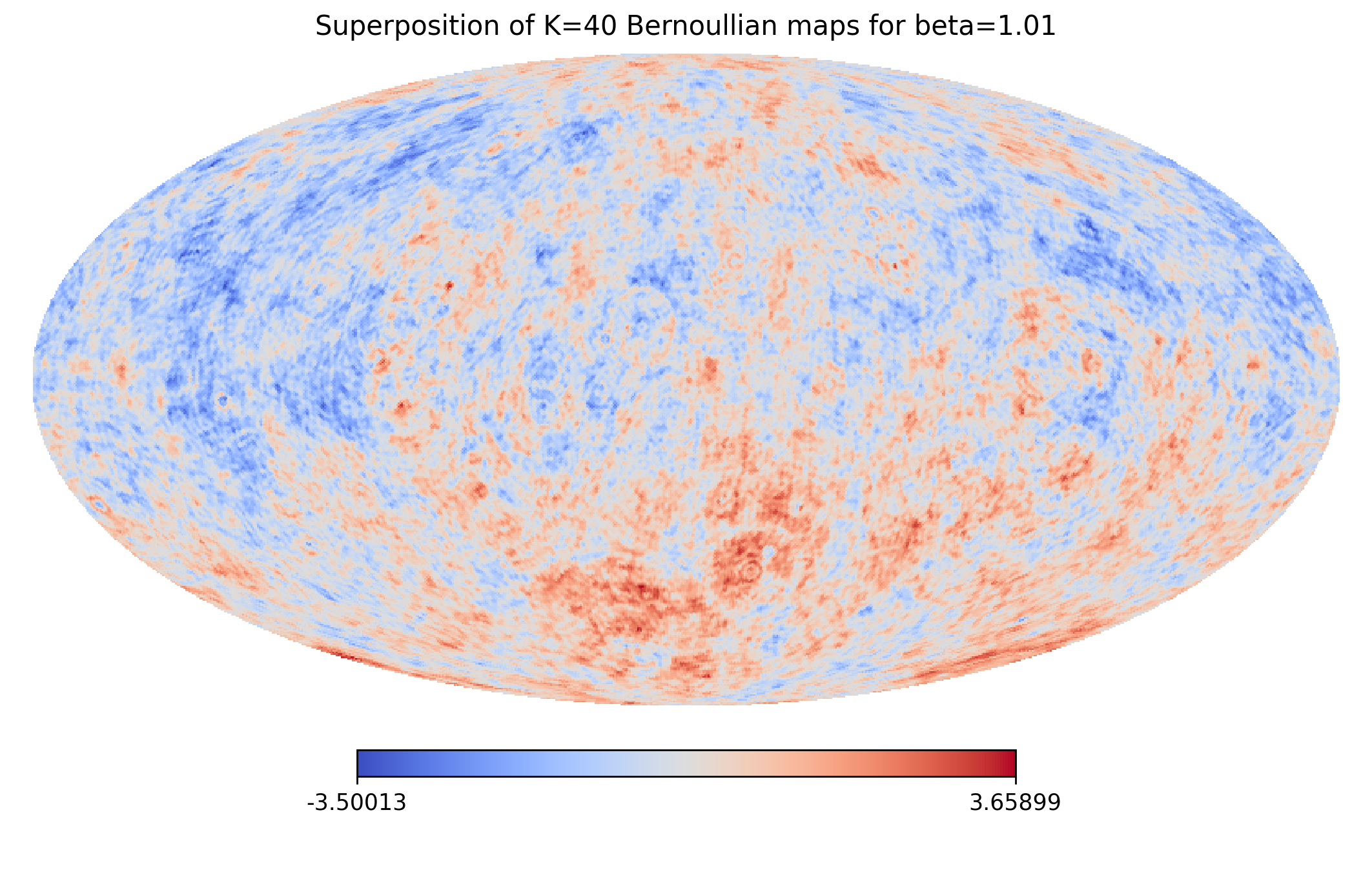}
    \end{subfigure}

     \begin{subfigure}{0.7\textwidth}
       \includegraphics[width=\linewidth]{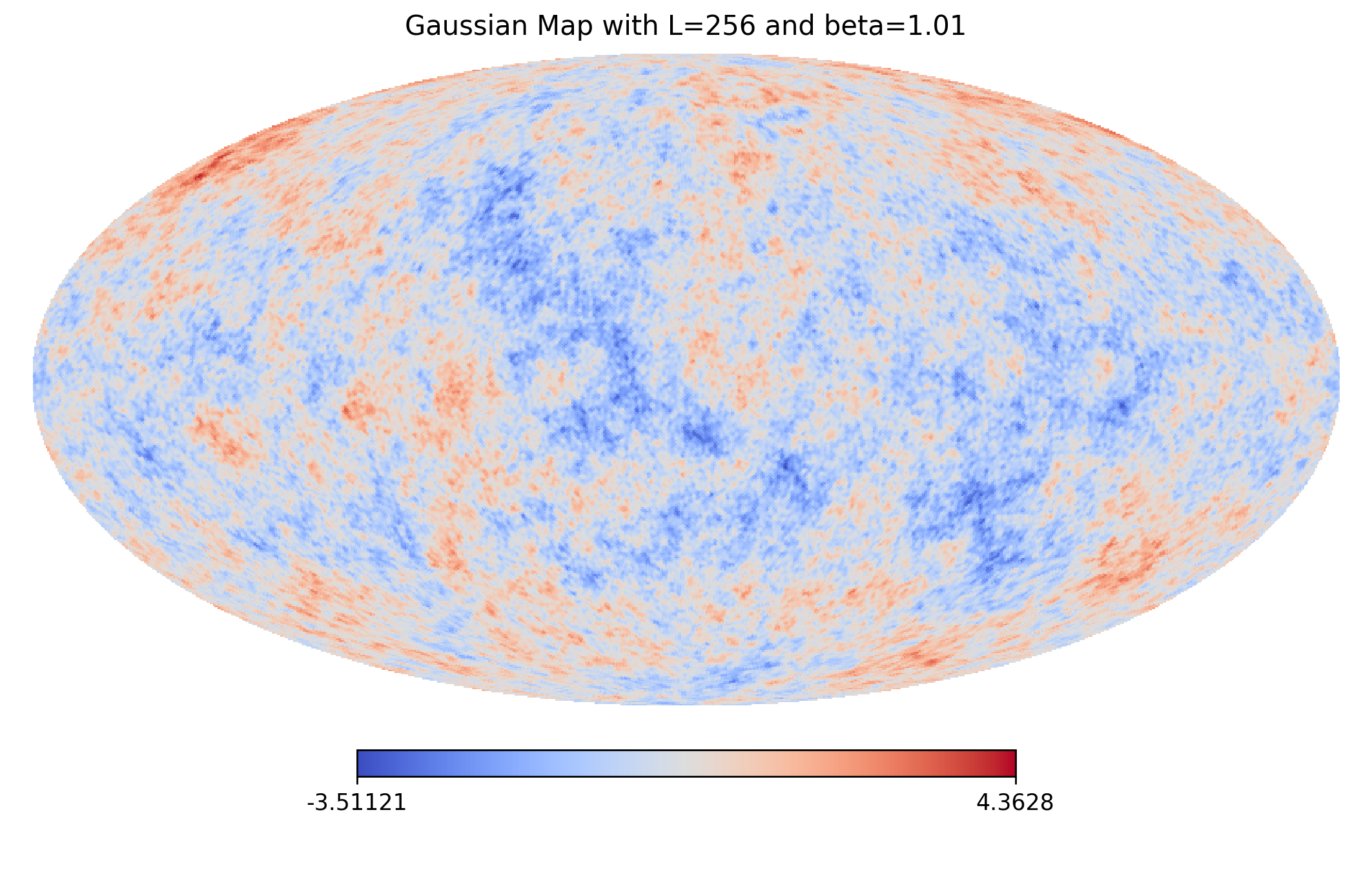}
    \end{subfigure}
    \hfill
    \caption{Simulation of a Whittle-Matérn field with $\beta=1.01$ (bottom image) with resolution $L_{\max}=256$ and $n_{\mathrm{side}}=128$. In the first three rows, we have plotted the sparse random fields written as superposition of $K=4,\,24,\,40$ random waves. On the left column, the random weights $\{\eta_{\ell k}\}$ are taken to be Gaussian random variables with variance $4\pi C_\ell/K$. On the right column, the random weights $\{\eta_{\ell k}\}$ are centered Bernoulli random variables normalized with $4\pi C_\ell/K$.}
    \label{fig:256_101}
\end{figure}

\begin{figure}[h!]
    \centering
    \begin{subfigure}{0.45\textwidth}
        \includegraphics[width=\linewidth]{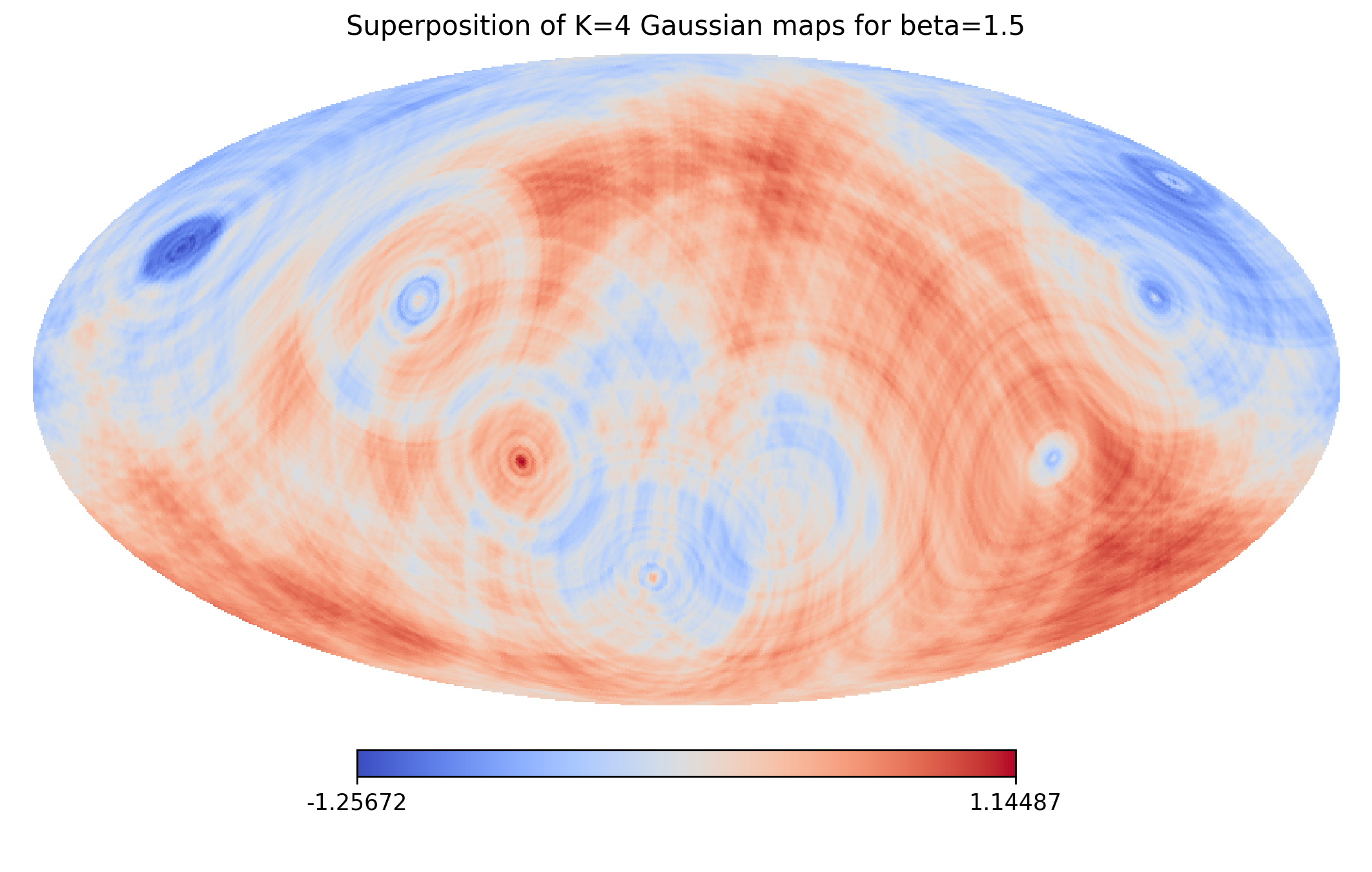}
    \end{subfigure}
        \hfill
    \begin{subfigure}{0.45\textwidth}
        \includegraphics[width=\linewidth]{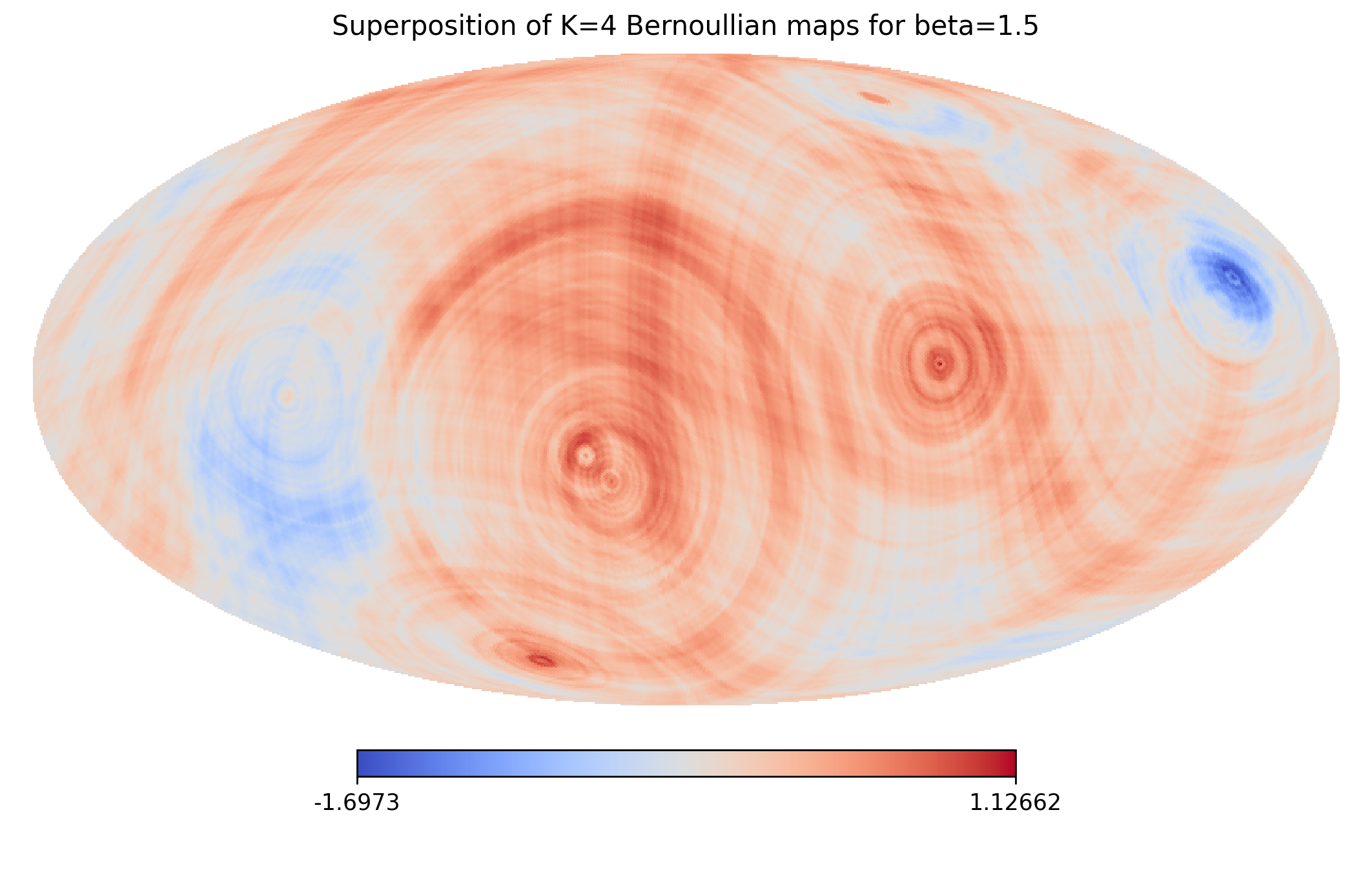}
    \end{subfigure}
        
     \begin{subfigure}{0.45\textwidth}
        \includegraphics[width=\linewidth]{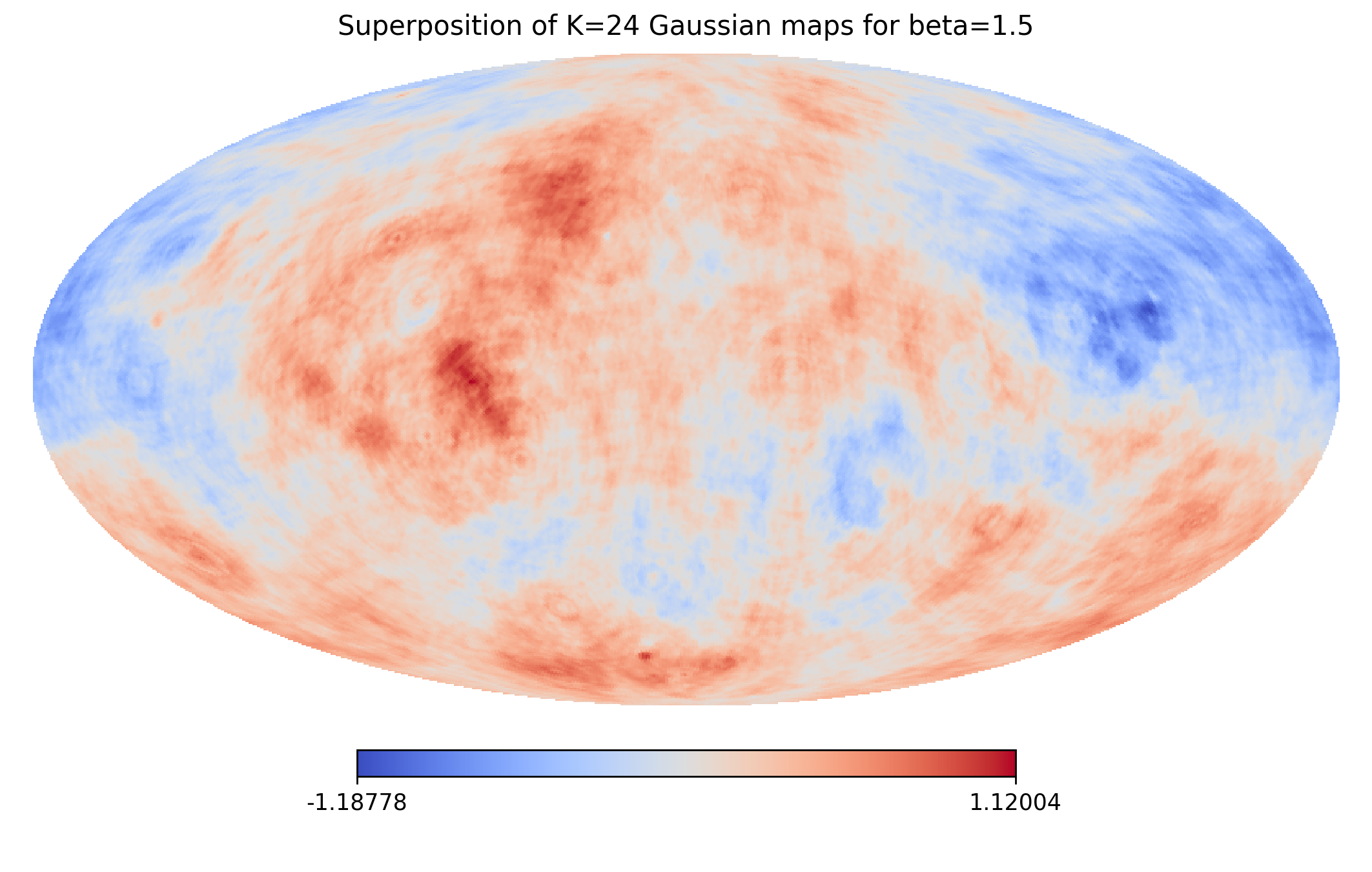}
    \end{subfigure}
    \hfill
    \begin{subfigure}{0.45\textwidth}
        \includegraphics[width=\linewidth]{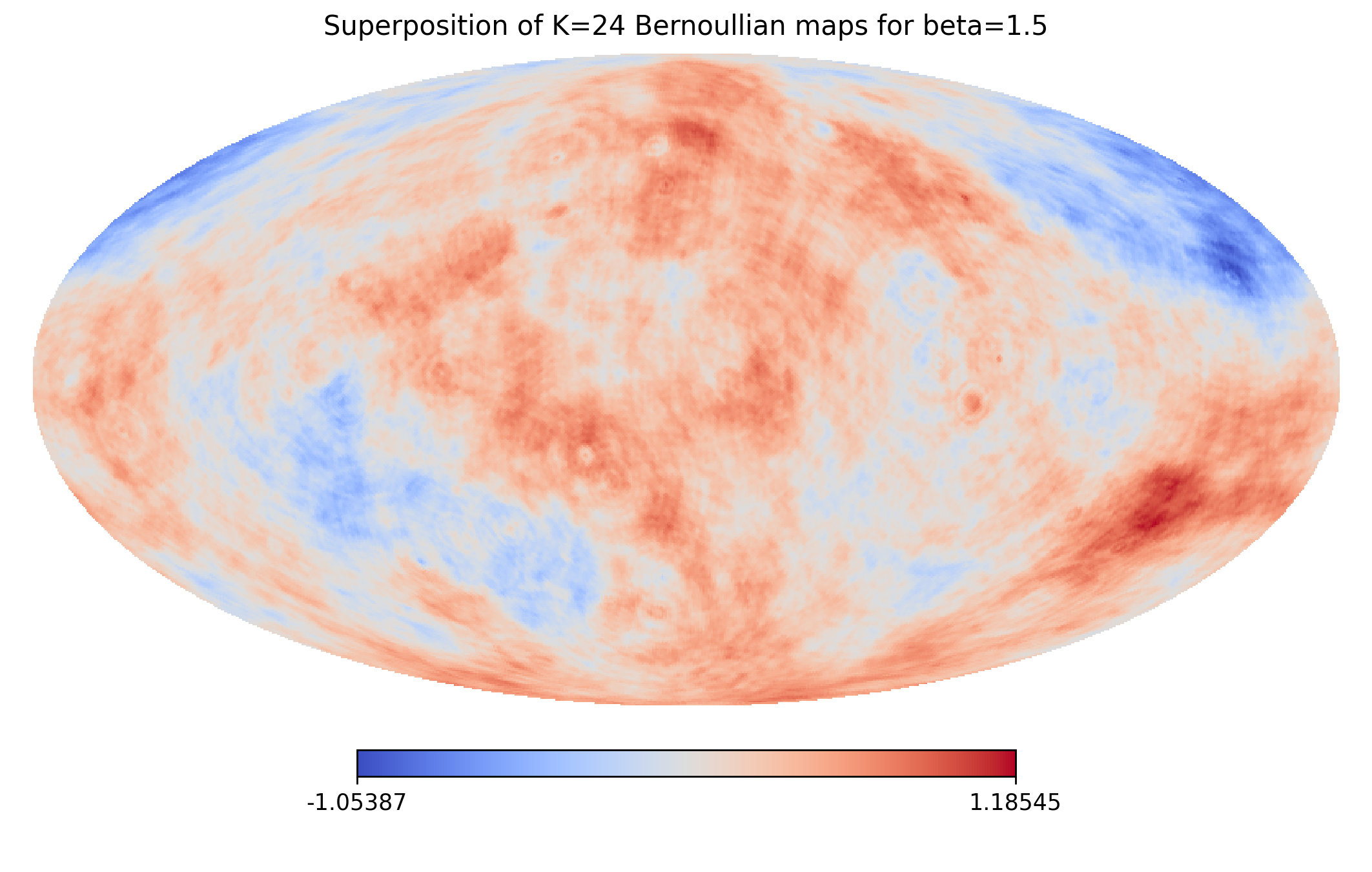}
    \end{subfigure}

     \begin{subfigure}{0.45\textwidth}
       \includegraphics[width=\linewidth]{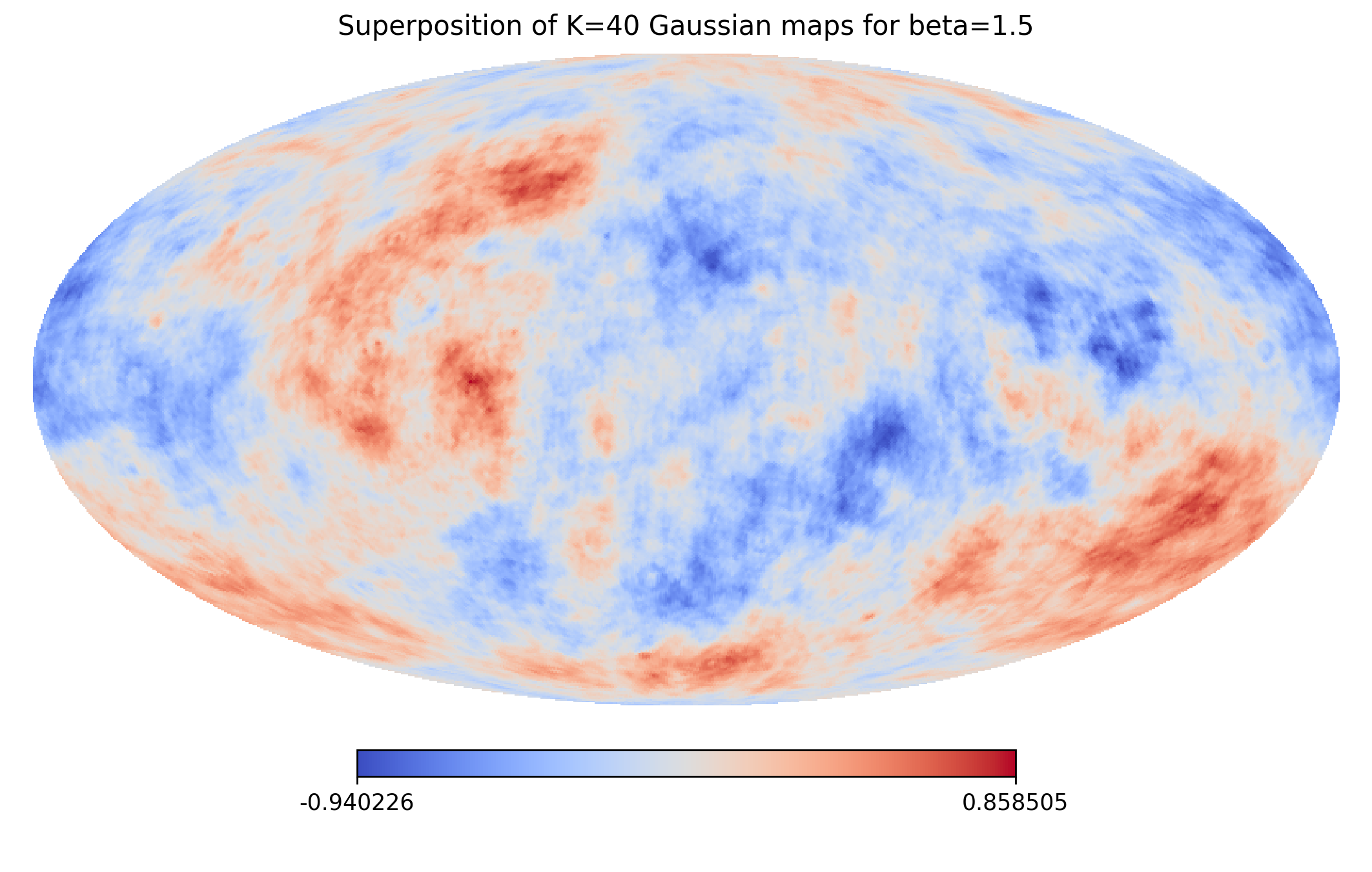}
    \end{subfigure}
    \hfill
    \begin{subfigure}{0.45\textwidth}
        \includegraphics[width=\linewidth]{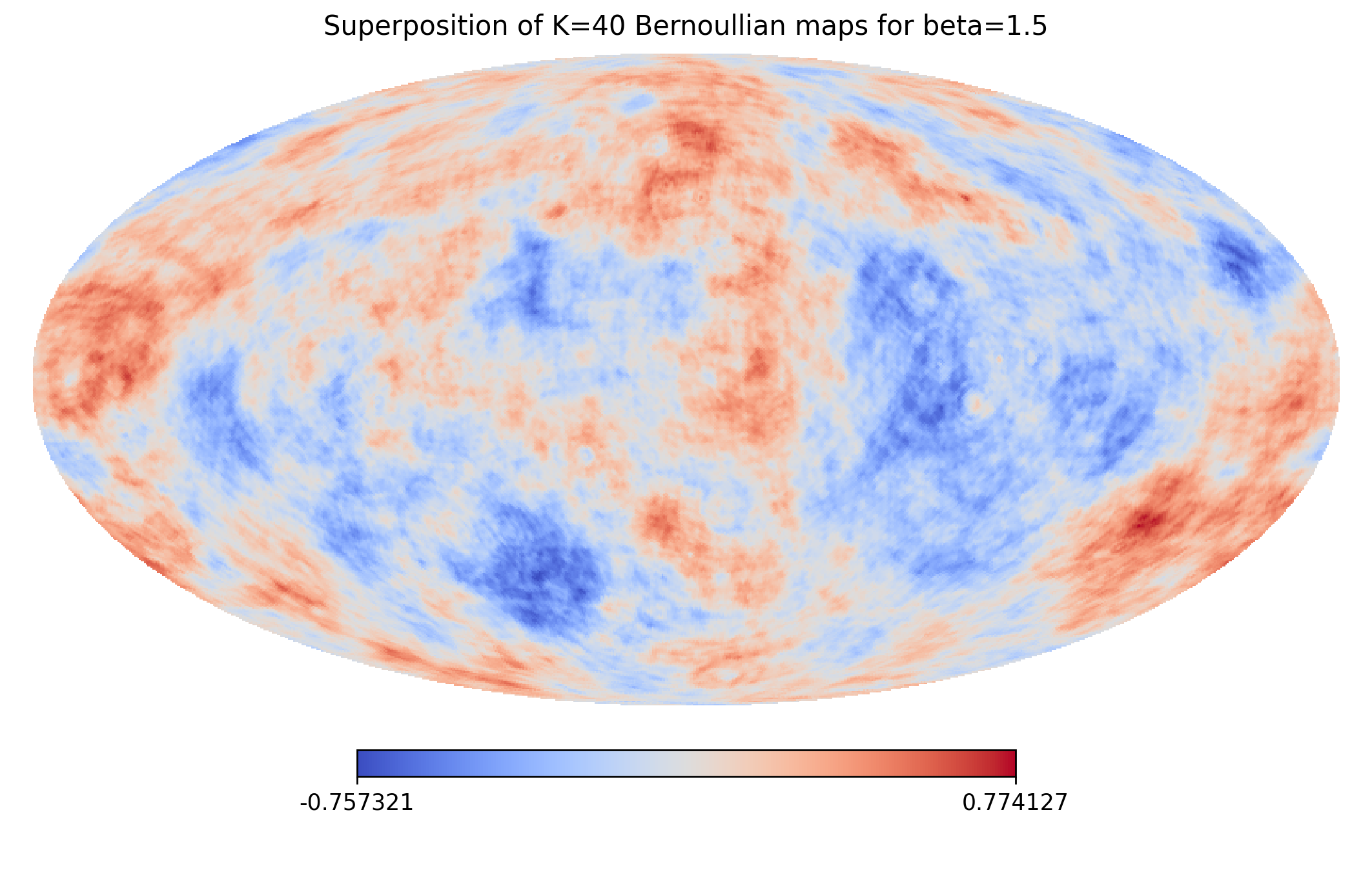}
    \end{subfigure}

     \begin{subfigure}{0.7\textwidth}
       \includegraphics[width=\linewidth]{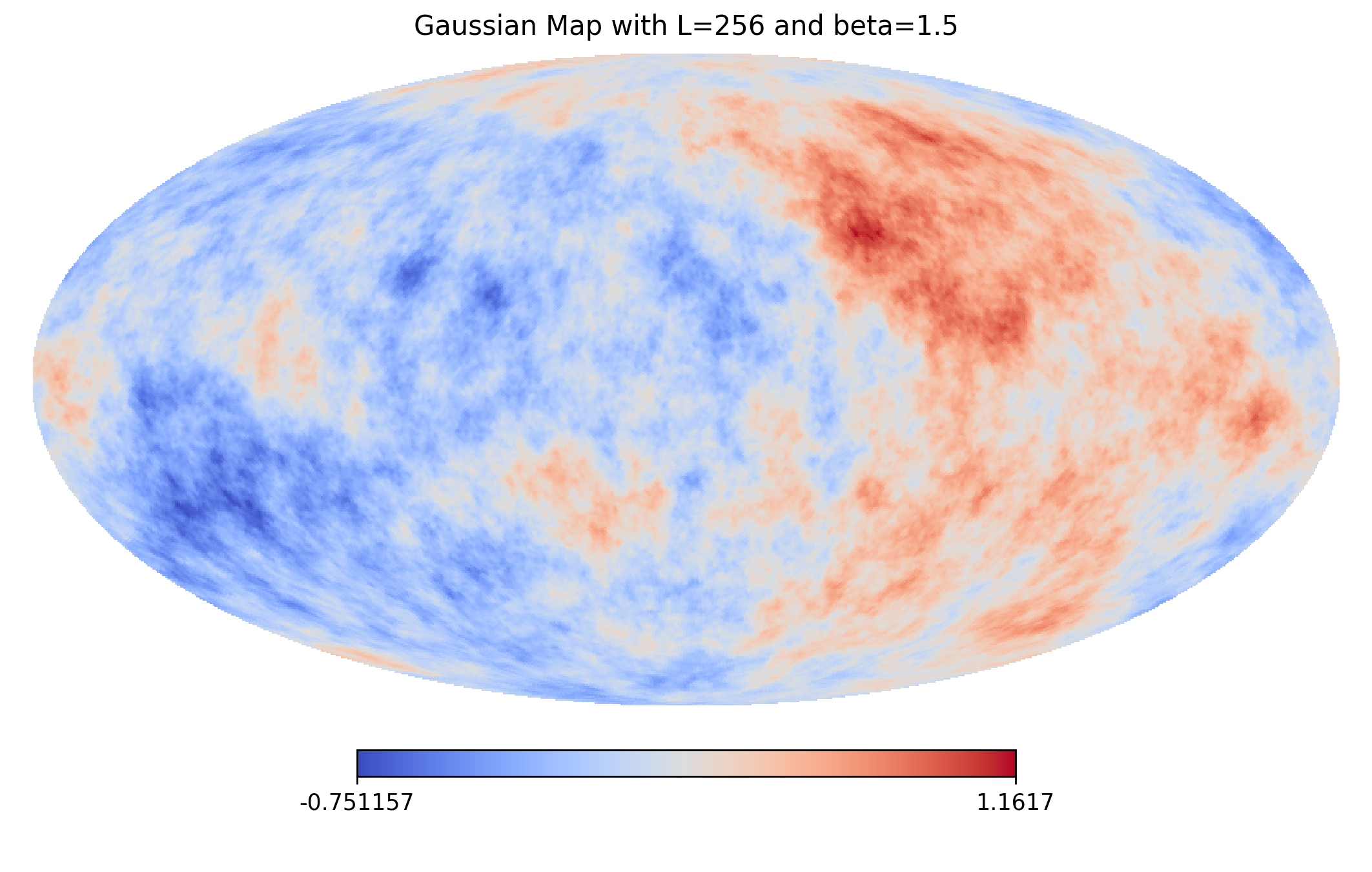}
    \end{subfigure}
    \hfill

     \caption{Simulation of a Whittle-Matérn field with $\beta=1.5$ (bottom image) with resolution $L_{\max}=256$ and $n_{\mathrm{side}}=128$. In the first three rows, we have plotted the sparse random fields written as superposition of $K=4,\,24,\,40$ random waves. On the left column, the random weights $\{\eta_{\ell k}\}$ are taken to be Gaussian random variables with variance $4\pi C_\ell/K$. On the right column, the random weights $\{\eta_{\ell k}\}$ are centered Bernoulli random variables normalized with $4\pi C_\ell/K$.}
    \label{fig:256_150}
\end{figure}

\newpage

\appendix

\section{On the Wick Product}\label{App:Wick}

In this appendix, we recall the definition and main properties of Wick's product for Gaussian random variables. Let $(\Omega,\mathcal{F},\mathbb{P})$ be a probability space and let
$Z=(Z_1,\dots,Z_d)$ be a centered Gaussian vector with covariance matrix
$\Sigma$. Wick products (or normal ordered products) provide a
renormalized notion of multiplication for real-valued Gaussian random variables,
which plays a fundamental role in Gaussian analysis and Wiener chaos
theory.

\medskip

\noindent\textbf{Main properties.}
Wick products satisfy the following fundamental properties:
\begin{itemize}
\item[(i)] \emph{Centering:} $\mathbb{E}[\colon Z_{i_1}\cdots Z_{i_k}\colon]=0$
for all $k\geq1$.
\item[(ii)] \emph{Orthogonality:} Wick monomials of different total
degrees are orthogonal in $\mrmL^2(\Omega)$.
\item[(iii)] \emph{Product formula:} for Gaussian variables $X,Y$,
\[
XY = \,\colon XY\colon \,+ \mathbb{E}[XY],
\]
and more generally any product of Gaussian variables can be expanded as
a finite sum of Wick products corresponding to all possible contractions.
\item[(iv)] \emph{Stability under limits:} Wick products extend by
continuity to limits of polynomial functionals in $\mrmL^2(\Omega)$.
\item[(v)] \emph{Connection to Hermite polynomials: } If $Z=(Z_1,\dots,Z_d)$ is a vector of independent standard Gaussian random
variables, then for any multi-index $\underline{\alpha}=(\alpha_1,\dots,\alpha_d)\in\mathbb{N}^d$ the Wick monomial $\colon Z^{\underline{\alpha}}\colon \;=\; \colon Z_1^{\alpha_1}\cdots Z_d^{\alpha_d}\colon$ 
coincides with the product of univariate Hermite
polynomials, namely
\begin{equation*}
\colon Z^{\underline{\alpha}}\colon = \prod_{j=1}^d H_{\alpha_j}(Z_j),
\end{equation*}
where $H_k$ denotes the $k^{th}$ Hermite polynomial.
\end{itemize}

\medskip

\noindent
Wick products therefore provide a canonical orthogonalization of
polynomial Gaussian functionals and are naturally identified with
elements of Wiener chaos, a viewpoint that is central in limit theorems
and diagram formulae. 

For instance, if
$Z$ is a centered Gaussian random variable with unit-variance, then for any integers $p,q\geq0$, the product of Wick
powers admits the expansion
\[
:Z^p:\,:Z^q:
=\sum_{r=0}^{p\wedge q}
r!\binom{p}{r}\binom{q}{r}
:Z^{p+q-2r}:,
\]
which follows from the multiplication formula for Hermite polynomials (see \cite{NourdinPeccati}).
In particular, Wick products are not multiplicative when sharing common
Gaussian factors, and all possible contractions explicitly contribute.

More generally, for jointly Gaussian random variables
$Z_1,\dots,Z_d$, the product of Wick monomials decomposes as a finite sum
of Wick products obtained by pairing common variables according to their
covariance structure, a representation that underlies the classical
diagram formula.
\medskip
\begin{example}
For instance, if $Z_1,Z_2,Z_3$ are jointly standard Gaussian, one obtains
\[
:Z_1 Z_2:\,:Z_2 Z_3:\,
= Z_1\,:Z_2^2:\,Z_3=H_1(Z_1)\,H_2(Z_2)\,H_1(Z_3)=Z_1(Z_2^2-1)Z_3=Z_1 Z_2^2 Z_3
+ Z_1 Z_3\,,
\]
which illustrates explicitly how common variables generate lower-order
Wick products (hence Hermite polynomials) through Gaussian contractions. 

Similarly, for $Z_1,Z_2$ jointly standard Gaussian one gets
\[
\colon Z_1Z_2\colon\, \colon Z_1Z_2^2\colon=\colon Z_1^2 Z_2^3\colon
= H_2(Z_1)\,H_3(Z_2)
= (Z_1^2-1)(Z_2^3-3Z_2)\,.
\]
\end{example}


\bibliographystyle{alpha}
\bibliography{main}

\end{document}